\documentclass[11pt]{amsart}
\usepackage{geometry}                
\usepackage{fancyhdr}

\usepackage{amsmath}
\usepackage{mathrsfs}
\usepackage{graphicx}
\usepackage{amssymb}
\usepackage{epstopdf}
\usepackage[all]{xypic}
\usepackage{palatino}
\usepackage{color}
\usepackage{calligra}

\usepackage{hyperref}

\usepackage[arrow,curve,matrix]{xy}

\newcommand{\reg}{{\rm reg}}

\renewcommand{\det}{{\rm det}}
\renewcommand{\Im}{{\rm Im}}

\renewcommand{\dim}{{\rm dim}}
\renewcommand{\deg}{{\rm deg}}

\newcommand{\rk}{{\rm rk}}
\newcommand{\coker}{{\rm coker}}
\renewcommand{\det}{{\rm det}}

\newcommand{\dblq}{{/\!/}}

\theoremstyle{plain}

\newtheorem{thm}{Theorem}[section]

\newtheorem{prop}[thm]{Proposition}
\newtheorem{lem}[thm]{Lemma}

\theoremstyle{definition}
\newtheorem{defn}[thm]{Definition}

\newtheorem{rmk}[thm]{Remark}

\def\kk{{\mathbf k}}

\def\FF{{\mathbf F}}
\def\GG{{\mathbf G}}

\def\PP{{\textbf P}}
\def\OO{\mathcal{O}}

\def\cB{\mathcal{B}}
\def\cA{\mathcal{A}}
\def\F{\mathcal{F}}

\def\cM{\mathcal{M}}

\def\H{\mathcal{H}}

\def\mm{\overline{\mathcal{M}}}

\def\mmp{\overline{\mathcal{M}}_g^{\mathrm{ps}}}

\newcommand{\bb}[1]{\mathbb{#1}}
\newcommand{\mc}[1]{\mathcal{#1}}
\newcommand{\mf}[1]{\mathfrak{#1}}
\newcommand{\ol}[1]{\overline{#1}}
\newcommand{\tl}[1]{\widetilde{#1}}
\newcommand{\ul}[1]{\underline{#1}}
\newcommand{\defi}[1]{{\em #1}}
\newcommand{\op}[1]{\operatorname{#1}}

\DeclareMathOperator{\ShExt}{\mathscr{E}\text{\kern -3pt {\calligra\large xt}}\,}

\renewcommand{\a}{\alpha}
\renewcommand{\b}{\beta}
\newcommand{\bw}{\bigwedge}
\newcommand{\chr}{\operatorname{char}}
\renewcommand{\coker}{\operatorname{Coker}}
\newcommand{\eps}{\epsilon}

\renewcommand{\ll}{\lambda}
\renewcommand{\ker}{\operatorname{Ker}}
\newcommand{\LL}{\Lambda}
\newcommand{\onto}{\twoheadrightarrow}
\newcommand{\oo}{\otimes}
\newcommand{\pd}{\partial}
\newcommand{\s}{\sigma}
\newcommand{\T}{\mathcal{T}}

\newcommand{\Lop}{\mathfrak{L}} 
\newcommand{\Rop}{\mathfrak{R}} 
\renewcommand{\sl}{\mathfrak{sl}}

\def\lra{\longrightarrow}
\def\llra{\longleftrightarrow}
\newcommand*{\lhra}{\ensuremath{\lhook\joinrel\relbar\joinrel\rightarrow}}

\newcommand{\id}{\operatorname{id}}
\newcommand{\D}{\operatorname{D}}
\newcommand{\Gr}{\operatorname{Gr}}

\newcommand{\Spec}{\operatorname{Spec}}
\newcommand{\relSpec}{\underline{\operatorname{Spec}}}
\newcommand{\Sym}{\operatorname{Sym}}
\newcommand{\Tor}{\operatorname{Tor}}

\theoremstyle{plain}
\newtheorem{corollary}[thm]{Corollary}
\newtheorem{lemma}[thm]{Lemma}
\newtheorem{proposition}[thm]{Proposition}

\theoremstyle{definition}
\newtheorem{remark}[thm]{Remark}
\newtheorem*{ack}{Acknowledgment}

\newtheorem*{thm-main*}{Main Theorem}

\begin{document}
\title{Koszul modules and Green's conjecture}


\author[M. Aprodu]{Marian Aprodu}
\address{Marian Aprodu: Simion Stoilow Institute of Mathematics
\hfill \newline\texttt{}
 \indent P.O. Box 1-764,
RO-014700 Bucharest, Romania \&
\hfill \newline\texttt{}
 \indent
 Faculty of Mathematics and Computer Science, University of Bucharest}
\email{{\tt marian.aprodu@imar.ro} \& {\tt marian.aprodu@fmi.unibuc.ro}}

\author[G. Farkas]{Gavril Farkas}
\address{Gavril Farkas: Institut f\"ur Mathematik, Humboldt-Universit\"at zu Berlin \hfill \newline\texttt{}
\indent Unter den Linden 6,
10099 Berlin, Germany}
\email{{\tt farkas@math.hu-berlin.de}}

\author[\c S. Papadima]{\c Stefan Papadima\textsuperscript{\textdagger} }
\address{\c Stefan Papadima: Simion Stoilow Institute of Mathematics \hfill \newline\texttt{}
\indent P.O. Box 1-764,
RO-014700 Bucharest, Romania}
\thanks{\c Stefan Papadima passed away on  January 10, 2018. }

\author[C. Raicu]{Claudiu Raicu}
\address{Claudiu Raicu: Department of Mathematics,
University of Notre Dame \hfill \newline\texttt{}
\indent 255 Hurley Notre Dame, IN 46556, USA, and \hfill\newline\texttt{}
\indent Simion Stoilow Institute of Mathematics, \hfill\newline\texttt{}
\indent  P.O. Box 1-764, RO-014700 Bucharest, Romania}
\email{{\tt craicu@nd.edu}}

\author[J. Weyman]{Jerzy Weyman}
\address{Jerzy Weyman: Department of Mathematics
University of Connecticutt \hfill \newline\texttt{}
\indent Storrs, CT,
06269-3009, USA}
\email{{\tt  jerzy.weyman@uconn.edu}}

\begin{abstract}
We prove a strong vanishing result for finite length Koszul modules, and use it to derive Green's conjecture for every $g$-cuspidal rational curve over an algebraically closed field~$\kk$, with $\chr(\kk)=0$ or $\chr(\kk)\geq \frac{g+2}{2}$. As a consequence, we deduce that  the general canonical curve of genus~$g$ satisfies Green's conjecture in this range. Our results are new in positive characteristic, whereas in characteristic zero they provide a different proof for theorems first obtained in two landmark papers by Voisin. Our strategy  involves establishing two key results of independent interest: (1) we describe an explicit, characteristic-independent version of Hermite reciprocity for $\sl_2$-representations; (2) we completely characterize, in arbitrary characteristics, the (non-)vanishing behavior of the syzygies of the tangential variety to a rational normal curve.
\end{abstract}

\maketitle

\section{Introduction}

Formulated in 1984, Green's Conjecture \cite[Conjecture~5.1]{G84} predicts that one can recognize in a precise way the intrinsic complexity of a smooth algebraic curve from the syzygies of its canonical embedding. If $C\hookrightarrow \mathbb P^{g-1}$ is a non-hyperelliptic canonically  embedded curve of genus $g$,  we denote by $K_{i,j}(C,\omega_C)$ the Koszul cohomology group of $i$-th syzygies of weight $j$. Green's Conjecture predicts the equivalence
\[K_{i,1}(C,\omega_C)=0 \Longleftrightarrow i\geq g-\mbox{Cliff}(C)-1,\]
where $\mbox{Cliff}(C)$ is the Clifford index of $C$. Equivalently, $K_{i,2}(C,\omega_C)=0$ if and only if $i<\mbox{Cliff}(C)$. Although for arbitrary curves the conjecture remains wide open, Green's Conjecture for a \emph{general} curve of every genus has been resolved using geometric methods in two landmark papers by Voisin \cite{V02,V05}. In this case one has  $\mbox{Cliff}(C)=\lfloor \frac{g-1}{2}\rfloor$ and Green's Conjecture reduces to a single vanishing statement
\begin{equation}\label{eq:generic-vanishing}
K_{\lfloor \frac{g}{2}\rfloor,1}(C,\omega_C) = 0.
\end{equation}

More direct approaches have been proposed over the years to solve the Generic Green Conjecture (even prior to Voisin's papers), but none has been so far brought to fruition. One of them is described in Eisenbud's paper \cite[Section~3.I]{E91} and relies on degeneration and proving the vanishing ~(\ref{eq:generic-vanishing}) for a general canonically embedded $g$-cuspidal rational curve. This approach reduces the Generic Green Conjecture to a rather impenetrable looking problem in $\mathfrak{sl}_2$-representation theory. Using a novel perspective on these questions essentially inspired from topology, and which we have already put to work in \cite{koszul-groups} to attack questions in geometric group theory, we complete this approach by reducing the Generic Green Conjecture to a strong vanishing result for finite length Koszul modules. In characteristic zero, this gives a new, relatively short proof of Green's Conjecture for generic curves of any genus. Compared to Voisin's approach, we do not use the geometry of $K3$ surfaces and Hilbert schemes of $0$-dimensional subschemes, but instead we view the Generic Green Conjecture as a question on the geometry of the Grassmannian of lines, which we then solve by using general vanishing results for homogeneous bundles (Bott's Theorem). This approach also works in positive characteristic, as we shall explain.

\vskip 4pt

Let $\T\subseteq \PP^{g}$ denote the tangential variety of the degree $g$ rational normal curve $\Gamma\subseteq \PP^g$, over some algebraically closed field~$\kk$. A general hyperplane section $C$ of $\T$ is a canonically embedded $g$-cuspidal rational curve, the cusps corresponding to the points of intersection of $C$ with $\Gamma$. Using a form of the Lefschetz hyperplane principle \cite[Lemma 2.19]{AN10}, one has the isomorphism $$K_{i,j}\bigl(\T,\OO_{\T}(1)\bigr)\cong K_{i,j}(C,\omega_C)$$ for all $i,j$. In characteristic zero, it is a consequence of the theory of limit linear series of Eisenbud and Harris \cite{EH1} that all $g$-cuspidal rational curves verify the Brill--Noether Theorem, and in particular they satisfy $\mbox{Cliff}(C)=\lfloor \frac{g-1}{2}\rfloor$. In this way, Green's Conjecture for $C$ is turned into a concrete question about the syzygies of~$\T$. We prove the following:

\begin{thm}
\label{thm:vanishing-tangential}
 If $g\geq 3$ and $\chr(\kk)=0$ or $\chr(\kk)\geq \frac{g+2}{2}$, then $K_{\lfloor \frac{g}{2}\rfloor,1}\bigl(\mathcal{T}, \OO_{\mathcal{T}}(1)\bigr)=0$.
\end{thm}

As a consequence of this result we give a new proof of Voisin's Theorem \cite{V02, V05} on the Generic Green Conjecture in characteristic zero, as well as an effective bound, in terms of the genus of the curve for the characteristics in which the Generic Green Conjecture holds.

\begin{thm}
\label{thm:genericGreen}
Let $g\geq 3$ and suppose that $\chr(\kk)=0$ or $\chr(\kk)\geq \frac{g+2}{2}$. For a general curve $C$ of genus $g$ we have that $K_{\lfloor \frac{g}{2}\rfloor,1}(C,\omega_C) = 0$, and $C$ verifies Green's Conjecture.
\end{thm}

Early on, Schreyer \cite{Sch1} observed that Green's Conjecture fails for some positive characteristics, for instance for genus $7$ and characteristics $2$, or genus $9$ and characteristics $3$. More recently, Eisenbud and Schreyer \cite[Conjecture 0.1]{ES} predicted that Green's Conjecture should hold for general curves of genus $g$ whenever $\chr(\kk)\geq \frac{g-1}{2}$. A version of Green's Conjecture involving arbitrary smooth curves in positive characteristic has been put forward recently by Bopp and Schreyer \cite{BS}.

\vskip 3pt

Our Theorem~\ref{thm:genericGreen} gives an almost complete answer to the Eisenbud--Schreyer conjecture and leaves at most one characteristic open. If $2\leq \chr(\kk)\leq \frac{g+1}{2}$ then the tangential variety $\T$ is contained in a rational normal scroll of codimension at least $\lfloor \frac{g}{2}\rfloor$. Therefore $K_{\lfloor \frac{g}{2} \rfloor,1}(\mathcal{T}, \OO_{\mathcal{T}}(1))\neq 0$ (see Remark~\ref{rem:T-on-scroll}) and the same non-vanishing holds for the general hyperplane section $C$. This shows that cuspidal curves cannot be used to prove the Eisenbud--Schreyer conjecture when  $\frac{g-1}{2}\leq\chr(\kk)\leq \frac{g+1}{2}$. One key ingredient in our proof of Theorem~\ref{thm:vanishing-tangential} comes from a solid understanding of the Hilbert function of finite length Koszul modules, as explained next.

\vskip 4pt

\noindent {\bf Koszul modules.} Suppose that $V$ is an $n$-dimensional $\kk$-vector space and fix a subspace $K\subseteq \bigwedge^2 V$ with $\dim(K)=m$. We denote by $S:=\mbox{Sym }V$ the symmetric algebra over $V$ and consider the Koszul complex resolving the residue field $\kk$:
$$\cdots \longrightarrow \bigwedge^3 V\otimes S\stackrel{\delta_3}\longrightarrow \bigwedge^2 V\otimes S\stackrel{\delta_2}\longrightarrow V\otimes S\stackrel{\delta_1}\longrightarrow S.$$
Truncating this complex to the last three terms, and restricting $\delta_2$ along the inclusion $\iota:K\hookrightarrow \bw^2V$ we obtain a $3$-term complex
\begin{equation}\label{eqn:W}
\xymatrixcolsep{5pc}
\xymatrix{
K \oo S \ar[r]^{\delta_2|_{K \oo S}} & V\oo S \ar[r]^{\delta_1} & S.
}
\end{equation}
Following \cite[\S 2]{PS15} as well as \cite{koszul-groups}, we define the {\em Koszul module} associated to the pair $(V,K)$ to be the middle homology of the complex (\ref{eqn:W}). We make the convention that $K$ is placed in degree zero, so that $W(V,K)$ is a graded $S$-module generated in degree zero. Using the exactness of the Koszul differentials  $\delta_{i,q}:\bigwedge^i V\otimes \Sym^qV \rightarrow \bigwedge^{i-1} V\otimes \Sym^{q+1} V$, we have the following alternative identification of the $q$-th graded piece of the Koszul module $W_q(V,K)=\mbox{Ker}(\delta_{1,q+1})/\delta_{2,q}\bigl(K\otimes \Sym^qV)$:

\begin{equation}
\label{eqnWqDef}
W_q(V,K)=\mathrm{Coker}\Bigl(\bigwedge^3V\otimes \Sym^{q-1}V\lra \frac{\bigwedge^2V}{K}\otimes \Sym^qV\Bigr) \cong \frac{\bigwedge^2 V\otimes \Sym^qV}{K\otimes \Sym^q V+\mbox{Ker}(\delta_{2,q})}.
\end{equation}

\vskip 3pt

The formation of the Koszul module $W(V,K)$ is natural in the following sense. An inclusion $K \subseteq K'$ induces a surjective morphism of graded $S$-modules
\begin{equation}
\label{eq:wnat}
W(V,K) \twoheadrightarrow W(V,K'),
\end{equation}
that is, \emph{bigger} subspaces $K\subseteq\bw^2 V$ correspond to \emph{smaller} Koszul modules. For instance, we have that $W(V,K)=0$ if and only if $K=\bw^2V$. We shall be interested more generally in studying Koszul modules that are finite dimensional as $\kk$-vector spaces, that is, those that satisfy $W_q(V,K)=0$ for $q\gg 0$. Since $W(V,K)$ is generated in degree zero, the vanishing $W_q(V,K)=0$ for some $q\geq 0$ implies that $W_{q'}(V,K)=0$ for all $q'\geq q$.

\vskip 4pt

We write $\iota^{\vee}:\bw^2 V^{\vee}\onto K^{\vee}$ for the dual to the inclusion $\iota$, and let $K^\perp:=\ker(\iota^{\vee})\subseteq \bigwedge^2 V^{\vee}$. It is shown in \cite[Lemma 2.4]{PS15} that the set-theoretic support of $W(V,K)$ in the affine space $V^\vee$ is given by the \emph{resonance variety} $\mc{R}(V,K)$, defined as
\begin{equation}
\label{eq:defr}
\mathcal R(V,K):=\left\{a\in V^\vee :  \mbox{ there exists }b\in V^\vee \mbox{ such that } a\wedge b\in K^\perp\setminus \{0\} \right\}\cup \{0\}.
\end{equation}
In particular, $W(V,K)$ has finite length if and only $\mc{R}(V,K)=\{0\}$. In view of (\ref{eq:defr}), this last condition is equivalent to the fact that the linear subspace $\bb{P}K^{\perp}\subseteq\bb{P}(\bw^2 V^{\vee})$ is disjoint from the Grassmann variety $$\GG:=\Gr_2(V^{\vee})$$ in its Pl\"ucker embedding, which can happen only when $m=\op{codim}(\bb{P}K^{\perp})> \dim (\GG)=2n-4$. Summarizing, we have the following equivalences:
\begin{equation}
\label{eq:mainequiv}
\mathbb P(K^\perp)\cap\GG=\emptyset \Longleftrightarrow \mathcal R(V,K)= \{0\}
\Longleftrightarrow \dim_{\kk}W(V,K) < \infty.
\end{equation}
Moreover, if the equivalent statements in (\ref{eq:mainequiv}) hold, then $m\geq 2n-3$.

\vskip 4pt

We provide a sharp effective bound for the vanishing of the graded components of a finite length Koszul module, extending the main result from \cite{koszul-groups} to (sufficiently) positive characteristics. Recalling that $n=\dim_{\kk}(V)$, we prove the following:

\begin{thm}
\label{thm:vanishing-koszul}
Suppose that $n\geq 3$. If $\chr(\kk)=0$ or $\chr(\kk)\geq n-2$, then we have the equivalence
\begin{equation}\label{eq:vanishing-koszul}
\mc{R}(V,K) = \{0\} \Longleftrightarrow W_q(V,K) = 0\mbox{ for }q\geq n-3.
\end{equation}
\end{thm}

Experiments in small characteristics suggest that the assumption $\chr(\kk)\geq n-2$ is probably necessary in order for (\ref{eq:vanishing-koszul}) to hold, but we do not know this in general. As the next theorem shows, the vanishing range $q\geq n-3$ is on the other hand optimal, since $W_{n-4}(V,K)\neq 0$ when $m=2n-3$ and $n\geq 4$.

\begin{thm}
\label{thm:boundHilb}
Suppose $\chr(\kk)=0$ or $\chr(\kk)\geq n-2$, and fix a subspace $K\subseteq \bigwedge ^2 V$.  If $\mathcal{R}(V,K) = \{0\}$, then
\[
\dim\, W_q(V,K) \leq {n+q-1 \choose q} \frac{(n-2)(n-q-3)}{q+2} \quad\mbox{ for }q=0, \ldots, n-4.
\]
Moreover, equality holds for all $q$ if $\dim(K)=2n-3$.
\end{thm}

The study of resonance varieties (and of the associated cohomology jumping loci) in Hodge theory has been initiated by Green and Lazarsfeld \cite{GL91}, and it has been actively pursued in the theory of hyperplane arrangements and topology. We refer to \cite{DPS}, \cite{PS15} and the references therein for an overview, and for the connection between topological invariants of groups and Koszul modules. Theorem \ref{thm:boundHilb} yields upper bounds for these invariants in a purely algebraic context, as explained in detail in \cite{koszul-groups}.

\vskip 4pt

\noindent  {\bf The Cayley-Chow form of the Grassmannian of lines.}
Theorem \ref{thm:vanishing-koszul} can be reformulated geometrically as follows. Recall that one defines the \emph{Cayley--Chow form} of the Grassmannian
$\GG\subseteq \mathbb P(\bigwedge^2V^\vee )$ in its Pl\"ucker embedding as the locus of linear subspaces of codimension $(2n-3)$ in $\mathbb P(\bigwedge^2V^\vee )$ that meet $\GG$. In view of (\ref{eq:mainequiv}), this can be seen equivalently as the locus of $(2n-3)$-dimensional subspaces in $\bigwedge^2 V$ for which $W(V,K)$ has non-zero resonance, that is, as
$$
\mathfrak{Chow}(\GG):=\Bigl\{K\in\mathrm{Gr}_{2n-3}\bigl(\bigwedge ^2V\bigr): \mathbb P(K^\perp)\cap \GG\ne\emptyset\Bigr\}.
$$
Note that $\mathfrak{Chow}(\GG)$ is a divisor in $\Gr_{2n-3}\bigl(\bigwedge^2 V\bigr)$ of degree equal to $\mbox{deg}(\GG)$, which is known to be computed by the Catalan number $\frac{1}{n-1}\cdot{2n-4\choose n-2}$. In view of the description (\ref{eqnWqDef}) for the space $W_{n-3}(V,K)$, an equivalent reformulation of Theorem \ref{thm:vanishing-koszul} (in the case when $\mbox{dim}(K)=2n-3$) is the following:
\begin{thm}\label{chow-form}
Suppose that $\chr(\kk)=0$ or $\chr(\kk)\geq n-2$. The Cayley--Chow form $\mathfrak{Chow}(\GG)$ consists set-theoretically of those subspaces $K\subseteq \bigwedge^2 V$ with $\dim(K)=2n-3$ satisfying
$$\Bigl(K\otimes \mathrm{Sym}^{n-3}(V)\Bigr)\cap \mathrm{Ker}(\delta_{2,n-3})\neq 0.$$
\end{thm}
In other words, the Cayley--Chow form of $\GG$ is the pull-back of the divisorial Schubert cycle defined by the Koszul space $\mbox{Ker}\bigl(\delta_{2,n-3}:\bigwedge^2 V\otimes \mbox{Sym}^{n-3} V\rightarrow V\otimes \mbox{Sym}^{n-2} V\bigr)$ under the
map $$\Gr_{2n-3}\bigl(\bigwedge^2 V\bigr)\rightarrow \Gr_{(2n-3){2n-4\choose n-1}}\Bigl(\bigwedge^2 V\otimes \mbox{Sym}^{n-3} V \Bigr), \ \mbox{ } \ \ \ K\mapsto K\otimes \mbox{Sym}^{n-3} V.$$

\vskip 5pt

Returning to the Generic Green Conjecture, our new approach and the proof of Theorem~\ref{thm:vanishing-tangential} is based on exhibiting a close relationship between the syzygies of the tangential variety $\mc{T}$ and the graded components of a specific finite length Koszul module,  whose study was initiated in \cite[Section~3.I.B]{E91}, and which is called a \emph{Weyman module}. This relationship lies at the heart of this paper, and we explain it in the remaining part of the Introduction.

\vskip 4pt

\noindent {\bf The syzygies of the tangential variety of the rational normal curve.} We assume for the rest of the Introduction that $\chr(\kk)\neq 2$. We prove in Theorem~\ref{thm:shape-res-R} that the homogeneous coordinate ring of $\mc{T}$ is Gorenstein, with Castelnuovo--Mumford regularity $3$. In other words, the dimensions $b_{i,j}$ of the Koszul cohomology groups $K_{i,j}\bigl(\mc{T},\OO_{\mc{T}}(1)\bigr)$ are encoded in a table having the following shape:
\begin{equation}\label{eq:betti-T}
\begin{array}{c|ccccccc}
     &0&1&2&\cdots&g-4&g-3&g-2 \\ \hline
     0&1&-&-&\cdots&-&-&-\\
     1&-&b_{1,1}&b_{2,1}&\cdots&b_{g-4,1}&b_{g-3,1}&-\\
     2&-&b_{1,2}&b_{2,2}&\cdots&b_{g-4,2}&b_{g-3,2}&-\\
     3&-&-&-&\cdots&-&-&1\\
\end{array}
\end{equation}
where a dash in position $(i,j)$ indicates the vanishing of the corresponding Koszul cohomology group. We prove in Theorem~\ref{thm:nonvanishing-kij} that:
\begin{thm}\label{thm:vanish-ki2}
 If $p=\chr(\kk)$ satisfies $p=0$ or $p\geq \frac{g+2}{2}$, then
 \[b_{i,2} \neq 0 \quad\mbox{ if and only if }\quad \frac{g-2}{2} \leq i \leq g-3.\]
\end{thm}
The Gorenstein property implies that $b_{i,1} = b_{g-2-i,2}$ for all $i$, so the vanishing statement in Theorem~\ref{thm:vanishing-tangential} is equivalent to $b_{g-2-\lfloor \frac{g}{2}\rfloor,2} = 0$, which follows from Theorem~\ref{thm:vanish-ki2}. In Theorem~\ref{thm:nonvanishing-kij} we also give a similar characterization for the (non-)vanishing of $b_{i,2}$ when $3\leq p\leq\frac{g+1}{2}$, and in Section~\ref{subsec:char2} we consider $p=2$. To prove Theorem~\ref{thm:nonvanishing-kij}, we realize the groups $K_{i,2}\bigl(\mc{T},\OO_{\mc{T}}(1)\bigr)$ as graded components of Weyman modules as explained next, and then apply Theorem~\ref{thm:vanishing-koszul} and the explicit description (\ref{eq:defr}) of the set-theoretic support of a Koszul module.

\vskip 4pt

\noindent {\bf Weyman modules.}
We fix a $\kk$-vector space $U$ of dimension two, and let $\bb{P}U$ denote the projective space of one-dimensional subspaces of $U$. For each $d\geq 0$ we have $H^0(\bb{P}U,\OO_{\bb{P}U}(d)) = \Sym^d(U^{\vee})$, and we define the \defi{$d$-th divided power of $U$} by
\begin{equation}\label{eq:DdU}
\D^d U := \Bigl( \Sym^d(U^{\vee}) \Bigr)^{\vee}.
\end{equation}
If $\chr(\kk)=0$, or more generally, if all the binomial coefficients ${d\choose i}$ are invertible in $\kk$, then there exists a natural isomorphism $\D^d U \cong \Sym^d U$, as explained in Section~\ref{subsec:Sym-D}. Readers interested only in characteristic zero can use this identification throughout the paper.

\vskip 4pt

For each $d\geq 1$ we consider the \defi{Gaussian--Wahl map} associated with the line bundle $\OO_{\bb{P}U}(d)$ (see \cite{Wahl}):
\[\mu_1: \bw^2 H^0\bigl(\bb{P}U,\OO_{\bb{P}U}(d)\bigr) \lra H^0\bigl(\bb{P}U,\omega_{\bb{P}U} \oo \OO_{\bb{P}U}(2d)\bigr).\]
We choose a basis $(1,y)$ for $U^{\vee}$, so that $\Sym^d(U^{\vee})$ can be identified with the space of polynomials of degree at most $d$ in $y$. Making the identification $\omega_{\bb{P}U} \cong \OO_{\bb{P}U}(-2)$, we can rewrite $\mu_1$ explicitly as
\begin{equation}\label{eq:Wahl-map}
 \mu_1:\bw^2 \Sym^{d}(U^{\vee}) \lra \Sym^{2d-2}(U^{\vee}),\quad \mu_1(y^i \wedge y^j) = (i-j)\cdot y^{i+j-1}\mbox{ for }0\leq i,j\leq d.
\end{equation}
In characteristic zero (or sufficiently large), we have the Clebsch--Gordan decomposition
\[\bigwedge^2 \mbox{Sym}^d(U^{\vee})=\bigoplus_{j=0}^{\lfloor\frac{d-1}{2}\rfloor} \Sym^{2d-2-4j}(U^{\vee}),\]
and the map $\mu_1$ is simply the projection onto the factor $\Sym^{2d-2}(U^{\vee})$ (see \cite[Exercise~11.30]{FH}).

\vskip 3pt

Under the assumption that $\chr(\kk)\neq 2$, the Wahl map (\ref{eq:Wahl-map}) is surjective. Dualizing it, using (\ref{eq:DdU}), and writing $V^{(d)} = \D^d U$ and $K^{(d)} = \D^{2d-2}U$, we get an inclusion $\Delta_1: K^{(d)} \hookrightarrow \bw^2 V^{(d)}$, and define
\[ W^{(d)} := W\bigl(V^{(d)},K^{(d)}\bigr).\]
Following \cite{E91}, we call $W^{(d)}$ a \defi{Weyman module}. The key observation is then the following (Theorem~\ref{thm:Ki2=Wi+2}):
\begin{thm}\label{thm:intro-ki2=Wi+2}
 If $\chr(\kk)\neq 2$,  then for each $i=1,\ldots,g-3$ we have a natural identification
 \[K_{i,2}\bigl(\mc{T},\OO_{\mc{T}}(1)\bigr) = W^{(i+2)}_{g-3-i}.\]
\end{thm}

This result reveals a peculiar property of the Weyman module $W^{(i+2)}$, for its graded pieces are isomorphic to the Koszul cohomology groups $K_{i,2}$ of generic curves of \emph{varying} genera!  The analysis of the Koszul cohomology groups $K_{i,j}\bigl(\mc{T},\OO_{\mc{T}}(1)\bigr)$ is the most technical part of our paper, and is explained in detail in Section~\ref{sec:repth}. It relies on applications of the Kempf--Weyman technique for constructing syzygies, and a summary that avoids technicalities is presented in what follows. Some of the constructions that we use in Section~\ref{sec:repth} were already outlined in~\cite{E91}, but a formal verification of their correctness requires quite a bit of care. An essential tool we employ in our arguments is an \emph{explicit, characteristic-free version of Hermite reciprocity} which is, as far as we know, new. It consists of an explicit, natural isomorphism
\[\Sym^d(\D^i U) \lra \bw^i(\Sym^{d+i-1}U)\]
for all $d,i\geq 1$, which is described in terms of the change of bases between elementary symmetric polynomials and Schur polynomials. This isomorphism is explained in Section~\ref{subsec:Hermite}.

\vskip 4pt

\noindent{\bf{Syzygies of the tangent developable $\mc{T}$: the geometric intuition.}}
We conclude our Introduction by discussing the maps and spaces involved in the identification in Theorem~\ref{thm:intro-ki2=Wi+2}. We consider the diagram
\begin{equation}\label{eq:key-diagram}
\begin{aligned}
 \xymatrix{
  \D^{2i+2}U \oo \Sym^{g-3-i}(\D^{i+2} U) \ar[rr] \ar[d]_{\delta_2} & & \D^{2i+2}U \oo \bw^{i+2} \Sym^{g-2} U \ar[d]^{\tl{\delta}_2} \\
  \D^{i+2}U \oo \Sym^{g-2-i}(\D^{i+2} U) \ar[d]_{\delta_1} \ar[rr] & & \D^{i+2}U \oo \bw^{i+2} \Sym^{g-1} U \ar[d]^{\tl{\delta}_1} \\
  \Sym^{g-1-i}(\D^{i+2} U) \ar[rr] & & \bw^{i+2} \Sym^{g} U \\
}
 \end{aligned}
 \end{equation}
where the horizontal arrows are identifications provided by Hermite reciprocity, the left column is the $3$-term complex whose homology defines $W^{(i+2)}_{g-3-i}$, while the maps $\tl{\delta}_1$ and $\tl{\delta}_2$ are the unique ones making the diagram commute. To prove Theorem~\ref{thm:intro-ki2=Wi+2}, it suffices to explain why $K_{i,2}\bigl(\mc{T},\OO_{\mc{T}}(1)\bigr)$ may be identified with $\ker(\tl{\delta}_1) / \Im(\tl{\delta}_2)$. This is achieved by describing geometrically the groups in the right column of (\ref{eq:key-diagram}).

\vskip 3pt

Recall that $\Gamma\subseteq \PP^g$ denotes the rational normal curve of degree $g$, which we view as the Veronese embedding of $\PP^1:=\mathbb P\bigl(U^{\vee}\bigr)$, where $U$ is a two-dimensional $\kk$-vector space. We identify $\OO_{\PP^1}(g)=\OO_{\Gamma}(1)$ and $H^0(\Gamma, \OO_{\Gamma}(1))\cong H^0(\PP^g, \OO_{\PP^g}(1))\cong \mbox{Sym}^g U$. We introduce the \defi{jet bundle} $\mc{J}:=\mc{P}\bigl(\OO_{\PP^1}(g)\bigr)$ of $\OO_{\PP^1}(g)$, which sits canonically in an exact sequence
$$0\longrightarrow \omega_{\PP^1}\otimes \OO_{\PP^1}(g)\longrightarrow \mc{J} \longrightarrow \OO_{\PP^1}(g)\longrightarrow 0.$$
We denote by $\PP(\mc{J})\rightarrow \PP^1$ the corresponding projectivized tangent bundle.  The Taylor homomorphism of the jet bundle is a surjective map of sheaves $\mbox{Sym}^g U\otimes \OO_{\PP^1}\longrightarrow \mc{J}$, defined by $f\oo 1 \mapsto (df,f)$. It gives rise to a morphism $\tau:\PP(\mc{J})\rightarrow \PP^g$, mapping $\PP(\mc{J})$ birationally onto $\mc{T}$, so $\tau:\PP(\mc{J})\rightarrow\mc{T}$ provides a resolution of singularities for the tangent developable. There exists an exact sequence on $\PP^g$
\[0\lra \OO_{\mc{T}}\lra \tau_*\OO_{\PP(\mc{J})}\lra \omega_{\Gamma}\rightarrow 0,\]
which induces a long exact sequence in Koszul cohomology
\[ \cdots \lra K_{i+1,1}\bigl(\mc{T}, \tau_*\OO_{\PP(\mc{J})}, \OO_{\mc{T}}(1)\bigr) \overset{\zeta}{\lra} K_{i+1,1}\bigl(\Gamma, \omega_{\Gamma}, \OO_{\Gamma}(1)\bigr) \lra K_{i,2}\bigl(\mc{T},\OO_{\mc{T}}(1)\bigr) \lra \]
\[\lra K_{i,2}\bigl(\mc{T}, \tau_*\OO_{\PP(\mc{J})}, \OO_{\mc{T}}(1)\bigr) \lra \cdots \]
The desired identification of $K_{i,2}\bigl(\mc{T},\OO_{\mc{T}}(1)\bigr)\cong \ker(\tl{\delta}_1) / \Im(\tl{\delta}_2)$ follows if we can prove:
\begin{enumerate}
 \item $K_{i+1,1}\bigl(\mc{T}, \tau_*\OO_{\PP(\mc{J})}, \OO_{\mc{T}}(1)\bigr) \cong \D^{2i+2}U \oo \bw^{i+2} \Sym^{g-2} U$ and $K_{i,2}\bigl(\mc{T}, \tau_*\OO_{\PP(\mc{J})}, \OO_{\mc{T}}(1)\bigr)=0$.
 \item $K_{i+1,1}\bigl(\Gamma, \omega_{\Gamma}, \OO_{\Gamma}(1)\bigr) \cong
 \ker(\tl{\delta}_1)$.
 \item Under the above isomorphisms, the map $\zeta$ can be identified with $\tl{\delta}_2$.
\end{enumerate}

\vskip 4pt

Step  (1) is carried out in Proposition~\ref{prop:resolution-Rbar}, where we establish the identifications
$$K_{i+1,1}\bigl(\mc{T}, \tau_*\OO_{\PP(\mc{J})}, \OO_{\mc{T}}(1)\bigr)\cong H^1\bigl(\Gamma,\bw^{i+2}\xi\bigr) \cong \D^{2i+2} U\otimes  \bigwedge^{i+2} \mathrm{Sym}^{g-2} U,$$
with $\xi$ being the kernel of the Taylor homomorphism. We also show that $K_{i,2}\bigl(\mc{T}, \tau_*\OO_{\PP(\mc{J})}, \OO_{\mc{T}}(1)\bigr) =0$. Concerning (2), using the fact that the kernel bundle $M_{\OO_{\Gamma}(1)}$ is obtained by restricting the tautological subbundle on $\PP^g$ to $\Gamma$, it follows from \cite[Theorem~5.8]{E04} that
\[ K_{i+1,1}\bigl(\Gamma, \omega_{\Gamma}, \OO_{\Gamma}(1)\bigr) = \ker\Bigl\{H^1\bigl(\Gamma,\bw^{i+2}M_{\OO_{\Gamma}(1)} \oo \omega_{\Gamma}\bigr) \overset{\alpha}{\lra} \bw^{i+2}\Sym^g U  \oo H^1(\Gamma,\omega_{\Gamma})  \Bigr\},\]
Since $H^1(\Gamma,\bw^{i+2}M_{\OO_{\Gamma}(1)} \oo \omega_{\Gamma}) \cong \D^{i+2}U \oo \bw^{i+2} \Sym^{g-1} U$ and $H^1(\Gamma,\omega_{\Gamma}) \cong \kk$, the maps $\a$ and $\tl{\delta}_1$ have the same source and the same target. It is plausible then that $\a=\tl{\delta}_1$ and that (2) should hold (see Corollary~\ref{cor:C_i=ker-pi+1}). A verification of (3) would then conclude Theorem~\ref{thm:intro-ki2=Wi+2}.

\vskip 3pt

It would be very interesting to obtain the maps $\tilde{\delta}_1$ and $\tilde{\delta}_2$ geometrically, as maps induced in cohomology. It is unlikely that this can be realized by working directly on the curve $\Gamma$, but perhaps, as suggested to us by Ein and Lazarsfeld,  on a thickening of $\Gamma$ in $\PP^g$. Achieving this goal would give a truly geometric understanding of the syzygies of the tangent developable $\mathcal{T}$ and the forthcoming  paper of of Ein and Lazarsfeld \cite{lazarsfeld} goes some way  towards  geometrizing  the proof.

\vskip 3pt

In this paper, we do not execute this geometric approach, particularly because it requires at each step a very careful bookkeeping of the identifications involved.  Instead we take an \emph{algebraic approach} in Section~\ref{sec:repth}, which goes beyond the scope of this Introduction, and whose extended summary is included in Section~\ref{subsec:summary-cor=kw}. The Koszul cohomology calculations sketched above are replaced in our approach by syzygy calculations involving the Kempf--Weyman technique.


\vskip 4pt


\vskip 3pt

\noindent {\bf Summary.} The paper is organized as follows. In Section~\ref{sec:fin-length-koszul} we discuss Koszul modules of finite length, a characteristic-free instance of Bott's vanishing theorem, and we prove Theorems~\ref{thm:vanishing-koszul} and~\ref{thm:boundHilb}. In Section~\ref{sec:Hermite-rec} we recall the characteristic-free construction of divided, exterior and symmetric powers, and prove our explicit version of Hermite reciprocity. In Section~\ref{sec:KW} we summarize some basic results on Koszul cohomology and minimal resolutions, and discuss the Kempf--Weyman technique for constructing syzygies. In Section~\ref{sec:repth} we analyze the syzygies of the tangential variety to a rational normal curve, and prove Theorems~\ref{thm:vanishing-tangential},~\ref{thm:vanish-ki2}, and~\ref{thm:intro-ki2=Wi+2}. We end with Section~\ref{sec:moduli-Green} where we use moduli of pseudo-stable curves and Theorem~\ref{thm:vanishing-tangential} to deduce (in suitable characteristics) the Generic Green Conjecture (Theorem~\ref{thm:genericGreen}).

\vskip 4pt

\begin{ack}
We acknowledge with thanks the contribution of A. Suciu. This project, including the companion paper \cite{koszul-groups}, started with the paper \cite{PS15}, and since then we benefited from numerous discussions with him. We warmly thank A. Beauville, L. Ein, D. Eisenbud, B. Klingler, P. Pirola, F.-O. Schreyer and C. Voisin for interesting discussions related to this circle of ideas. We are particularly grateful to R. Lazarsfeld who read an early version of the paper and suggested many improvements which significantly clarified the exposition.

\vskip 3pt

{\small{Aprodu was partially supported by the Romanian Ministry of Research and
Innovation, CNCS - UEFISCDI, grant
PN-III-P4-ID-PCE-2016-0030, within PNCDI III. Farkas was supported by the DFG grant \emph{Syzygien und Moduli}. Raicu was supported by the Alfred P. Sloan Foundation and by the NSF Grant No.~1600765. Weyman was partially supported by the Sidney Professorial Fund and the NSF grant No.~1802067.
}}
\end{ack}

\section{Koszul modules in algebraic geometry}
\label{sec:fin-length-koszul}

The goal of this section is to extend the main result of \cite{koszul-groups} to positive characteristic, by establishing Theorem~\ref{thm:vanishing-koszul}. Once this theorem is proved, the arguments from \cite[Section~3.3]{koszul-groups} carry over to deduce Theorem~\ref{thm:boundHilb}. We prove Theorem~\ref{thm:vanishing-koszul} using a suitable positive characteristic version of Bott vanishing, in conjunction with the hypercohomology spectral sequence that was featured in the work of Voisin on the Green conjecture \cite{V02}. We recall the basics on partitions and Schur functors in Section~\ref{subsec:Schur-functors}, and use them to formulate appropriate versions of Bott's theorem for flag varieties in Section~\ref{subsec:bott}, and for Grassmannians in Section~\ref{subsec:Grassmannians}. The proof of the vanishing Theorem~\ref{thm:vanishing-koszul} is explained in Section~\ref{subsec:vanish-Koszul}.

\subsection{Partitions and Schur functors}\label{subsec:Schur-functors}

The reference for this part is  \cite[Chapter~2]{weyman}. A partition $\ll$ is a sequence $\ll=(\ll_1\geq\ll_2\geq\cdots\geq\ll_r)$
of non-negative integers, where the \defi{parts} $\ll_j$ of $\ll$ are assumed to be eventually $0$. The \emph{size} of $\ll$ is written as $|\ll|=\ll_1+\ll_2+\cdots+\ll_r$. For each partition $\ll$ we define the {\em conjugate partition} $\ll'$ by letting $\ll'_j$ be the number of parts $\ll_i$ satisfying $\ll_i\geq j$. We often omit trailing zeros, and group parts of the same size together by writing $(a^b)$ for the sequence $(a,a,\ldots,a)$, where the entry $a$ is repeated $b$ times. For instance, we write $(4^3,2^4,1)$ for $(4,4,4,2,2,2,2,1,0,0,0,\ldots)$.

\vskip 3pt

We let $\kk$ be a field, and for each partition $\ll$ we consider the \defi{Schur functor $\bb{L}_{\ll}$} on the category of finite dimensional $\kk$-vector spaces, defined as in \cite[Section~2.1]{weyman}: For each vector space $V$, we construct $\bb{L}_{\ll}V$  as a suitable quotient
\[\bw^{\ll_1}V \oo \bw^{\ll_2}V \oo \cdots \oo \bw^{\ll_r}V \onto \bb{L}_{\ll}V,\]
satisfying the following properties:
\begin{itemize}
 \item If $\ll_1>\dim(V)$, then $\bb{L}_{\ll}(V)=0$.
 \item If $\ll_1=r$ and $\ll_i=0$ for $i>1$, then $\bb{L}_{\ll}V = \bw^r V$.
 \item If $\ll=(1^d)$, then $\bb{L}_{\ll}V = \Sym^d V$.
\end{itemize}
More generally, $\bb{L}_{\ll}$ can be applied to any locally free sheaf $\mc{E}$ on a variety $B$ over $\kk$, yielding a locally free sheaf $\bb{L}_{\ll}\mc{E}$ on $B$.

\vskip 3pt

We write $\bb{Z}^n_{\op{dom}}$ for the set of \emph{dominant weights} in $\bb{Z}^n$, that is, tuples $\ll=(\ll_1,\ldots,\ll_n)\in\bb{Z}^n$ satisfying $\ll_1\geq\ll_2\geq\ldots\geq\ll_n$, and note that any partition with $\ll_{n+1}=0$ can be identified with a dominant weight in $\bb{Z}^n_{\op{dom}}$.  Given $\ll\in\bb{Z}^n_{\op{dom}}$, for a locally free sheaf $\mathcal{E}$ of rank $n$ on $B$ we set
\[\bb{S}_{\ll}\mc{E}:= \bb{L}_{\mu'}\mc{E} \oo (\det\,\mc{E})^{\oo \ll_n}\]
where $\mu=(\ll_1-\ll_n,\ll_2-\ll_n,\ldots,\ll_{n-1}-\ll_n)$, and refer to $\bb{S}_{\ll}$ also as a \defi{Schur functor}. We have the identifications $\bb{S}_{(d,0^{n-1})}V = \Sym^d V$ and $\bb{S}_{(1^r,0^{n-r})}V = \bw^r V$ for $d\geq 0$ and $0\leq r\leq n$.

\subsection{Bott's Theorem for flag varieties {\cite[Ch. 3--4]{weyman}, \cite[Sections ~II.4--II.5]{jantzen}}}
\label{subsec:bott}

Let $V$ be an $n$-dimensional vector space over an algebraically closed field $\kk$. We denote by $\FF(V)$ the variety of complete flags
\[V_{\bullet}:\quad V = V_n \onto V_{n-1}\onto \cdots \onto V_1\onto V_0=0,\]
where $V_p$ is a $p$-dimensional quotient of $V$ for each $p=0,\ldots,n$. We write $\mc{Q}_p(V)$ for the tautological rank $p$ quotient bundle on $\FF(V)$, whose fiber over a flag $V_{\bullet}$ is $V_p$. We consider the tautological line bundle $\mc{L}_p(V)$ on $\FF(V)$ defined as the kernel of the quotient map $\mc{Q}_p(V)\onto\mc{Q}_{p-1}(V)$. For $\ll\in\bb{Z}^n$ we define
\[\mc{L}^{\ll}(V) := \mc{L}_1(V)^{\oo \ll_1} \oo \mc{L}_2(V)^{\oo \ll_2} \oo \cdots \oo \mc{L}_n(V)^{\oo \ll_n}.\]
More generally, if $B$ is a variety over $\kk$ and $\mc{E}$ is a locally free sheaf of rank $n$ on~$B$ we denote by $\FF_B(\mc{E})$ the \defi{relative flag variety}, equipped with a natural map
\[\pi:\FF_B(\mc{E}) \lra B,\]
such that for every point $b\in B$ the fiber $\pi^{-1}(b)$ is isomorphic to the variety $\FF(\mc{E}_b)$ of flags on the fiber $\mc{E}_b = \mc{E} \oo \kk(b)$. We define the tautological bundles $\mc{Q}_p(\mc{E})$ and $\mc{L}^{\ll}(\mc{E})$ in analogy to the absolute case, and note that we have a tautological sequence of quotient maps on $\FF_B(\mc{E})$
\[\pi^*(\mc{E}) = \mc{Q}_n(\mc{E}) \onto \mc{Q}_{n-1}(\mc{E}) \onto \cdots \onto \mc{Q}_1(\mc{E}) \onto \mc{Q}_0(\mc{E}) = 0.\]
We need the following characteristic-free version of Bott's Theorem, which is due to Kempf (and often referred to as \defi{Kempf vanishing}) -- see \cite[Section~II.4]{jantzen} and \cite[Theorem ~4.1.10]{weyman}.

\begin{thm}\label{thm:kempf-vanishing}
 If $\mc{E}$ is a locally free sheaf of rank $n$ on $B$ and $\ll\in\bb{Z}^n_{\op{dom}}$, then
 \begin{equation}\label{eq:Rpi*}
 R^k\pi_*(\mc{L}^{\ll}(\mc{E})) = \begin{cases}
 \bb{S}_{\ll}\mc{E} & \mbox{if }k=0; \\
 0 & \mbox{if }k>0.
 \end{cases}
 \end{equation}
 In particular, if $B=\Spec(\kk)$ and $\mc{E}=V$, then $\mc{L}^{\ll}(V)$ has vanishing higher cohomology and $$H^0(\FF(V),\mc{L}^{\ll}(V)) = \bb{S}_{\ll}V.$$
\end{thm}

We shall be interested more generally in the cohomology of line bundles $\mc{L}^{\ll}(V)$ when $\ll$ is not dominant. To that end we follow the analysis described in \cite[Section II.5]{jantzen}. We consider the action of the symmetric group on $\bb{Z}^n$ defined by
\[s_i\bullet\ll := (\ll_1,\ldots,\ll_{i-1},\ll_{i+1}-1,\ll_i+1,\ll_{i+2},\ldots,\ll_n),\]
where $s_i$ denotes the transposition $(i,i+1)$.

\begin{proposition}[{\cite[Proposition II.5.4]{jantzen}}]\label{prop:jantzen-cohom}
 Suppose that $V$ is a $\kk$-vector space with $\dim(V)=n$, consider a weight $\ll\in\bb{Z}^n$, and fix $1\leq i\leq n-1$.
\begin{enumerate}
 \item[(a)] If $\ll_i = \ll_{i+1}-1$, then $H^k(\FF(V),\mc{L}^{\ll}(V))=0$ for all $k$.
 \item[(b)] Suppose that $\ll_i\geq\ll_{i+1}$ and that either $\op{char}(\kk)=0$, or $\op{char}(\kk)=p>0$ and furthermore
 \[ \ll_i-\ll_{i+1}\leq p-1.\]
 We have that for all $k\in\bb{Z}$
 \[H^k\bigl(\FF(V),\mc{L}^{\ll}(V)\bigr) = H^{k+1}\bigl(\FF(V),\mc{L}^{s_i\bullet\ll}(V)\bigr).\]
\end{enumerate}
\end{proposition}

We shall use Proposition~\ref{prop:jantzen-cohom} in order to prove the following vanishing result, which plays a key role in our vanishing theorem for Koszul modules.

\begin{lem}\label{lem:coh-vanishing}
 Let $V$ denote a $\kk$-vector space of dimension $n\geq 3$, and assume that $p:=\op{char}(\kk)$ satisfies either $p=0$ or $p\geq n-2$. Suppose $2\leq r\leq 2n-3$ and consider the weight
 \[\ll = (n-2-r,1-r,0^{n-2})\in\bb{Z}^n.\]
 We have that
 \begin{equation}\label{eq:Hr-1=0}
 H^{r-1}\bigl(\FF(V),\mc{L}^{\ll}(V)\bigr)=0.
 \end{equation}
\end{lem}

\begin{proof}
 Our proof consists of several applications of Proposition~\ref{prop:jantzen-cohom}, whose hypothesis is more restrictive when $p>0$. It is then sufficient to consider the case $p\geq n-2$ for the rest of the argument. We split our analysis into two cases. To simplify notation, we write $\FF$ for $\FF(V)$ and $\mc{L}^{\mu}$ for $\mc{L}^{\mu}(V)$.  

\vskip 3pt

Suppose that $2\leq r\leq n-1$ and consider the sequence of partitions $\ll^1,\ldots,\ll^{r-1}$ defined by
\[\ll^1=(n-2-r,(-1)^{r-1},0^{n-r})\mbox{ and }\ll^{j+1}:= s_{r-j}\bullet \ll^j, \mbox{ for }j=1,\ldots,r-2.\]
Note that $\ll^j = \bigl(n-2-r,(-1)^{r-1-j},-j,0^{n-r+j-1}\bigr)$ for $1\leq j\leq r-1$, and that in particular $\ll^{r-1}=\ll$. Since $\ll^1_r = \ll^1_{r+1}-1$, it follows from Proposition~\ref{prop:jantzen-cohom}(a) with $i=r$ that
\begin{equation}\label{eq:H1=0}
H^1\bigl(\FF,\mc{L}^{\ll^1}\bigr)=0.
\end{equation}
We have for $1\leq j\leq r-2$ that the following inequalities
\[0\leq \ll^{j}_{r-j} - \ll^{j}_{r-j+1} = j-1 \leq r-3\leq n-4 < p-1\]
hold, so it follows from Proposition~\ref{prop:jantzen-cohom}(b) with $i=r-j$ that
\begin{equation}\label{eq:Hj-1=Hj}
 H^{j}\bigl(\FF,\mc{L}^{\ll^{j}}\bigr) = H^{j+1}\bigl(\FF,\mc{L}^{\ll^{j+1}}\bigr) \  \mbox{ for  } j=1,\ldots,r-2.
\end{equation}
Combining (\ref{eq:H1=0}) with (\ref{eq:Hj-1=Hj}) and the fact that $\ll^{r-1}=\ll$, we obtain (\ref{eq:Hr-1=0}).

\vskip 4pt

Suppose now that $n\leq r\leq 2n-3$, and consider the partitions $\ll^1,\ldots,\ll^{r-n+1}$ defined by
\[\ll^1=\bigl(-r,(-1)^{r-n+1},0^{2n-2-r}\bigr)\mbox{ and } \ll^{j+1}:= s_{r-n+2-j}\bullet \ll^{j} \  \mbox{ for } j=1,\ldots,r-n.\]
Note that $\ll^j = \bigl(-r,(-1)^{r-n+1-j},-j,0^{2n-3-r+j}\bigr)$ for $1\leq j\leq r-n+1$. Since $\ll^1_{r-n+2} = \ll^1_{r-n+3}-1$, it follows from Proposition~\ref{prop:jantzen-cohom}(a) with $i=r-n+2$ that
\begin{equation}\label{eq:Hn+1=0}
H^{n}\bigl(\FF,\mc{L}^{\ll^1}\bigr)=0.
\end{equation}
For $1\leq j\leq r-n$ we obtain  that the following inequalities
\[0\leq \ll^{j}_{r-n+2-j} - \ll^{j}_{r-n+3-j} = j-1 \leq r-n-1\leq n-4 < p-1\]
hold, so it follows from Proposition~\ref{prop:jantzen-cohom}(b) with $i=r-n+2-j$ that
\begin{equation}\label{eq:Hn+j-1=Hn+j}
 H^{n+j-1}(\FF,\mc{L}^{\ll^j}) = H^{n+j}(\FF,\mc{L}^{\ll^{j+1}}), \mbox{ for all }j=1, \ldots, r-n.
\end{equation}
Combining (\ref{eq:Hn+1=0}) with (\ref{eq:Hn+j-1=Hn+j}) we conclude that
\[H^r(\FF,\mc{L}^{\ll^{r-n+1}})=0,\]
and note that $\ll^{r-n+1}=(-r,n-1-r,0^{n-2})$. We now observe that $\ll^{r-n+1} = s_1\bullet \ll$ and that
\[0\leq \ll_1-\ll_2 = n-3 \leq p-1\]
so we can apply Proposition~\ref{prop:jantzen-cohom}(b) to obtain the vanishing statements
\[H^{r-1}\bigl(\FF,\mc{L}^{\ll}\bigr) = H^r\bigl(\FF,\mc{L}^{\ll^{r-n+1}}\bigr)=0,\]
concluding our proof.
\end{proof}

\subsection{Bott's Theorem for Grassmannians}
\label{subsec:Grassmannians}

Let $\GG(V)=\Gr_2(V^{\vee})$ be the Grassmannian of $2$--dimensional quotients of $V$, or equivalently $2$-dimensional subspaces of $V^{\vee}$. We write $\GG:=\GG(V)$ and consider the tautological exact sequence
\begin{equation}\label{eq:tautological}
0\lra\mc{R}\lra V\oo\mc{O}_{\GG}\lra\mc{Q}\lra 0,
\end{equation}
where $\mc{R}$ (respectively $\mc{Q}$) denotes the universal rank $n-2$ subbundle (respectively  rank $2$ quotient bundle) of the trivial bundle $V\otimes \OO_{\GG}$. We write $\mc{O}_{\GG}(1):=\bw^2\mc{Q}$ for the Pl\"ucker line bundle and note that for every $\a=(\a_1,\a_2)\in\bb{Z}^2_{\op{dom}}$ we have the identification
\begin{equation}\label{eq:SS-alpha}
\bb{S}_{\a}\mc{Q} = \Sym^{\a_1-\a_2}\mc{Q} \,\oo \mc{O}_{\GG}(\a_2).
\end{equation}
We wish to compute the cohomology groups of some of the sheaves $\bb{S}_{\a}\mc{Q}$. To that end we reduce the calculation to the case of line bundles on the complete flag variety $\FF(V)$. We write $\FF:=\FF(V)$ and consider the natural map
\begin{equation}\label{eq:def-psi}
\psi:\FF \to \GG,\mbox{ given by }V_{\bullet}\mapsto V_2,
\end{equation}
which enables us to realize $\FF$ as the fiber product $\FF_{\GG}(\mc{Q}) \times_{\GG}\FF_{\GG}(\mc{R})$. Alternatively, let $Y:=\FF_{\GG}(\mc{Q})$ be the \defi{partial flag variety} parametrizing partial flags
\[V \onto V_2\onto V_1\onto V_0=0\]
and write $\psi_1:Y\to\GG$ for the structure map when we think of $Y$ as a relative flag variety over $\GG$ (which in this case is a $\bb{P}^1$-bundle). We then identify $\FF$ with the relative flag variety $\FF_Y(\psi_1^*\mc{R})$, and write $\psi_2:\FF\to Y$ for the structure map. In particular, we have a factorization
$$\psi=\psi_1\circ\psi_2:\FF\rightarrow \GG,$$ where
$\psi_2$ forgets the subspaces $V_{n-1},\ldots, V_3$ and  $\psi_1$  forgets $V_1$ respectively. Recall that $\mathcal{L}^{\lambda}=\mathcal{L}^{\lambda}(V)$.

\begin{thm}\label{thm:bott}
 Let $\a\in\bb{Z}^2_{\op{dom}}$ and  $\b\in\bb{Z}^{n-2}_{\op{dom}}$, and let $\ll=(\a|\b)\in\bb{Z}^n$ denote their concatenation. Using notation (\ref{eq:def-psi}) we have
 \begin{equation}\label{eq:Rpsi*}
 R^k\psi_*(\mc{L}^{\ll}) = \begin{cases}
 \bb{S}_{\a}\mc{Q} \oo \bb{S}_{\b}\mc{R} & \mbox{if }k=0; \\
 0 & \mbox{if }k>0.
 \end{cases}
 \end{equation}
In particular, we have that
\begin{equation}\label{eq:HF=HG}
 H^k\bigl(\GG,\bb{S}_{\a}\mc{Q} \oo \bb{S}_{\b}\mc{R}\bigr) = H^k(\FF,\mc{L}^{\ll})\ \ \mbox{ for all }k,
\end{equation}
and furthermore:
\begin{enumerate}
\item\label{it:1} If $q\geq 0$, then $H^0(\GG,\Sym^q\mc{Q}) = \Sym^q V$ and $H^k(\GG,\Sym^q\mc{Q})=0$ for all $k>0$.
\item\label{it:2} For all $a,k\geq 0$, we have that $H^k\bigl(\GG,\Sym^a\mc{Q}\oo\mc{R}\bigr) = 0$.
\item\label{it:3} If $\op{char}(\kk)=p$ satisfies either $p=0$ or $p\geq n-2$, then
\begin{equation}\label{eq:Hr-1-vanishing}
H^{r-1}\bigl(\GG,\bb{S}_{(n-2-r,1-r)}\mc{Q}\bigr) = 0, \mbox{ for all }r\geq 2.
\end{equation}
\end{enumerate}
\end{thm}

\begin{proof} To prove (\ref{eq:Rpsi*}) we first note that via the natural identifications above we have
\[\mc{L}^{\ll} = \mc{L}^{\ll}(V) \cong \psi_2^*(\mc{L}^{\a}(\mc{Q})) \oo_{\mc{O}_{\FF}} \mc{L}^{\b}(\psi_1^*\mc{R}).\]
Using the projection formula for the higher direct images along $\psi_2$, we obtain
\[R^k\psi_{2*}(\mc{L}^{\ll}) = \mc{L}^{\a}(\mc{Q}) \oo_{\mc{O}_Y} R^k\psi_{2*}(\mc{L}^{\b}(\psi_1^*\mc{R})) \overset{(\ref{eq:Rpi*})}{=} \begin{cases}
 \mc{L}^{\a}(\mc{Q}) \oo_{\mc{O}_Y} \bb{S}_{\b}(\psi_1^*\mc{R}) & \mbox{if }k=0; \\
 0 & \mbox{if }k>0.
 \end{cases}
\]
From the compatibility of $\psi_1^*$ with tensor constructions, it follows that $\bb{S}_{\b}(\psi_1^*\mc{R}) = \psi_1^*(\bb{S}_{\b}\mc{R})$. The projection formula along $\psi_1$ then yields
\[
\begin{aligned}
R^k\psi_{1*}(\mc{L}^{\a}(\mc{Q}) \oo_{\mc{O}_Y} \psi_1^*(\bb{S}_{\b}\mc{R}) ) &= R^k\psi_{1*}(\mc{L}^{\a}(\mc{Q})) \oo \bb{S}_{\b} \mc{R} \\
&\overset{(\ref{eq:Rpi*})}{=} \begin{cases}
 \bb{S}_{\a} \mc{Q} \oo \bb{S}_{\b} \mc{R} & \mbox{if }k=0; \\
 0 & \mbox{if }k>0.
 \end{cases}
\end{aligned}
\]
Using the fact that $R\psi_* = R\psi_{1*}\circ R\psi_{2*}$, we obtain (\ref{eq:Rpsi*}). The conclusion (\ref{eq:HF=HG}) follows from (\ref{eq:Rpsi*}) and the Leray spectral sequence.

\vskip 4pt

To establish conclusion (\ref{it:1}) we apply (\ref{eq:HF=HG}) with $\a=(q,0)$ and $\b=(0^{n-2})$, together with Theorem~\ref{thm:kempf-vanishing}. To prove conclusion (\ref{it:2}) we apply (\ref{eq:HF=HG}) with $\a=(a,0)$ and $\b=(1,0^{n-3})$, together with Proposition~\ref{prop:jantzen-cohom}(a) with $i=2$ (noting that $\ll_2=\ll_3-1$). To prove conclusion (\ref{it:3}) we apply (\ref{eq:HF=HG}) with $\a=(n-2-r,1-r)$ and $\b=(0^{n-2})$, together with Lemma~\ref{lem:coh-vanishing} when $r\leq 2n-3$. Finally, when $r>2n-3$, then we have $r-1>2n-4=\dim(\GG)$, therefore the vanishing (\ref{eq:Hr-1-vanishing}) is a consequence of Grothendieck's vanishing theorem, see e.g. \cite[Theorem~III.2.7]{hartshorne}.
\end{proof}

\subsection{Vanishing for Koszul modules}
\label{subsec:vanish-Koszul}

We start by describing a geometric construction of the kernel $\ker(\delta_1)$ of the Koszul differential $\delta_1:V\oo S\lra S$. We let $\GG=\Gr_2(V^{\vee})\subseteq \mathbb  P \bigl(\bigwedge^2 V^{\vee}\bigr)$ as in Section~\ref{subsec:Grassmannians}, and consider
\[\mc{S} := \Sym_{\mc{O}_{\GG}}(\mc{Q}) = \mc{O}_{\GG} \oplus \mc{Q} \oplus \Sym^2\mc{Q} \oplus \cdots\]
viewed as a sheaf of graded $\mc{O}_{\GG}$-algebras on $\GG$. We define Koszul differentials
\[\delta_1^{\mc{Q}} : \mc{Q} \oo_{\mc{O}_{\GG}} \mc{S} \lra \mc{S}\quad\mbox{ and }\quad\delta_2^{\mc{Q}} : \bw^2\mc{Q} \oo_{\mc{O}_{\GG}} \mc{S} \lra \mc{Q} \oo_{\mc{O}_{\GG}} \mc{S}\]
in the usual way. We write $\bw^2\mc{Q} \oo_{\mc{O}_{\GG}} \mc{S}=\mc{S}(1)$, since $\bw^2\mc{Q}=\mc{O}_{\GG}(1)$ is the Pl\"ucker line bundle.

\begin{lem}\label{lem:H0-mcS} We have a commutative diagram
\begin{equation}\label{eq:H0-mcS}
\begin{gathered}
\xymatrix{
& \bw^2 V \oo S \ar[rr]^{\delta_2} \ar[d]_{f_2} & & V \oo S \ar[rrr]^{\delta_1} \ar[d]_{f_1} & & & S \ar[d]_{f_0} \\
0\ar[r] & H^0(\GG,\mc{S}(1)) \ar[rr]^{H^0(\GG,\delta_2^{\mc{Q}})} & & H^0(\GG,\mc{Q} \oo \mc{S}) \ar[rrr]^{H^0(\GG,\delta_1^{\mc{Q}})} & & & H^0(\GG,\mc{S}) \\
}
\end{gathered}
\end{equation}
where the maps $f_1$ and $f_0$ are isomorphisms. Moreover, we have an isomorphism $\ker(\delta_1)\cong H^0(\GG,\mc{S}(1))$.
\end{lem}

\begin{proof}
 The proof is along the lines  of \cite[Lemma~3.4]{koszul-groups}, which is set in  characteristic~$0$ context. The only potential differences in positive characteristic come from the two calculations of sheaf cohomology on Grassmannians in the said proof, namely
\[H^0(\GG,\Sym^q\mc{Q}) = \Sym^q V,\]
which is covered by conclusion (\ref{it:1}) of Theorem~\ref{thm:bott}, and
\[ H^0(\GG,\mc{R}\oo\Sym^q\mc{Q}) = H^1(\GG,\mc{R}\oo\Sym^q\mc{Q}) = 0\mbox{ for all }q\geq 0,\]
which follows from conclusion (\ref{it:2}) of Theorem~\ref{thm:bott}. The rest of the argument is manifestly characteristic independent, so we do not reproduce it here.
\end{proof}

Using Lemma~\ref{lem:H0-mcS}, we obtain the following geometric description of the Koszul module
\begin{equation}\label{eq:WVK-from-cohom}
 W(V,K) = \coker\Bigl\{
\xymatrix{
H^0(\GG,K\oo\mc{S}) \ar[rr]^{H^0(\GG,\eta)} & & H^0\bigl(\GG,\mc{S}(1)\bigr)
}\Bigr\},
\end{equation}
where $\eta:K\oo\mc{S} \lra \mc{S}(1)$ is induced by
\[K\oo\mc{O}_{\GG}\hookrightarrow\bw^2 V \oo\mc{O}_{\GG} \onto \bw^2\mc{Q} = \mc{O}_{\GG}(1).\]
With all these preparations in place, the proof of Theorem~\ref{thm:vanishing-koszul} follows by extending the argument for \cite[Theorem~3.1]{koszul-groups} to positive characteristic, as follows.

\begin{proof}[Proof of Theorem~\ref{thm:vanishing-koszul}]
If $W_q(V,K)= 0$ for $q\geq n-3$ then $\mathcal{R} (V,K)=\{ 0\}$ by (\ref{eq:mainequiv}). For the converse, it suffices to prove that $W_{n-3}(V,K)$: since the module $W(V,K)$ is generated in degree zero, this implies that $W_q(V,K)=0$ for all $q\geq n-3$. The assumption that $\mathcal{R} (V,K)=\{ 0\}$ yields using (\ref{eq:mainequiv}) that $(K,\mc{O}_{\GG}(1))$ is base point free, and hence $K\oo\mc{O}_{\GG}\lra\mc{O}_{\GG}(1)$ is surjective. This gives rise to an exact Koszul complex
\[\mc{K}^{\bullet}: 0\lra \bw^m K\oo \mc{O}_{\GG}(1-m) \lra \cdots \lra \bw^2 K\oo \mc{O}_{\GG}(-1) \lra K \oo \mc{O}_{\GG} \lra \mc{O}_{\GG}(1) \lra 0\]
where $\mc{K}^{-i} = \bw^i K \oo \mc{O}_{\GG}(1-i)$. Tensoring with $\Sym^{n-3}\mc{Q}$ we get a hypercohomology spectral sequence $E_1^{-i,j} = H^j\left(\GG,\mc{K}^{-i}\oo\Sym^{n-3}\mc{Q}\right)  \Longrightarrow 0$, where
 \begin{equation}\label{eq:Eij}
  E_1^{-i,j} = \bw^i K \oo H^j\left(\GG,\bb{S}_{(n-2-i,1-i)}\mc{Q}\right).
 \end{equation}
 We have using (\ref{eq:WVK-from-cohom}) that
 \[W_{n-3}(V,K) = \coker\bigl(E^{-1,0}_1 \lra E^{0,0}\bigr)\]
 and we suppose by contradiction that this map is not surjective. Since $E_{\infty}^{0,0}=0$, there must be some non-zero differential
 \[E^{-r,r-1}_r \lra E_r^{0,0}\mbox{ for }r\geq 2,\mbox{ or }E_r^{0,0} \lra E_r^{r,1-r}\mbox{ for }r\geq 1.\]
 Since $E_1^{r,1-r}=0$ for all $r\geq 1$ it follows that $E_r^{r,1-r}=0$ as well, so the latter case does not occur. To prove that the former case does not occur either and obtain a contradiction, it suffices to check that $E^{-r,r-1}_1=0$ for all $r\geq 2$, which follows from (\ref{eq:Eij}) and from conclusion (\ref{it:3}) of Theorem~\ref{thm:bott}.
\end{proof}


\section{Symmetric, divided and exterior powers, and Hermite reciprocity}
\label{sec:Hermite-rec}

The goal of this section is to study various tensor constructions in arbitrary characteristic, that will be used in our analysis of the syzygies of the tangent developable to a rational normal curve in Section~\ref{sec:repth}. For a more thorough discussion, we refer the reader to \cite{ABW}.
In Section~\ref{subsec:Sym-D} we introduce the symmetric, divided and exterior powers of a vector space $V$, which we denote by $\Sym^{\bullet}V$, $\D^{\bullet}V$, and $\bw^{\bullet}V$ respectively. In Section~\ref{subsec:Veronese} we explain how the natural ambient space of the Veronese embedding is defined in terms of divided powers. We then focus on tensor constructions for a vector space $U$ of dimension $2$, discuss a number of $\sl(U)$-equivariant maps in Section~\ref{subsec:sl2}, and prove a version of Hermite reciprocity in arbitrary characteristic in Section~\ref{subsec:Hermite}: more precisely, we produce an explicit $\sl(U)$-equivariant isomorphism
\[\Sym^d(\D^i U) \lra \bw^i(\Sym^{d+i-1}U)\]
for all $d,i\geq 1$. We work over an arbitrary field $\kk$.


\subsection{Symmetric, divided and exterior powers}
\label{subsec:Sym-D}

Let $V$ denote a finite dimensional vector space over a field $\kk$, and for $d>0$ consider the tensor power $T^d V := V^{\oo d} = V \oo V \oo \cdots \oo V$
with the natural action of the symmetric group $\mf{S}_d$ by permuting the factors. The \defi{divided power $\D^d V$} is defined as the set of symmetric tensors in $T^d V$, that is,
\[ \D^d V := \{ \omega \in T^d V : \s(\omega) = \omega \mbox{ for all }\s \in \mf{S}_d\}.\]
If we consider the subspace of $T^d V$ defined by
\[ \Sigma_d := \op{Span} \{ \s(\omega) - \omega : \s\in \mf{S}_d\mbox{ and }\omega\in T^d V\},\]
then the \defi{symmetric power $\Sym^d V$} is defined as the quotient $\Sym^d V := T^d V / \Sigma_d$.
If $V$ has a basis $(x_1,\ldots,x_n)$, then $\Sym^d V$ identifies with the space of homogeneous polynomials of degree $d$ in $x_1,\ldots,x_n$, and as such it has a basis of monomials $x_1^{a_1}\cdots x_n^{a_n}$, where $a_1+\cdots+a_n=d$. To get a basis for $D^d V$ we first consider the orbits
\[O_{a_1,\ldots,a_n} := \mf{S}_d \cdot x_1^{\oo a_1} \oo x_2^{\oo a_2} \oo \cdots \oo x_n^{\oo a_n}\]
and for $a_1+\cdots+a_n=d$ consider the \defi{divided power monomials}
\[x_1^{(a_1)} \cdots x_n^{(a_n)} := \sum_{\omega\in O_{a_1,\ldots,a_n}} \omega.\]
They form a basis for $D^d(V)$, and in particular we have that $\dim(\Sym^d V) = \dim(\D^d V)$. By composing the inclusion of $\D^d V$ into $T^d V$ with the projection onto $\Sym^d V$ we obtain a natural map
\begin{equation}\label{eq:D-to-Sym}
\D^d V \lra \Sym^d V,\quad x_1^{(a_1)} \cdots x_n^{(a_n)} \mapsto {d \choose a_1,\ldots,a_n} x_1^{a_1}\cdots x_n^{a_n},\mbox{ where }{d \choose a_1,\ldots,a_n} = \frac{d!}{a_1! \cdots a_n!}.
\end{equation}
This map is an isomorphism in characteristic zero, or more generally when the multinomial coefficients are invertible. In general it is neither injective, nor surjective.

\vskip 3pt

We next consider the subspace of $T^d V$ defined by
\[ \Xi_d := \op{Span} \{ v_1\oo \cdots \oo v_n : v_i\in V\mbox{ and }v_i = v_j\mbox{ for some }i\neq j\}\]
and define the \defi{exterior power $\bw^d V$} as the quotient $\bw^d V := T^d V / \Xi_d$.
If we write $v_1\wedge \cdots \wedge v_d$ for the class of $v_1\oo \cdots \oo v_d$ in the quotient, then $\bw^d V$ has a basis consisting of $x_{i_1} \wedge \cdots \wedge x_{i_d}$, where $1\leq i_1 < \cdots < i_d \leq n$. There is a natural inclusion of $\bw^d V$ into $T^d V$ given by
\begin{equation}\label{eq:sub-wedge-d}
 v_1\wedge \cdots \wedge v_d \lra \sum_{\s\in\mf{S}_d} \op{sgn}(\s) \cdot \s(v_1 \oo \cdots \oo v_n).
\end{equation}
This gives a splitting of the quotient map $T^d V \onto \bw^d V$ if and only if $d!$ is invertible in $\kk$.

\vskip 4pt

There is a natural $\mf{S}_n$-invariant perfect pairing
\begin{equation}\label{eq:pairing-TdV}
 T^d(V) \times T^d(V^{\vee}) \lra \kk,
\end{equation}
defined on pure tensors via
\[ \langle v_1\oo\cdots\oo v_d,f_1\oo\cdots\oo f_d\rangle = f_1(v_1) f_2(v_2)\cdots f_d(v_d),\mbox{ where }v_i\in V \mbox{ and } f_j\in V^{\vee}.\]
This induces a perfect pairing between $\Sym^d V$ and $D^d(V^{\vee})$, giving a natural identification
\[ (\Sym^d V)^{\vee} \cong \D^d(V^{\vee}).\]
If $(e_1,\ldots,e_n)$ is the basis of $V^{\vee}$ dual to the basis $(x_1,\ldots,x_n)$ of $V$, then the divided power monomials $e_1^{(a_1)}\cdots e_n^{(a_n)}$ give a basis of $D^d(V^{\vee})$ dual to the monomial basis of $\Sym^d V$.

If we think of $\bw^d V$ as a quotient of $T^d V$, and of $\bw^d V^{\vee}$ as a subspace of $T^d(V^{\vee})$ via (\ref{eq:sub-wedge-d}), the pairing (\ref{eq:pairing-TdV}) induces a perfect pairing $\bw^d V \times \bw^d(V^{\vee}) \lra \kk$, described on decomposable elements via
\[ \langle v_1\wedge\cdots\wedge v_d,f_1\wedge\cdots\wedge f_d\rangle \lra \det\bigl[f_i(v_j)\bigr]_{1\leq i,j\leq d},\mbox{ where }v_i\in V \mbox{ and }  \  f_j\in V^{\vee}.\]
We get a natural isomorphism $\Bigl(\bw^d V\Bigr)^{\vee} \cong \bw^d\bigl(V^{\vee}\bigr)$.
We write $\bw^d V^{\vee}$ for either of the two spaces above. By contrast, we shall be careful in distinguishing between $\Sym^d(V^{\vee})$ and $(\Sym^d V)^{\vee} \cong \D^d(V^{\vee})$!

If $V=U$ is of dimension two, then we get a natural isomorphism $U \cong \bw^2 U \oo U^{\vee}$, which extends to tensor powers to an isomorphism $T^d(U) \cong \bigl(\bw^2 U\bigr)^{\oo d} \oo T^d(U^{\vee})$. Restricting to $\mf{S}_n$-invariant tensors, we obtain an isomorphism $\D^d(U) \cong \bigl(\bw^2 U\bigr)^{\oo d} \oo \D^d(U^{\vee})$, which will be used  in Sections~\ref{sec:resln-R-bar} and~\ref{sec:resln-C}.

\subsection{The Veronese embedding}
\label{subsec:Veronese}



 We fix a vector space $U$ of dimension two and denote by $\PP^1:=\mathbb P\bigl(U^{\vee}\bigr)$
the projective space of $1$-dimensional subspaces of $U^{\vee}$. Then $H^0\bigl(\PP^1,\mc{O}_{\PP^1}(d)\bigr) = \Sym^d U$, which is the vector space of linear forms on $(\Sym^d U)^{\vee} \cong \D^d(U^{\vee})$. Setting $\PP^d:=\mathbb P\bigl(D^d(U^{\vee})\bigr)$, the linear series $|\mc{O}_{\PP^1}(d)|$ gives an embedding
$$\nu_d : \PP^1 \lra \PP^d,  \ \ \ [f]\mapsto  [f \oo \cdots \oo f].$$

\vskip 3pt

If $U$ has a  basis $(x_1,x_2)$, we consider the induced monomial bases for $\Sym^d U$ and $\D^d U$ as in Section~\ref{subsec:Sym-D}. For simplicity, it will be convenient to use a \emph{dehomogenized} notation, where
\begin{equation}\label{eq:conv-x1x2-x}
x_1=x,\ x_2=1,\ x_1^i x_2^{d-i} = x^i \mbox{ and }x_1^{(i)} x_2^{(d-i)} = x^{(i)}.
\end{equation}
If we write $(1,y)$ for the basis of $U^{\vee}$ dual to $(1,x)$ and use the analogous conventions to (\ref{eq:conv-x1x2-x}), then the Veronese map $\nu_d:\PP^1\lra\PP^d$ is expressed as
\[[a+by] \lra [(a+by)\oo\cdots\oo(a+by)] = [a^d\cdot y^{(0)} + a^{d-1}b\cdot y^{(1)} + a^{d-2}b^2\cdot y^{(2)} + \cdots + b^d \cdot y^{(d)}],\]
that is, relative to the chosen bases, $\nu_d$ takes the familiar form $[a:b] \lra [a^d:a^{d-1}b:a^{d-2}b^2:\cdots:b^d]$.


\subsection{Some $\sl_2$-equivariant maps}
\label{subsec:sl2}

We let $U$ as in the previous section, with $\kk$-basis $(1,x)$ and corresponding bases $(1,x,\ldots,x^d)$ and $(x^{(0)},x^{(1)},x^{(2)},\ldots,x^{(d)})$ for $\Sym^d U$ and $\D^d U$ respectively. We let $\sl_2 := \sl(U)$ denote the Lie algebra of $\kk$-endomorphisms of $U$ with trace $0$. We write $\Lop$ (respectively $\Rop$) for the \defi{lowering} (respectively \defi{raising}) operators in $\sl_2$, whose action on $\Sym^d U$ and $\D^d U$ is given by
\[\Lop\cdot x^i = i\cdot x^{i-1},\quad \Lop\cdot x^{(i)} = (d-i+1)\cdot x^{(i-1)},\quad \Rop\cdot x^i = (d-i)\cdot x^{i+1},\quad \Rop\cdot x^{(i)} = (i+1)\cdot x^{(i+1)}.\]
Relative to the conventions (\ref{eq:conv-x1x2-x}), one can think of $\Lop$ as the operator $x_2\cdot\frac{\pd}{\pd x_1}$, and of $\Rop$ as $x_1\cdot\frac{\pd}{\pd x_2}$. The natural map (\ref{eq:D-to-Sym}) is described on the basis elements by
\begin{equation}\label{eq:(i)-i}
 \D^d U \lra \Sym^d U,\qquad x^{(i)} \mapsto {d \choose i}\cdot x^i,
\end{equation}
and can be checked to be $\sl_2$-equivariant. By abuse of notation, whenever we refer to $x^{(i)}$ as an element of $\Sym^d U$ we will interpret it as ${d \choose i}\cdot x^i$.

\vskip 4pt

For $a,b\geq 0$, we define two $\mf{sl}_2$-equivariant maps, namely the \defi{multiplication} map
\begin{equation}\label{eq:multiplication}
 \mu=\mu_U: \Sym^a U \oo \Sym^b U \lra \Sym^{a+b}U,\quad x^i \oo x^j \lra x^{i+j}
\end{equation}
and the \defi{co-multiplication} map
\begin{equation}\label{eq:co-multiplication}
\Delta=\Delta_U: \D^{a+b}U \lra \D^a U \oo \D^b U,\quad x^{(t)} \lra \sum_{\substack{i+j=t \\ i\leq a,\ j\leq b}} x^{(i)}\oo x^{(j)}.
\end{equation}

We may abuse language and refer to $\Delta$ as the \defi{dual} of $\mu$: the correct interpretation of this statement is that via the isomorphism $\D^d(U^{\vee})\cong (\Sym^d U)^{\vee}$ one has that $\mu_U$ is dual to $\Delta_{U^{\vee}}$, and $\mu_{U^{\vee}}$ is dual to $\Delta_U$. We shall need in Proposition~\ref{prop:resolution-Rbar} the map $\Delta^{(2)}:\D^{a+2}U \lra \D^a U \oo \Sym^2 U$ obtained as the composition:
\[\Delta^{(2)}:\D^{a+2}U \overset{(\ref{eq:co-multiplication})}{\lra} \D^a U \oo \D^2 U \overset{(\ref{eq:(i)-i})}{\lra} \D^a U \oo \Sym^2 U.\]
It is important to observe that the following diagram is commutative:
\begin{equation}\label{eq:comm-diag-co-multiplication}
\begin{aligned}
\xymatrix{
 \D^{a+2}U \ar[rr]^{\Delta^{(2)}} \ar[d]_{\Delta} & & \D^a U \oo \Sym^2 U \\
 \D^{a+1}U \oo U \ar[rr]^{\Delta\oo\op{id}_U} & & \D^{a}U \oo U \oo U \ar[u]_{\op{id}_{\D^aU} \oo \mu}
}
\end{aligned}
\end{equation}

\begin{lemma}\label{lem:mu1-delta1}
 Let $a\geq 1$. There exist $\sl_2$-equivariant maps
 \[\mu_1 : \Sym^a U \oo \Sym^a U \lra \Sym^{2a-2}U\mbox{ and } \Delta_1: \D^{2a-2}U \lra \D^a U \oo \D^a U\]
defined by  \ $\mu_1(x^i \oo x^j) = (i-j) \cdot x^{i+j-1}$ and respectively

 \begin{equation}\label{eq:delta1}
 \Delta_1(x^{(t)}) = \sum_{i+j = t+1} (i-j)\cdot x^{(i)}\oo x^{(j)}
 \end{equation}
 Moreover, if $\op{char}(\kk)\neq 2$ then $\mu_1$ is surjective and $\Delta_1$ is injective.
\end{lemma}

\begin{proof}
 Note that $\Delta_1$ is the dual of $\mu_1$ (in the same way that $\Delta$ is dual to $\mu$), so it suffices to verify that $\mu_1$ is well-defined, $\sl_2$-equivariant and surjective when $\op{char}(\kk)\neq 2$. Note first that if $(i,j)\neq (0,0),(a,a)$, then $0\leq i+j-1\leq 2a-2$, so $\mu_1(x^i \oo x^j) \in \Sym^{2a-2}U$. If $(i,j) = (0,0)$ or $(i,j) = (a,a)$, then $\mu_1(x^i \oo x^j) = 0$ is still well-defined. To check equivariance, it suffices to verify that $\mu_1$ commutes with the operators $\Lop$ and $\Rop$. The condition $\mu_1(\Lop(x^i \oo x^j)) = \Lop(\mu_1(x^i \oo x^j))$ is equivalent to
 $$i\cdot [(i-1) - j] + j\cdot[ i -  (j-1)] = (i+j-1)\cdot(i-j),$$
 which follows by inspection. Similarly, $\mu_1(\Rop(x^i \oo x^j)) = \Rop(\mu_1(x^i \oo x^j))$ follows from a direct calculation.
Consider any $0\leq t\leq 2a-2$. If $t=2r$ then $\mu_1(x^{r+1}\oo x^r) = x^t$, so $x^t\in\Im(\mu_1)$. If $t=2r+1$ then $\mu_1(x^{r+2}\oo x^r) = 2x^t$, so $x^t\in\Im(\mu_1)$ when $2$ is invertible in $\kk$. This proves the surjectivity of $\mu_1$.
\end{proof}

From the definition of $\mu_1$, observe that $\mu_1(x^i\oo x^i) = 0$ for all $i$, and $\mu_1(x^i \oo x^j + x^j \oo x^i) = 0$ for $i\neq j$. It follows that $\mu_1(v\oo v)=0$ for all $v\in\Sym^a U$, so $\mu_1$ factors through $\bw^2\Sym^a U$. We get a map
\begin{equation}\label{eq:mu1-to-wedge2}
\mu_1:\bigwedge^2 \Sym^a U\lra \Sym^{2a-2}U,
\end{equation}
which, as explained in the Introduction, can be thought of as the Gaussian--Wahl map of $\mc{O}_{\PP^1}(a)$ (see (\ref{eq:Wahl-map})). Dually, the map $\Delta_1$ in Lemma~\ref{lem:mu1-delta1} can be viewed as a map
\begin{equation}\label{eq:Del1-to-wedge2}
\Delta_1:D^{2a-2}U\lra \bigwedge^2 D^a U.
\end{equation}

\vskip 3pt

\subsection{A Hermite type identification in arbitrary characteristic.}
\label{subsec:Hermite}

Hermite reciprocity \cite{hermite} over a field $\kk$ of characteristic $0$ asserts the existence of an $\sl_2$-isomorphism $\Sym^d(\Sym^i U)\cong \Sym^i(\Sym^d U)$, see \cite[Exercise ~11.34]{FH}. Under the same assumption, one has the isomorphism $\Sym^d(\Sym^i U)\cong \bw^d(\Sym^{d+i-1}U)$, or equivalently $\Sym^d(\Sym^i U)\cong \bw^i(\Sym^{d+i-1}U)$, see \cite[Exercise ~11.35]{FH}. These isomorphisms are usually proved non-constructively by identifying the characters of the two representations, and they no longer hold in positive characteristic. The goal of this section is to describe an explicit $\sl_2$-isomorphism
\begin{equation}\label{eq:isom-psi_d}
\psi_d^i: \Sym^d(\D^i U) \lra \bw^i(\Sym^{d+i-1}U)
\end{equation}
in arbitrary characteristic.

\begin{rmk} Since $\bw^2 U$ is isomorphic to the trivial $\sl_2$-representation, we get that $U\cong U^{\vee}$. More generally, if $V$ is an $\sl_2$-representation with $N=\dim(V)$ then we have isomorphisms $\bw^r V \cong \bw^{N-r}V^{\vee}$.
Combining this with $(\Sym^d V)^{\vee} \cong \D^d(V^{\vee})$, we get a chain of $\sl_2$-isomorphisms
\[\Sym^d(\D^i U) \cong \bw^i(\Sym^{d+i-1}U) \cong \bw^d(\Sym^{d+i-1}U)^{\vee} \cong \bw^d(\D^{d+i-1}(U^{\vee})) \cong \bw^d(\D^{d+i-1}U),\]
which is an arbitrary characteristic analogue of \cite[Exercise~11.35]{FH}. Moreover,
\[\Sym^d(\D^i U) \cong \bw^d(\Sym^{d+i-1}U)^{\vee}\cong (\Sym^i(\D^d U))^{\vee}\cong \D^i(\Sym^d(U^{\vee})) \cong \D^i(\Sym^d U)\]
is an arbitrary characteristic version of Hermite reciprocity \cite[Exercise ~11.34]{FH}.
\end{rmk}

To construct (\ref{eq:isom-psi_d}) we identify both sides with an appropriate subspace of the ring of symmetric polynomials, and construct $\psi_d^i$ from the change of basis between elementary symmetric and Schur polynomials.

\vskip 4pt

For $i>0$ we let $\mc{P}_i$ be the collection of partitions with at most $i$ parts: an element $\ll\in\mc{P}_i$ is written as $\ll=(\ll_1,\ll_2,\ldots,\ll_i)$, with $\ll_1\geq\ll_2\geq\ldots\geq\ll_i\geq 0$. We let $\mc{P}'_i$ denote the collection of partitions with parts of size at most $i$, write $\ll'$ for the conjugate of a partition $\ll$, and note that $\ll\in\mc{P}_i$ if and only if $\ll'\in\mc{P}'_i$. We let $\LL_i \subseteq \kk[z_1,\ldots,z_i]$ denote the ring of symmetric polynomials in $z_1,\ldots,z_i$, and consider the following two well-studied bases of $\LL_i$ (see \cite[Chapter~I]{macdonald}), indexed by~$\mc{P}_i$ and $\mc{P}'_i$:
\begin{itemize}
 \item For every $\ll\in\mc{P}_i$, the \defi{Schur polynomial} $s_{\ll}$ is defined via
 \[s_{\ll}(z_1, \ldots, z_i) = \frac{\det(z_k^{\ll_{\ell} + i - k})_{1\leq k,\ell\leq i}}{\det(z_k^{i - k})_{1\leq k,\ell\leq i}}.\]
 The collection $\{s_{\ll}:\ll\in\mc{P}_i\}$ is a basis for $\LL_i$, which we call the \defi{Schur basis}.
 \item The \defi{elementary symmetric polynomials} $e_0=1,e_1,\ldots, e_i$ are defined via the equality
 \[e_0 + e_1\cdot T + e_2\cdot T^2 + \cdots + e_i\cdot T^i = (1+z_1\cdot T)\cdots(1+z_i\cdot T),\]
 where $T$ is an auxiliary variable. For each $\mu\in\mc{P}'_i$ let $e_{\mu}:= e_{\mu_1}\cdot e_{\mu_2}\cdots$.
The collection $\{e_{\mu}:\mu\in\mc{P}'_i\}$ is also a basis for $\LL_i$, which we call the \defi{elementary symmetric basis}.
\end{itemize}

For each $j=0,\ldots, i$, one has $e_j = s_{(1^j)}$, where $(1^j) = (1,1,\ldots,1)$. More generally, for $d\geq 1$ we consider
the following collections of partitions
\[\mc{P}_i(d):= \bigl\{\ll\in\mc{P}_i : \ll_1\leq d\bigr\},\qquad \mc{P}'_i(d) := \bigl\{\mu\in\mc{P}'_i : \mu_{d+1}=0\bigr\}.\]
The correspondence $\ll\llra\ll'$ induces a bijection between $\mc{P}_i(d)$ and $\mc{P}'_i(d)$. Moreover, we have that $\{s_{\ll}:\ll\in\mc{P}_i(d)\}$ spans the same subspace of $\LL_i$ as $\{e_{\mu}:\mu\in\mc{P}'_i(d)\}$. We denote this subspace by $\LL_i(d)$, and get in this way a filtration of $\LL_i$ satisfying $\LL_i(d)\cdot\LL_i(d')=\LL_i(d+d')$, for all $d,d'\geq 0$.

\vskip 4pt

The multiplication map on $\LL_i$ is easy to describe with respect to the elementary symmetric basis, but it is more subtle (based on the Littlewood--Richardson rule) with respect to the Schur basis. For our purposes it is sufficient to consider products of the form $s_{\ll}\cdot s_{(1^j)} = s_{\ll} \cdot e_j$, where the multiplication is described by Pieri's rule. To describe it we first make some conventions. We write $[i]:= \{1,\ldots,i\}$ and let ${[i] \choose k}$ be the collection of $k$-element subsets of $[i]$.
\begin{itemize}
 \item We extend the definition of $s_{\ll}$ for all $\ll\in\bb{Z}^i_{\geq 0}$ by setting $s_{\ll}=0$ when $\ll\notin\mc{P}_i$, that is, when there exists some index $j$ for which $\ll_j<\ll_{j+1}$.
 \item If $I\subseteq [i]$, we write $(1^I)$ for the element $\ll\in\bb{Z}^i_{\geq 0}$ with $\ll_j = 1$ for $j\in I$ and $\ll_j = 0$ for $j\notin I$. When $I=[k]$, for $k\leq i$ we have $(1^I) = (1^k)$.
\end{itemize}

\vskip 4pt

With these conventions, Pieri's rule can be stated as follows:
\begin{equation}\label{eq:Pieri}
s_{\ll} \cdot e_j = \sum_{I \in {[i] \choose j}} s_{\ll + (1^I)} = \sum_{I \in {[i] \choose j}, \  \ll + (1^I) \in \mc{P}_i} s_{\ll + (1^I)}  .
\end{equation}

Even though $\sl_2$ does not act apriori on $\LL_i(d)$, we construct an $\sl_2$-equivariant vector space isomorphism $\psi_d$ as in (\ref{eq:isom-psi_d}) by identifying both sides with $\LL_i(d)$. We define the \defi{elementary symmetric basis} of $\Sym^d(\D^i U)$ by associating to each element $\mu\in \mc{P}'_i(d)$ the basis element $e_{\mu}(x) \in \Sym^d(\D^i U)$ defined by
\begin{equation}\label{eq:def-emu}
 e_{\mu}(x) = x^{(\mu_1)} \cdot x^{(\mu_2)}  \cdots x^{(\mu_d)}.
\end{equation}
This choice of basis induces a vector space isomorphism
\[\eps_d^i: \Sym^d(\D^i U) \lra \LL_i(d),\qquad \eps_d^i(e_{\mu}(x)) = e_{\mu}, \mbox{ for all }\mu\in\mc{P}'_i(d).\]
Thinking of $\Sym(\D^i U)$ as a polynomial ring on $x^{(0)},x^{(1)},\ldots,x^{(i)}$, we can define a ring homomorphism $\Sym(\D^i U) \lra \LL_i$ by sending $x^{(0)} \mapsto 1$, $x^{(j)} \mapsto e_j$, for $j=1,\ldots,i$. The map $\eps_d^i$ is then obtained by restricting this homomorphism to $\Sym^d(\D^i U)$. It follows that
\begin{equation}\label{eq:comp-eps}
 \eps_{d+1}^i( m \cdot x^{(j)} ) = \eps_d^i(m) \cdot \eps_1^i(x^{(j)})\mbox{ for all }m\in\Sym^d(\D^i U)\mbox{ and }j=0,\ldots,i.
\end{equation}

\vskip 4pt

We next define the \defi{Schur basis} of $\bw^i(\Sym^{d+i-1}U)$ by associating to an element $\ll\in \mc{P}_i(d)$ the basis element $s_{\ll}(x) \in \bw^i(\Sym^{d+i-1}U)$ defined by
\begin{equation}\label{eq:def-sll}
 s_{\ll}(x):= x^{\ll_1 + i - 1} \wedge x^{\ll_2 + i - 2} \wedge \cdots \wedge x^{\ll_i}.
\end{equation}
This choice of basis induces a vector space isomorphism
\[\s_d^i: \LL_i(d) \lra  \bw^i(\Sym^{d+i-1}U),\qquad \s_d^i(s_{\ll}) = s_{\ll}(x), \mbox{ for all }\ll\in \mc{P}_i(d).\]
The maps $\s_d^i$ satisfy a compatibility analogous to (\ref{eq:comp-eps}) as explained next. For all $\kk$-vector spaces $T$ and $T'$ we have a natural inclusion, see \cite[Theorem~III.2.4]{ABW} and \cite[p.96]{weyman}.
\[\Bigl(\bw^i T\Bigr) \oo \D^i (T') \lra \bw^i\left(T\oo T'\right).\]
Applying this when $T = \Sym^{d+i-1}U$ and $T' = U$, we obtain an $\sl_2$-equivariant inclusion followed by an $\sl_2$-equivariant map induced by $\mu:\Sym^{d+i-1}U\oo U \lra \Sym^{d+i}U$:
\[
\xymatrix{
\bigl(\bw^i \Sym^{d+i-1}U\bigr) \oo \D^i U\ar[r] \ar@/_2pc/[rr]_{\nu_d^i} & \bw^i\bigl(\Sym^{d+i-1}U\oo U\bigr) \ar[r] & \bw^i\bigl(\Sym^{d+i}U\bigr) \\
}
\]
The map $\nu_d^i$ can be expressed concretely as follows: for every $\ll\in\mc{P}_i(d)$ and every $j=0,\ldots,i$, we have
\begin{equation}\label{eq:nud-formula}
 \nu_d^i( s_{\ll}(x) \oo x^{(j)}) = \sum_{I\in {[i] \choose j}} s_{\ll + (1^I)}(x),
\end{equation}
where in order for the right hand side to make sense, we have to extend the definition (\ref{eq:def-sll}) for arbitrary $\ll\in\bb{Z}^i_{\geq 0}$. Notice however that if $\mu = \ll + (1^I)$ appears in (\ref{eq:nud-formula}) and $\mu\in \bb{Z}^i_{\geq 0} \setminus \mc{P}_i(d+1)$ then there exists some index $k\in [i-1]$ with $\mu_{k+1} = \mu_k + 1$, and therefore $\mu_k + i - k = \mu_{k+1} + i - (k+1)$, which yields
\[s_{\mu}(x) = x^{\mu_1 + i - 1} \wedge \cdots \wedge x^{\mu_k + i - k} \wedge x^{\mu_{k+1} + i - (k+1)} \wedge \cdots \wedge x^{\mu_i} = 0.\]
We conclude from (\ref{eq:Pieri}) and (\ref{eq:nud-formula}) that
\begin{equation}\label{eq:comp-sig}
 \s_{d+1}^i( f \cdot e_j) = \nu_d^i( \s_d(f) \oo x^{(j)} ), \mbox{ for all }f\in\LL_i(d)\mbox{ and }j=0,\ldots,i.
\end{equation}

\begin{lemma}\label{lem:psi_d-equivariant}
 For each $d\geq 0$, the vector space isomorphism
\[\psi_d^i: \Sym^d(\D^i U) \lra \bw^i(\Sym^{d+i-1}U)\]
defined by $\psi_d^i = \s_d^i \circ \eps_d^i$ is $\sl_2$-equivariant.
\end{lemma}

\begin{proof}
 We prove this by induction on $d$. In the case $d=0$ we have
\[\psi_0^i:\kk \lra \bw^i(\Sym^{i-1}U),\quad 1\mapsto x^{i-1}\wedge x^{i-2}\wedge\cdots\wedge 1,\]
which is an $\sl_2$-equivariant isomorphism. Assume now $d>0$ and consider the commutative diagram
 \begin{equation}\label{eq:sym-to-wedge}
 \begin{gathered}
 \xymatrix{
 \Sym^d(\D^i U) \oo \D^i U \ar[r]^(.6){\eps_d^i \oo \eps_1^i} \ar[d] &  \LL_i(d) \oo \LL_i(1) \ar[r]^(.4){\s_d^i \oo (\eps_1^i)^{-1}} \ar[d] &  \bw^i(\Sym^{d+i-1}U) \oo \D^i U \ar[d]^{\nu_d^i} \\
 \Sym^{d+1}(\D^i U) \ar[r]^{\eps_{d+1}^i} & \LL_i(d+1)  \ar[r]^{\s_{d+1}^i} & \bw^i(\Sym^{d+i}U) \\
 }
 \end{gathered}
 \end{equation}
where the left square commutes by (\ref{eq:comp-eps}) while the right square commutes by (\ref{eq:comp-sig}). The left vertical map and $\nu_d^i$ are $\sl_2$-equivariant. The composition  on the top line is $(\s_d^i \oo (\eps_1^i)^{-1}) \circ (\eps_d^i \oo \eps_1^i) = \psi_d^i \oo \mbox{id}$,
which is $\sl_2$-equivariant since $\psi_d^i$ is $\sl_2$-equivariant by induction. It follows that $\psi_{d+1}^i$, which is the composition of the maps on the bottom line of the diagram, is $\sl_2$-equivariant, concluding our induction.
\end{proof}

For later use, we record that the Koszul differential gives rise to an $\sl_2$-equivariant inclusion
\[ k_i^d : \bw^i  \Sym^{d}U \lra \bw^{i-1} \Sym^{d}U \oo  \Sym^{d}U, \mbox{ given by}\]
\begin{equation}\label{eq:dkos-slam}
 k_i^d(s_{\ll}(x)) = \sum_{j=1}^i (-1)^{j-1} s_{\ll^{\hat{j}}}(x) \oo x^{\ll_j + i - j}, \mbox{ for }\ll\in\mc{P}_i(d-i+1), \ \mbox{ where}
\end{equation}
\begin{equation}\label{eq:lam-hat-j}
\ll^{\hat{j}} = (\ll_1+1,\ldots,\ll_{j-1}+1,\ll_{j+1},\ldots,\ll_i) \in \mc{P}_{i-1}(d-i+2).
\end{equation}

\begin{remark}\label{hermite_geometric} When $\mathrm{char}(\kk)=0$, the Hermite identification $\Sym^d(\Sym^i U)\cong \bw^i(\Sym^{d+i-1}U)$ has a nice interpretation in terms of a simple calculation on the Hilbert scheme of points on  $\PP^1:=\mathbb P\bigl(U^{\vee}\bigr)$.
We thank Ein and Lazarsfeld for drawing our attention to this fact, which links our algebraic approach to Voisin's geometric strategy \cite{V02, V05} of interpreting syzygies as sections of tautological sheaves on Hilbert schemes. In the interest of the geometrically minded reader we recall this fact. We fix integers $i$ and $d$ and identify $\Sym^i(\PP^1)\cong \PP^i=\mathbb P \bigl((\Sym^i U)^{\vee}\bigr).$ Following \cite[5.1]{AN10}, we denote by $\Xi\subseteq \Sym^i(\PP^1)\times \PP^1$ the incidence correspondence endowed with the projections $p:\Xi\rightarrow \PP^1$ and $q:\Xi\rightarrow \Sym^i(\PP^1)$ respectively. 
We set $L:=\mathcal{O}_{\PP^1}(d+i-1)$ and consider the rank $i$ vector bundle
$L^{[i]}:=q_*p^*(L)$. On the one hand $\mbox{det}(L^{[i]})=\OO_{\PP^i}(d)$, see for instance \cite[Lemma 5.8]{AN10}, therefore 
\begin{equation}\label{hilb1}
H^0\bigl(\Sym^i(\PP^1), \mathrm{det}\ L^{[i]}\bigr)\cong H^0\bigl(\PP^i,\OO_{\PP^i}(d)\bigr)=\Sym^d(\Sym^i U).
\end{equation}
On the other hand, the evaluation map $H^0(\PP^1,L)\otimes \mathcal{\OO}_{\Sym^i(\PP^1)}\rightarrow L^{[i]}$ induces an isomorphism at the level of global sections, see for instance \cite[Lemma 5.2]{AN10}
\begin{equation}\label{hilb2}
H^0\bigl(\Sym^i(\PP^1), \mathrm{det}\ L^{[i]}\bigr)\cong \bigwedge^i H^0(\PP^1,L)=\bigwedge^i (\Sym^{d+i-1} U).
\end{equation}
Even in characteristic zero, it is however not immediately clear that the identification obtained by comparing the isomorphisms 
(\ref{hilb1}) and (\ref{hilb2}) matches the explicit $\mathfrak{sl}_2$-isomorphism $\psi_d^i$ constructed in (\ref{eq:isom-psi_d}).
\end{remark}  

\section{Syzygies and the Kempf--Weyman geometric technique}
\label{sec:KW}

The goal of this section is to summarize a number of standard results that relate to syzygies, minimal resolutions, and Koszul cohomology groups, and to establish some relevant notation that will be used in the rest of the article. In Section~\ref{subsec:syz-koszul} we discuss a relationship between Koszul cohomology groups and kernel bundles, which will be important in Section~\ref{sec:moduli-Green}. In Section~\ref{subsec:minimization} we describe the minimization of the mapping cone for certain morphisms of complexes, that will be used in the proof of Corollary~\ref{cor:C_i=ker-pi+1} and in that of Theorem~\ref{cor=kw}. Finally, in Section~\ref{subsec:Kempf-Weyman} we recall the basic aspects of the Kempf--Weyman geometric technique for constructing syzygies, that will be featured several times in Sections~\ref{sec:resln-R-bar} and \ref{sec:resln-C}.

\subsection{Syzygies and Koszul cohomology}
\label{subsec:syz-koszul}

We begin by recalling a few basic facts and notation regarding Koszul cohomology groups, and we refer the reader to \cite{G84} and \cite{AN10} for more details. Let $X$ be a projective variety over a field $\kk$, let $L$ be a very ample line bundle on $X$ and consider the associated embedding  $\varphi_L:X\lra \mathbb P\bigl(H^0(X,L)^{\vee}\bigr)\cong \mathbb P^r$. For a sheaf $\mc{F}$ on $X$, we consider the module
$$\Gamma_X(\mc{F}, L):=\bigoplus_{n\geq 0} H^0(X,\mc{F}\otimes L^{n}),$$
which can be regarded as a graded module over the polynomial algebra $S:=\Sym  H^0(X,L)$. For $i,j\geq 0$, the \defi{Koszul cohomology group of $i$-th order syzygies of weight $j$} is defined by
$$K_{i,j}(X,\mc{F}, L):=\mbox{Tor}^S_i\bigl(\Gamma_X(\mc{F}, L),\kk\bigr)_{i+j},$$
or equivalently as the middle homology of the $3$-term complex
\begin{equation}\label{eq:Kij-3term}
 \bw^{i+1} H^0(X,L) \oo \Gamma_X(\mc{F}, L)_{j-1} \lra \bw^i H^0(X,L) \oo \Gamma_X(\mc{F}, L)_j \lra \bw^{i-1} H^0(X,L) \oo \Gamma_X(\mc{F}, L)_{j+1}.
\end{equation}
When $\mc{F}=\OO_X$, we write as usual $\Gamma_X(L):=\Gamma_X(\OO_X,L)$ and $K_{i,j}(X,L):=K_{i,j}(X,\OO_X, L)$. The Koszul cohomology groups give rise to the \defi{minimal graded free $S$-resolution} of $\Gamma_X(\mc{F}, L)$: this is a complex $F_{\bullet}$ of graded $S$-modules, unique up to isomorphism, producing an exact sequence
\[ \cdots \lra F_{i+1} \lra F_i \lra \cdots \lra F_1 \lra F_0 \lra \Gamma_X(\mc{F},L) \lra 0,\]
and satisfying $F_i=\bigoplus_{j\geq 0} S(-i-j)\otimes K_{i,j}(X,\mc{F},L)$ for all $i\geq 0$. Here $S(-d)$ denotes a free graded $S$-module of rank one, generated in degree $d$. We define the \defi{kernel bundle} $M_L$ associated to the pair $(X,L)$, as the kernel of the evaluation map on sections, that is, via the exact sequence
\begin{equation}\label{lazb}
0\longrightarrow M_{L}\longrightarrow H^0(X,L) \otimes \mathcal{O}_X \longrightarrow L\longrightarrow 0,
\end{equation}
where the surjectivity on the right follows from the global generation of $L$. Applying the construction (\ref{lazb}) to the pair $(\bb{P}^r,\OO_{\bb{P}^r}(1))$, we get a short exact sequence
\begin{equation}\label{lazb-on-Pr}
0\longrightarrow M_{\OO_{\bb{P}^r}(1)}\longrightarrow H^0(\bb{P}^r,\OO_{\bb{P}^r}(1)) \otimes \OO_{\bb{P}^r} \longrightarrow \OO_{\bb{P}^r}(1) \longrightarrow 0.
\end{equation}
In fact (\ref{lazb}) is the pull-back of (\ref{lazb-on-Pr}) along $\varphi_L$. Moreover, (\ref{lazb-on-Pr}) can be identified with the twist by $\OO_{\bb{P}^r}(1)$ of the Euler sequence on $\bb{P}^r$. The proof of the following result appears
in \cite[Proposition 2.8]{AN10}.

\vskip 4pt

\begin{lemma}
 Let $L$ be a very ample line bundle on $X$. Then for $i\geq 0$ we have that
\begin{equation}\label{koszint}
K_{i+1,1}(X,L) \cong \ker\Bigl\{H^0\bigl(\bb{P}^r,\bw^i M_{\OO_{\bb{P}^r}(1)}(2)\bigr) \overset{\a}{\lra} H^0\bigl(X,\bw^i M_L \oo L^{2}\bigr)\Bigr\},
\end{equation}
where $\a$ is the natural map induced by the identifications $\varphi_L^*(\OO_{\bb{P}^r}(1))=L$ and $\varphi_L^*(M_{\OO_{\bb{P}^r}(1)})=M_L$.
\end{lemma}

\subsection{Minimizations and mapping cones}
\label{subsec:minimization}

Let $S$ be a graded polynomial ring over a field $\kk$. If $A$ is any graded $S$-module, we write $\ol{A} = A \oo_S \kk$ for the reduction of $A$ modulo the maximal homogeneous ideal. If $\phi:A\lra B$ is a map of $S$-modules then $\ol{\phi}:\ol{A}\lra\ol{B}$ denotes the induced map $\phi\,\oo_S \op{id}_\kk$. Similarly, if $A_{\bullet}$ is any complex of graded $S$-modules then $\ol{A}_{\bullet} = A_{\bullet} \oo_S \kk$.

Recall that any complex $F_\bullet$ of finitely-generated, graded, free $S$-modules is a direct sum of a minimal complex  $M_{\bullet}$ and an exact complex. The complex $M_{\bullet}$ in question is unique up to isomorphism and is called \emph{the minimization of $F_\bullet$}. Explicitly, the terms of $M_{\bullet}$ are computed by
\begin{equation}\label{eq:minimization-M}
M_i = H_i(\ol{F}_\bullet)\otimes_\kk S = \bigoplus_{j\in\bb{Z}} H_i(\ol{F}_\bullet)_j\otimes_\kk S,
\end{equation}
where $H_i(\ol{F}_\bullet)_j$ is a graded vector space concentrated in degree $j$, so that $H_i(\ol{F}_\bullet)_j\otimes_\kk S$ is a direct sum of copies of $S(-j)$. On several occasions in Section~\ref{sec:repth}, we shall need to compute the minimization of a mapping cone, and in each case the following result will apply.

\begin{lem}
\label{lem:minimization}
Let $F_\bullet$ and $G_\bullet$ be two minimal complexes of finitely-generated, free, graded $S$-modules, and let $f_\bullet:F_\bullet\to G_\bullet$ be a morphism of complexes such that for all $i$, all the entries of the matrix representing $f_i$ are constant. The terms of the minimization $M_{\bullet}$ of the mapping cone of $f_\bullet$ are $M_i=\ker(f_{i-1})\oplus \coker(f_i)$.
\end{lem}

\proof
We write $d_i^F$ (resp. $d_i^G$) for the differentials in $F_{\bullet}$ (resp. $G_{\bullet}$), and let $C(f)_{\bullet}$ denote the mapping cone of $f_{\bullet}$. We have that $C(f)_i=F_{i-1}\oplus G_i$, and the $i$-th differential in $C(f)_{\bullet}$ is given by the matrix
\[
\pd_i=\left(
\begin{array}{cc}
d_{i-1}^{F} &0\\
 f_{i-1} & - d_i^{G}
\end{array}
\right).
\]
Since $F_\bullet$ and $G_\bullet$ are minimal, the reduction of $C(f)_{\bullet}$ modulo the maximal homogeneous ideal of $S$ yields a complex $\ol{C(f)}_{\bullet}$, whose terms are $\ol{C(f)}_i=\ol{F}_{i-1}\oplus \ol{G}_i$ and with differentials
\[
\ol{\pd}_i=\left(
\begin{array}{cc}
0 & 0\\
\ol{f}_{i-1} & 0
\end{array}
\right).
\]
It follows that $H_i(\ol{C(f)}_\bullet)=\ker(\ol{f}_{i-1})\oplus\coker(\ol{f}_i)$ for all $i$.

Our assumption on the maps $f_i$ implies that $f_i = \ol{f}_i \oo_{\kk} S$ is obtained by base change from a map of graded vector spaces. Since this base change is flat, we conclude that
\[ \ker(f_i) = \ker(\ol{f}_i) \oo_{\kk} S\mbox{ and }\coker(f_i) = \coker(\ol{f}_i) \oo_{\kk} S\mbox{ for all }i.\]
The desired conclusion about the terms $M_i$ follows now from the description (\ref{eq:minimization-M}) of the minimization:
\[M_i = H_i(\ol{C(f)}_\bullet) \oo_{\kk} S = \ker(\ol{f}_{i-1})\oo_{\kk} S \oplus\coker(\ol{f}_i)\oo_{\kk} S =  \ker(f_{i-1})\oplus \coker(f_i). \qedhere\]
\endproof

\subsection{The geometric technique for constructing syzygies}
\label{subsec:Kempf-Weyman}

Let $Y$ be a projective variety over a field~$\kk$, and let $V$ be a finite dimensional $\kk$-vector space. Suppose  we have a short exact sequence
\begin{equation}\label{eq:ses-xi-V-eta}
 0 \lra \xi \lra V \oo \OO_{Y} \lra \eta \lra 0,
\end{equation}
of vector bundles over $Y$ and set $r:=\mbox{rk}(\xi)$. We view $S := \Sym(V)$ as the coordinate ring of the affine space $V^{\vee}$. Let $X := \ul{\op{Spec}}_{Y}(\OO_{Y} \oo S) = V^{\vee} \times Y$
denote the trivial bundle on $Y$ with fiber $V^{\vee}$ and write $\pi:X\lra V^{\vee}$ and $p:X\lra Y$ for the projections. We consider the sheaf $\mc{S}: = \Sym_{\OO_{Y}}(\eta)$, and let
\[Z := \ul{\op{Spec}}_{Y}(\mc{S}),\]
be the total space of the  vector bundle associated with $\eta$. The surjection in (\ref{eq:ses-xi-V-eta}) makes  $Z$  a subbundle of the trivial bundle $X$ and we get the following  commutative diagram:
\begin{equation}
\label{eqn:KLW}
\vcenter{\vbox{%
\xymatrix{%
Z \ar@{^{(}->}[r] \ar[d]^{\pi_{|_Z}} & X \ar[d]^{\pi} \ar[r]^{p} & Y\\
\pi(Z) \ar@{^{(}->}[r] & V^{\vee} }}}
\end{equation}
The sheaf $\OO_Z$ is resolved over $X$ by the following Koszul complex
\[
\mathcal{K}(\xi)_\bullet:\qquad 0\lra \bigwedge^r p^*(\xi) \lra \cdots \lra p^*(\xi) \lra \mathcal{O}_{X} \Bigl[\lra \mathcal{O}_Z \lra 0\Bigl].
\]
If $\mc{V}$ denotes a  locally free sheaf on $Y$,  we can tensor $\mathcal{K}(\xi)_\bullet$ with $p^*(\mc{V})$ to obtain a locally free $\mc{O}_{X}$-resolution of $\mc{M}(\mc{V}) := p^*(\mc{V})\otimes \OO_Z$, given by
\[
\mathcal{K}(\xi,\mathcal{V})_\bullet:\qquad 0\lra p^*(\mc{V}) \oo \bigwedge^r p^*(\xi) \lra \cdots \lra p^*(\mc{V}) \oo p^*(\xi) \lra p^*(\mc{V}) \Bigl[\lra \mc{M}(\mc{V}) \lra 0\Bigl].
\]
The derived direct image of $\mc{M}(\mc{V})$ along $\pi$, or equivalently, that of the complex $\mathcal{K}(\xi,\mathcal{V})_\bullet$, is represented by a minimal complex of graded $S$-modules constructed as follows (see \cite[Theorem~5.1.2]{weyman}).

\begin{thm}\label{thm:Basic-Theorem}
 Let $\mc{V}$ be a locally free sheaf on $Y$ and set $\mc{M}(\mc{V}):= p^*(\mc{V})\otimes \OO_Z$. There exists a minimal complex $F(\mathcal{V})_\bullet$ of graded $S$-modules, whose terms are defined by
\[
F(\mathcal{V})_i=\bigoplus_{j\ge 0}H^j\bigl(Y, \bigwedge^{i+j}\xi\otimes \mathcal{V}\bigr)\otimes_\kk S(-i-j)\  \mbox{ for all }i\in\bb{Z}.
\]
This complex is exact in positive homological degrees, and its homology groups in negative degrees are given by
\[
H_{-i}\bigl(F(\mathcal{V})_\bullet\bigr)=H^i\bigl(Z,\mc{M}(\mc{V})\bigr) = \bigoplus_{d\geq 0} H^i\bigl(Y,\mc{V} \oo \Sym^d\eta\bigr)\quad\mbox{ for all }i\geq 0.
\]
\end{thm}

If $\mc{M}(\mc{V})$ has vanishing higher cohomology, then the complex $F(\mathcal{V})_\bullet$ gives the minimal free resolution of $H^0\bigl(Z,\mc{M}(\mc{V})\bigr)$, as will be the case in the proof of Propositions~\ref{prop:resolution-Rbar}. We also apply Theorem~\ref{thm:Basic-Theorem} in the proof of Proposition~\ref{prop:J-complex}, where we construct a complex with two non-vanishing homology groups. 



\section{Syzygies of the tangent developable to a rational normal curve}
\label{sec:repth}

The goal of this section is to establish a link, following \cite{E91}, between the syzygies of the tangent developable $\mc{T}$ to a rational normal curve of degree $g$, and Koszul modules of a particular kind. This link, together with the general results on finite length Koszul modules that we established in Section~\ref{sec:fin-length-koszul}, allows us to completely characterize the (non-)vanishing behavior for the syzygies of $\mc{T}$.

\vskip 4pt

Throughout this section we fix a field $\kk$, a $2$-dimensional $\kk$-vector space $U$ and set $\PP^1:=\mathbb P\bigl(U^{\vee}\bigr)$. We fix $g\geq 3$, let $\PP^g:=\mathbb P\bigl(\D^g(U^{\vee})\bigr)$, and consider as in Section~\ref{subsec:Veronese} the degree $g$ \defi{Veronese embedding} $\nu:\PP^1 \lra \PP^g$, whose image is the \defi{rational normal curve} $\Gamma$ of degree $g$. We let $\mc{T}\subseteq \PP^g$ denote the \defi{tangential variety (or tangent developable)} of the curve $\Gamma$, and set
\begin{itemize}
 \item $S:=\Sym(\Sym^g U)\cong \Sym H^0(\mc{T},\OO_{\mc{T}}(1))$, the homogeneous coordinate ring of $\PP^g$.
 \item $I:=I_{\mc{T}}\subseteq S$, the homogeneous ideal defining $\mc{T}$.
 \item $R:= S/I$, the homogeneous coordinate ring of $\mc{T}$.
\end{itemize}
The goal of this section is to describe the minimal free resolution of $R$ as an $S$-module, and in particular to establish Theorem~\ref{thm:vanishing-tangential}. Our first result describes the shape of this minimal resolution.

\begin{thm}\label{thm:shape-res-R}
If $\chr(\kk)\neq 2$ then $R$ is Gorenstein with Castelnuovo--Mumford regularity $\reg(R)=3$.
\end{thm}

Since $R$ is the homogeneous coordinate ring of the surface $\mc{T}$, it follows that $\dim(R)=3$. Theorem~\ref{thm:shape-res-R} implies then that the Betti table of the $S$-module $R$ is described by (\ref{eq:betti-T}), where
\[b_{i,j} := \dim \  K_{i,j}  \bigl( \mc{T}, \OO_{\mc{T}}(1)\bigr).\]
Moreover, the Gorenstein property implies that $b_{i,1} = b_{g-2-i,2}$, for $i=1,\ldots,g-3$.
Based on this equality, we characterize the (non-)vanishing behavior of the Betti numbers of $\mathcal{T}$.

\begin{thm}\label{thm:nonvanishing-kij}
 Let $p=\chr(\kk)\neq 2$. If $p=0$ or $p\geq\frac{g+2}{2}$, then
 \begin{equation}\label{eq:ki2-not-0-plarge}
  b_{i,2} \neq 0\quad \Longleftrightarrow\quad \frac{g-2}{2} \leq i \leq g-3.
 \end{equation}
 If $3\leq p\leq\frac{g+1}{2}$,  then
 \begin{equation}\label{eq:ki2-not-0-psmall}
 b_{i,2} \neq 0\quad \Longleftrightarrow\quad p-2 \leq i \leq g-3.
 \end{equation}
\end{thm}

The vanishing Theorem~\ref{thm:vanishing-tangential} in the Introduction is now a direct consequence of Theorem~\ref{thm:nonvanishing-kij}. Indeed, it can be rephrased as $b_{\lfloor \frac{g}{2}\rfloor,1} = 0$ if $p=\chr(\kk)$ satisfies $p=0$ or $p\geq \frac{g+2}{2}$. Using the Gorenstein property of $R$ and that $g-2-\lfloor \frac{g}{2}\rfloor = \lfloor \frac{g-3}{2} \rfloor$, we have $b_{\lfloor \frac{g}{2}\rfloor,1} = b_{\lfloor \frac{g-3}{2} \rfloor,2}$. The vanishing of $b_{\lfloor \frac{g-3}{2} \rfloor,2}$ follows from (\ref{eq:ki2-not-0-plarge}).

\vskip 4pt

The key to proving Theorem~\ref{thm:nonvanishing-kij} is to interpret the Koszul cohomology groups $K_{i,2}(\mathcal{T}, \OO_{\mathcal{T}}(1))$ in the framework of Koszul modules. To that end, we consider for $a\geq 1$ the $\sl_2$-equivariant map (\ref{eq:Del1-to-wedge2}), which is an inclusion if $\chr(\kk)\neq 2$ (see Lemma~\ref{lem:mu1-delta1}), in which case we denote by $W^{(a)}$ the associated Koszul module $W(\D^{a} U,\D^{2a-2}U)$ (see (\ref{eq:def-W-mod})). This module was considered in \cite[Section~3.I.B]{E91} and was called a \defi{Weyman module}. The following is the crucial result of this section.

\begin{thm}\label{thm:Ki2=Wi+2}
If $\chr(\kk)\neq 2$, then for each $i=1,\ldots,g-3$, we have an identification
 \[K_{i,2}(\mc{T},\OO_{\mc{T}}(1)) = W^{(i+2)}_{g-3-i}.\]
\end{thm}

Our analysis of the syzygies of $\mc{T}$ is for the most part characteristic-free, and we will be careful to point out when the assumption that $\chr(\kk)\neq 2$ is used (see Section~\ref{subsec:char2} for the case $\chr(\kk)=2$). In particular, our analysis yields the following characteristic-free description of the groups $K_{i,1}(\mathcal{T}, \OO_{\mathcal{T}}(1))$.

\begin{thm}
\label{cor=kw}
If $\chr(\kk)$ is arbitrary, then for $i=0, \ldots, g-2$, we have an identification
\begin{equation}\label{eq:ker-delta-2}
K_{i,1}(\mathcal{T}, \OO_{\mathcal{T}}(1)) = \mathrm{Ker} \Bigl\{ \delta_2 :  \D^{2i}U \otimes \Sym^{g-2-i}(\D^{i+1}U) \lra
\D^{i+1}U \otimes \mathrm{Sym}^{g-1-i}(\D^{i+1}U )\Bigr\},
\end{equation}
where $\delta_2$ is the composition of the Koszul differential
\[\bw^2(\D^{i+1}U) \otimes \mathrm{Sym}^{g-2-i}(\D^{i+1}U ) \lra \D^{i+1}U \otimes \mathrm{Sym}^{g-1-i}(\D^{i+1}U)\]
with the map induced by (\ref{eq:Del1-to-wedge2})
\[\D^{2i}U \otimes \mathrm{Sym}^{g-2-i}(\D^{i+1}U )\lra \bw^2(\D^{i+1}U) \otimes \mathrm{Sym}^{g-2-i}(\D^{i+1}U).\]

\end{thm}

\vskip 4pt

Before delving into the proofs, we summarize the contents of this section. We begin by defining Weyman modules in Section~\ref{subsec:rep-theory}, and explaining how Theorem~\ref{thm:nonvanishing-kij} follows from Theorems~\ref{thm:shape-res-R},~\ref{thm:Ki2=Wi+2}, and from the results in Section~\ref{sec:fin-length-koszul}. In Section~\ref{subsec:summary-cor=kw} we give a roadmap to the proofs of Theorems~\ref{thm:shape-res-R},~\ref{thm:Ki2=Wi+2}, and~\ref{cor=kw}, which occupy the rest of the section. In Section~\ref{subsec:cohom-P1}, we summarize the $\sl_2$ description of cohomology of line bundles on $\PP^1$, which will be used throughout. In Section~\ref{subsec:pparts} we explain how to parametrize the affine cone over $\mc{T}$ using the bundle of principal parts. Sections~\ref{sec:resln-R-bar} -- \ref{sec:resln-R} describe the technical parts of our arguments: they rely heavily on the preliminaries discussed in Sections~\ref{sec:Hermite-rec} and~\ref{sec:KW}, particularly on Hermite reciprocity and the Kempf--Weyman technique, and they conclude with a proof of Theorem~\ref{cor=kw}. Using the foundation laid in the preceding sections, we give a quick proof of Theorems~\ref{thm:shape-res-R} and~\ref{thm:Ki2=Wi+2} in Section~\ref{subsec:char-not-2}. We end in Section~\ref{subsec:char2} with a concrete discussion of the resolution of $\mc{T}$ in characteristic $2$.

\subsection{Weyman modules in characteristic different from $2$}
\label{subsec:rep-theory}

For $n\geq 2$ we define
\begin{equation}\label{eq:def-W-mod}
W^{(n-1)} := W(\D^{n-1}U,\D^{2n-4}U)
\end{equation}
where $\D^{2n-4}U$ is regarded as a subspace of $\bw^2(\D^{n-1}U)$ via the map $\Delta_1$ (see (\ref{eq:Del1-to-wedge2}) and Lemma~\ref{lem:mu1-delta1} with $a=n-1$). We define $\mc{R}^{(n-1)}$ to be the resonance variety $\mc{R}(\D^{n-1}U,\D^{2n-4}U)$ (see (\ref{eq:defr})).

\begin{lemma}\label{lem:fdiml-Weyman}
 Let $p=\op{char}(\kk)$. If $p=0$ or $p\geq n$, then $\mc{R}^{(n-1)}=\{0\}$. If $3\leq p\leq n-1$ then $\mc{R}^{(n-1)}\neq\{0\}$.
\end{lemma}

\begin{proof}
Consider the map $\mu_1:\bw^2\Sym^{n-1}(U^{\vee})\lra\Sym^{2n-4}(U^{\vee})$ (defined in (\ref{eq:mu1-to-wedge2}) and Lemma~\ref{lem:mu1-delta1}), which is dual to $\Delta_1$. The condition $\mc{R}^{(n-1)}\neq\{0\}$ is equivalent to the fact that $\ker(\mu_1)$ contains a non-zero decomposable form $f_1\wedge f_2$. We choose a basis $(1,y)$ for $U^{\vee}$, so that for each $d\geq 0$ we can identify $\Sym^d(U^{\vee})$ with the space of polynomials of degree at most $d$ in $y$.

\vskip 3pt

If $3\leq p\leq n-1$ then $f_1=1$ and $f_2=y^p$ belong to $\Sym^{n-1}(U^{\vee})$, and $\mu_1(f_1\wedge f_2) = \mu_1(1\wedge y^p) = 0$. Since $1\wedge y^p\neq 0$, we conclude using (\ref{eq:defr}) that $\mc{R}^{(n-1)}\neq\{0\}$. Suppose now by contradiction that $\ker(\mu_1)$ contains a non-zero decomposable form $f_1\wedge f_2$, and that $p=0$ or $p\geq n$. By rescaling $f_1\wedge f_2$, we may assume that $f_1,f_2$ are monic polynomials in $y$, say
 \[f_1=y^r + b_{r-1}y^{r-1}+\cdots + b_0,\quad f_2 = y^s + c_{s-1}y^{s-1}+\cdots+c_0.\]
If $r\neq s$, we may assume $r>s$ and we get $\mu_1(f_1\wedge f_2) = (r-s)y^{r+s-1} + h(y)$, where $\deg(h)<r+s-1$  and $0<r-s\leq n-1$.
It follows that $p$ cannot divide $r-s$, hence $\mu_1(f_1\wedge f_2)\neq 0$, a contradiction. Suppose finally that $r=s$, and let $i$ be the maximal index for which $b_i\neq c_i$, which exists since $f_1\wedge f_2\neq 0$. We get
\[\mu_1(f_1\wedge f_2) = (r-i)\cdot(b_i-c_i)\cdot y^{r+i-1} + h(y),\mbox{ where }\deg(h)<r+i-1.\]
Since $0<r-i\leq n-1$ and $b_i-c_i\neq 0$, it follows again that $f_1\wedge f_2\not\in\ker(\mu_1)$, yielding a contradiction.
\end{proof}

\vskip 5pt

We can now give a quick proof of Theorem~\ref{thm:nonvanishing-kij}, assuming the statements of Theorems~\ref{thm:shape-res-R} and~\ref{thm:Ki2=Wi+2}.

\begin{proof}[Proof of Theorem~\ref{thm:nonvanishing-kij}]
By Theorem~\ref{thm:shape-res-R}, the only potentially non-zero $b_{i,2}$ occur for $1\leq i\leq g-3$, so we fix one such index $i$. Since $p\neq 2$, Theorem~\ref{thm:Ki2=Wi+2} applies to give
\begin{equation}\label{eq:ki2=dimWi+2}
b_{i,2} = \dim \bigl(W^{(i+2)}_{g-3-i}\bigr).
\end{equation}
Let $n=i+3$, so that $W^{(i+2)}=W^{(n-1)}$, to which we apply Lemma~\ref{lem:fdiml-Weyman}. Suppose first $3\leq p\leq n-1 = i+2$, so the support $\mc{R}^{(n-1)}$ of $W^{(n-1)}$ is positive dimensional, in particular $W^{(n-1)}_q \neq 0$ for all $q\geq 0$. Taking $q=g-3-i$ we obtain $b_{i,2}\neq 0$, in particular the implication ``$\Longleftarrow$" in  (\ref{eq:ki2-not-0-psmall}) holds when $3\leq p\leq \frac{g+1}{2}$. If $p\geq \frac{g+2}{2}$,  using the assumption $p\leq i+2$, we conclude that $i\geq \frac{g-2}{2}$, so conditioned on $3\leq p \leq i+2$, the equivalence in (\ref{eq:ki2-not-0-plarge}) holds.

\vskip 3pt

Suppose now that $p=0$ or $p\geq n=i+3$, so that $\mc{R}^{(n-1)}=\{0\}$. From Theorems~\ref{thm:vanishing-koszul} and~\ref{thm:boundHilb} we obtain the equivalence $W^{(n-1)}_q \neq 0$ if and only if  $0\leq q\leq n-4$,
which, in view of (\ref{eq:ki2=dimWi+2}), implies that
\begin{equation}\label{eq:equiv-ki2}
b_{i,2} \neq 0 \Longleftrightarrow 0\leq g-3-i\leq i-1 \Longleftrightarrow \frac{g-2}{2}\leq i\leq g-3.
\end{equation}

If $3\leq p\leq \frac{g+1}{2}$ then $i\leq p-3 \leq \frac{g-5}{2}$, proving based on (\ref{eq:equiv-ki2}) that $b_{i,2}=0$ and establishing the remaining part of the equivalence (\ref{eq:ki2-not-0-psmall}). If $p=0$ or $p\geq \frac{g+2}{2}$ then (\ref{eq:equiv-ki2}) agrees with (\ref{eq:ki2-not-0-plarge}).
\end{proof}

\subsection{The roadmap to the minimal free resolution of $R$.}
\label{subsec:summary-cor=kw}
For the  convenience of the reader, we summarize here our analysis of the minimal free resolution of $R$ as an $S$-module, which concludes with the proofs of Theorems~\ref{thm:shape-res-R},~\ref{thm:Ki2=Wi+2}, and~\ref{cor=kw}. The key players in our proof are four minimal complexes of free graded $S$-modules, denoted $F_{\bullet},J_{\bullet},C_{\bullet}$ and $K_{\bullet}$, the first two of which are constructed via the Kempf--Weyman geometric technique (see Section~\ref{subsec:Kempf-Weyman}). For visualization purposes, we record their Betti diagrams below: our convention is that for a complex $A_{\bullet}$, a $``*"$ in row $j$ and column $i$ indicates that the free module $A_i$ has generators of degree $i+j$, while a $``-"$ indicates that no such generators exist. We will also write $``1"$ instead of $``*"$ when $A_i \cong S(-i-j)$ is a free module of rank one, generated in degree $i+j$. With these conventions, our four complexes take the following shapes:

\[
F_{\bullet}:\qquad
\begin{array}{c|cccccc}
     &0&1&2&\cdots&g-3&g-2\\ \hline
     0&1&-&-&\cdots&-&-\\
     1&*&*&*&\cdots&*&*\\
\end{array}
\qquad\qquad
C_{\bullet}:\qquad
\begin{array}{c|ccccccc}
     &0&1&2&\cdots&g-3&g-2&g-1 \\ \hline
     0&-&-&-&\cdots&-&-&-\\
     1&*&*&*&\cdots&*&*&-\\
     2&-&-&-&\cdots&-&-&1\\
\end{array}
\]
\[
J_{\bullet}:\qquad
\begin{array}{c|cccccc}
     &0&1&2&\cdots&g-1&g\\ \hline
     0&1&*&*&\cdots&*&*\\
\end{array}
\qquad\qquad
K_{\bullet}:\qquad
\begin{array}{c|cccccc}
     &0&1&2&\cdots&g-1&g\\ \hline
     0&1&*&*&\cdots&*&1\\
\end{array}
\]

\vskip 4pt

\noindent
\textbf{The complex $F_{\bullet}$.} The affine cone $\widehat{\mc{T}}$ over $\mc{T}$ admits a dominant map from the total space $Z$ of the bundle of principal parts associated with the sheaf $\OO_{\PP^1}(g)$ (see Section~\ref{subsec:pparts}). As such, $R$ is naturally a subring of $\tl{R} = H^0(Z,\mc{O}_Z)$. The complex $F_{\bullet}$ gives the minimal free resolution of $\tl{R}$ (see Section~\ref{sec:resln-R-bar}). In characteristic different from two, the map $Z\to \widehat{\mc{T}}$ is birational, and $\tl{R}$ is the normalization of $R$.

\vskip 4pt
\noindent
\textbf{The complex $C_{\bullet}$.} The affine cone $\widehat{\Gamma}$ over the rational normal curve is Cohen--Macaulay, so it has a canonical module which we denote by $C$. The complex $C_{\bullet}$ is the minimal resolution of $C$ (see Section~\ref{sec:resln-C}).

\vskip 4pt
\noindent
\textbf{The complex $J_{\bullet}$.} The complex $J_{\bullet}$ is the only one that is not a resolution (see Proposition~\ref{prop:J-complex}): it is a linear complex with two non-zero cohomology groups, namely $H_0(J_{\bullet}) = \kk$, and $H_1(J_{\bullet})=C$.

\vskip 4pt
\noindent
\textbf{The complex $K_{\bullet}$.} The complex $K_{\bullet}$ is the Koszul complex resolving the residue field $\kk \cong S/\mf{m}$.

\vskip 4pt

We now describe the various links between the complexes $F_{\bullet},J_{\bullet},C_{\bullet}$ and $K_{\bullet}$.

\vskip 4pt

\noindent
\textbf{The relationship between $C_{\bullet},J_{\bullet}$ and $K_{\bullet}$.} There is a degree preserving map  $p_{\bullet}:J_{\bullet}\lra K_{\bullet}$, constructed in Proposition~\ref{prop:map-J-to-K}, inducing an isomorphism $\kk=H_0(J_{\bullet}) \cong H_0(K_{\bullet})=\kk$. The mapping cone of $p_{\bullet}$ gives rise after minimization (as in Section~\ref{subsec:minimization}) and a homological shift, to a minimal resolution of $C$, and is therefore quasi-isomorphic to $C_{\bullet}$. Precisely, $C_i = \ker(p_{i+1})$ for $i=0,\ldots,g-2$ (see Corollary~\ref{cor:C_i=ker-pi+1}).

\vskip 4pt

\noindent
\textbf{The relationship between $F_{\bullet},J_{\bullet}$ and (\ref{eq:ker-delta-2}).} There is a map $q_{\bullet}:F_{\bullet}\lra J_{\bullet+1}$ of complexes, which is constructed in Proposition~\ref{prop:map-F-to-J}, inducing a surjection $\tl{R} = H_0(F_{\bullet}) \onto H_1(J_{\bullet})=C$. Using the explicit Hermite reciprocity discussed in Section~\ref{subsec:Hermite}, it follows that if we restrict the map $q_i$ to the minimal generators of degree $i+1$ of $F_i$ and $J_{i+1}$, then we get precisely the map $\delta_2$ in (\ref{eq:ker-delta-2}), so
\begin{equation}\label{eq:kerd2=kerqi-oo-kk}
\ker(\delta_2)=\ker\bigl((q_i \oo \kk)_{i+1}\bigr)\mbox{ for }i=0,\ldots,g-2.
\end{equation}

\vskip 4pt

\noindent
\textbf{The relationship between $q_{\bullet}$ and the first linear strand of the resolution of $R$.}
The morphism $q_0$ sends the unique generator of $F_0$ of degree $0$ to zero, and therefore $R$ is contained in $R':=\ker(\tl{R}\onto C)$. The quotient $R'/R$ is generated in degree $\geq 2$, so the resolutions of $R$ and $R'$ have the same first linear strand. To construct the minimal resolution of $R'$, we show in Corollary~\ref{cor:F-to-C} that $q_i\circ p_{i+1}=0$ for $i=0, \ldots,g-2$, so $q_i$ sends $F_i$ into $C_i\subseteq J_{i+1}$. By restricting the range, this induces a map of complexes $\tl{q}_{\bullet}:F_{\bullet}\lra C_{\bullet}$, lifting the quotient map $\tl{R}\onto C$. Minimizing the mapping cone of $\tl{q}_{\bullet}$ as in Lemma~\ref{lem:minimization}, we get a minimal resolution of $R'$, with
\begin{equation}\label{eq:Tor-R'}
\Tor_i^S(R',\kk)_{i+1} = \ker\bigl((\tl{q}_i \oo \kk)_{i+1}\bigr),\quad \Tor_i^S(R',\kk)_{i+2} = \coker\bigl((\tl{q}_{i+1} \oo \kk)_{i+2}\bigr).
\end{equation}
We conclude in Section~\ref{sec:resln-R} that
\[ K_{i,1}(\mc{T},\OO_{\mc{T}}(1)) = \Tor_i^S(R,\kk)_{i+1} = \Tor_i^S(R',\kk)_{i+1} = \ker\bigl((\tl{q}_i \oo \kk)_{i+1}\bigr) = \ker\bigl((q_i \oo \kk)_{i+1}\bigr).\]
It then follows from (\ref{eq:kerd2=kerqi-oo-kk}) that $K_{i,1}(\mc{T},\OO_{\mc{T}}(1))=\ker(\delta_2)$ as in (\ref{eq:ker-delta-2}), proving Theorem~\ref{cor=kw}.

\vskip 4pt

\noindent
\textbf{The minimal resolution of $R$ for $\chr(\kk)\neq 2$.}
By construction, we have $\coker(\tl{q}_i) = \ker(p_{i+1}) / \Im(q_i)$, which allows us to realize $\coker(\tl{q}_i)$ as the middle homology of a $3$-term complex
\[ F_i \overset{q_i}{\lra} J_{i+1} \overset{p_{i+1}}{\lra} K_{i+1}.\]
Restricting this sequence to minimal generators of degree $(i+1)$ we obtain via Hermite reciprocity the $3$-term complex defining the degree $g-2-i$  component of the Weyman module $W^{(i+1)}$, and conclude in Lemma~\ref{lem:coker-tl-q} that $\coker(\tl{q}_i) \oo_S \kk = W^{(i+1)}_{g-2-i}$. Taking $i=0$ we find that $\tl{q}_1$ is surjective, and deduce that $R'=R$ in Corollary~\ref{cor:R=R'}. Combining this with (\ref{eq:Tor-R'}) proves Theorem~\ref{thm:Ki2=Wi+2}. To prove Theorem~\ref{thm:shape-res-R} one is left with verifying that $\tl{q}_{g-2}$ is injective, which follows from one final application of Hermite reciprocity and is explained at the end of Section~\ref{subsec:char-not-2}.

\vskip 4pt

\noindent
\textbf{The minimal resolution of $R$ for $\chr(\kk)=2$.} It follows from the discussion in Section~\ref{subsec:pparts} that $\mc{T}$ is a variety of minimal degree in characteristic two. We show in Section~\ref{subsec:char2} that it is a rational normal scroll, and conclude that its minimal resolution is given by an Eagon--Northcott complex.

\subsection{Cohomology on $\PP^1$}
\label{subsec:cohom-P1}

Let $U$ be a two dimensional vector space as before and recall that $\PP^1=\mathbb P\bigl(U^{\vee}\bigr)$. The tautological quotient map $U\oo\OO_{\PP^1} \onto \OO_{\PP^1}(1)$ gives rise to an exact Koszul complex
\begin{equation}\label{eq:tautological-seq}
0\lra\bw^2 U \oo \OO_{\PP^1}(-1)\lra U\oo\OO_{\PP^1} \lra\OO_{\PP^1}(1)\lra 0,
\end{equation}
and using (\ref{lazb}) we get a natural identification $M_{\OO_{\PP^1}(1)} \cong  \bw^2 U \oo \OO_{\PP^1}(-1)$. The canonical bundle $\omega_{\PP^1}$ is naturally identified with $\bw^2 U \oo \OO_{\PP^1}(-2)$. The cohomology groups of twists of $\OO_{\PP^1}$ are then given~by:
\[
H^0(\PP^1,\OO_{\PP^1}(a)) = \begin{cases}
 \Sym^a U & \mbox{if }a\geq 0, \\
0 & \mbox{otherwise,}
\end{cases}
\quad
H^1(\PP^1,\OO_{\PP^1}(a)) = \begin{cases}
 \bw^2 U^{\vee} \oo \D^{-a-2}(U^{\vee}) & \mbox{if }a \leq -2, \\
0 & \mbox{otherwise.}
\end{cases}
\]


\noindent Thinking of (\ref{eq:tautological-seq}) as a two-step resolution of $\OO_{\PP^1}(1)$, the sheaf $\mc{O}_{\PP^1}(a) \cong \Sym^a(\mc{O}_{\PP^1}(1))$ is then resolved by the $a$-th symmetric power of the said two-step resolution. This gives an exact sequence
\begin{equation}\label{eq:koszul-P1}
0\lra \Sym^{a-1} U\oo \bw^2 U \oo \OO_{\PP^1}(-1) \overset{\iota_a}{\lra} \Sym^{a}U\oo\OO_{\PP^1} \lra \OO_{\PP^1}(a)\lra 0,
\end{equation}
where the second map is given by evaluation of sections, and the first is induced by the inclusion in (\ref{eq:tautological-seq}), and by the multiplication  $\mu:\Sym^{a-1}U\oo U\lra \Sym^a U$. Using (\ref{lazb}) we get from (\ref{eq:koszul-P1}) an identification
\begin{equation}\label{eq:Mbun-Oa}
 M_{\OO_{\PP^1}(a)} \cong  \Sym^{a-1} U \oo \bw^2 U \oo \OO_{\PP^1}(-1).
\end{equation}


\subsection{Jet bundles and the Gauss map}\label{subsec:pparts}
Recall that $\Gamma$ is the rational normal curve of degree $g$ in $\PP^g$, obtained as the image of the embedding of $\PP^1$ via the complete linear system $|\OO_{\PP^1}(g)|$, and  that $\mc{T}\subseteq \PP^g$ denotes the tangential variety of $\Gamma$. The first jet bundle $\mc{J}:=\mc{P}^1\bigl(\OO_{\PP^1}(g)\bigr)$ has the explicit description
\begin{equation}\label{eq:eta-split}
\mc{J}= \begin{cases}
U \otimes \OO_{\PP^1}(g-1) & \mbox{if }p\nmid g; \\
\bw^2 U \oo \OO_{\PP^1}(g-2)\oplus\OO_{\PP^1}(g) & \mbox{if }p|g.
\end{cases}
\end{equation}
The vector bundle $\mc{J}$ sits in an exact sequence
\begin{equation}\label{eq:ses-eta}
 0 \lra \bigl(\bw^2 U\bigr)^{\oo 2} \oo \Sym^{g-2}U \oo \OO_{\PP^1}(-2) \overset{\iota}{\lra} \Sym^g U \oo \mc{O}_{\PP^1} \lra \mathcal{J} \lra 0,
\end{equation}
where $\bigl(\bw^2 U\bigr)^{\oo 2} \oo \Sym^{g-2}U\otimes \OO_{\PP^1}(-2) \cong \mc{N}_{\Gamma/\PP^g}^{\vee}(g)$, see for instance \cite[Corollary ~2.5]{CR} for a characteristic free calculation. The inclusion map $\iota$ in (\ref{eq:ses-eta}) is induced by composing the map $\iota_{g}$ in (\ref{eq:koszul-P1}) with $\op{id}_{\bw^2 U} \oo \iota_{g-1}(-1)$. We denote by
$\nu:\PP(\mc{J})\rightarrow \PP^g$ the map induced by the subspace of sections
$$\mbox{Sym}^g U\subseteq H^0(\PP^1,\mc{J})\cong H^0\bigl(\PP(\mc{J}), \OO_{\PP(\mc{J})}(1)\bigr).$$
The image of $\tau$ is the tangential variety $\mc{T}$ and $\tau$ is  a resolution of singularities if $\chr(\kk)\neq 2$. Let
\begin{equation}\label{eq:def-mcS-Z-X}
\mc{S} := \Sym_{\OO_{\PP^1}}(\mc{J}) \ \mbox{ and } Z := \relSpec_{\PP^1}(\mc{S}) \lhra \D^g(U^{\vee}) \times \PP^1 =: X.
\end{equation}
If we let $\bb{G}$ denote the Grassmannian of two dimensional subspaces of $\D^g(U^{\vee})$, then the surjection in (\ref{eq:ses-eta}) gives rise to the \defi{Gauss map} $\gamma:\Gamma\lra\bb{G}$, whose image we denote by $\Gamma^*$. We have a diagram (see \cite{kaji})
\begin{equation}\label{eq:pparts-gauss-diag}
\begin{aligned}
\xymatrix{
& \Gamma\cong\PP^1 \ar[d]_{\gamma} & Z \ar[l]_(.35){\rho} \ar@{^{(}->}[r] \ar[d]_{\pi_{|_Z}} & X \ar[d]^{\pi} \\
\bb{G} & \Gamma^* \ar@{_{(}->}[l] & \widehat{\mc{T}} \ar@{^{(}->}[r] & \D^g(U^{\vee})
}
\end{aligned}
\end{equation}
where $\widehat{\mathcal{T}}$ is the affine cone over $\mathcal{T}$. We recall that $R=H^0(\widehat{\mc{T}},\mc{O}_{\widehat{\mc{T}}})$ denotes the homogeneous coordinate ring of $\mc{T}$, and let $\tl{R} = H^0(Z,\mc{O}_Z)$. Since $\pi$ maps $Z$ onto $\widehat{T}$, the pull-back along $\pi_{|_Z}$ leads to an inclusion of $R$ into $\tl{R}$, which will be analyzed later on. Applying \cite[Lemma~1.6,~Corollary~1.7,~Remark ~1.9]{kaji}, we conclude that $\deg(\mc{T}) = \deg(\Gamma^*)$ and $\deg(\gamma)=\deg(\pi_{|_Z})$.  Moreover, $\gamma$ can be written in an affine chart as
\begin{equation}\label{eq:chart-T}
t \overset{\gamma}{\lra} \mbox{row Span}\begin{bmatrix}
1 & t & t^2 & \cdots & t^g \\
0 & 1 & 2t & \cdots & gt^{g-1}
\end{bmatrix}.
\end{equation}

Suppose first that $\chr(\kk)\neq 2$. Since the $2\times 2$ minors of the above matrix span the space of polynomials in $t$ of degree up to $2g-2$, it follows that when viewed in its Pl\"ucker embedding, $\Gamma^*$ is a rational normal curve (in its span) of degree $2g-2$. Moreover, the map $\gamma$ is an isomorphism. We conclude that $\deg(\Gamma^*)=2g-2$ and $\deg(\gamma)=1$, and therefore $\deg(\mc{T})=2g-2$ and $\pi_{|_Z}$ is birational. It follows that the map $\pi_{|_Z}$ gives a resolution of singularities for $\widehat{\mc{T}}$, and in particular $\tl{R}$ is the normalization of $R$.

Suppose now that $\chr(\kk)=2$. The $2\times 2$ minors of (\ref{eq:chart-T}) span the space of polynomials in $t^2$ of degree $g-1$, so $\gamma$ can be regarded as the composition of the Frobenius map with a Veronese embedding of degree $g-1$. It follows that $\deg(\mc{T}) = \deg(\Gamma^*)=g-1$ and $\deg(\pi_{|_Z})=\deg(\gamma)=2$.
\vskip 4pt

\subsection{The resolution of $\tl{R}$}
\label{sec:resln-R-bar}

Recall that $S=\Sym(\Sym^g U)\cong \mbox{Sym}\ H^0\bigl(\mc{T}, \OO_{\mc{T}}(1)\bigr)$ denotes the homogeneous coordinate ring of~$\PP^g$, let $X,Z$ as in (\ref{eq:def-mcS-Z-X}), and let $\tl{R}:=H^0(Z,\mc{O}_Z)=\bigoplus_{n\geq 0} H^0(\PP(\mc{J}),\OO_{\PP(\mc{J})}(n))$. We describe the minimal free resolution of $\tl{R}$ as an $S$-module.

\begin{proposition}
\label{prop:resolution-Rbar}
 The $S$-module $\tl{R}:=H^0(Z,\OO_Z)$ has an $\mf{sl}_2$-equivariant minimal free resolution given by
 \[F_{\bullet}:\qquad F_{g-2} \lra \cdots \lra F_1 \lra F_0 \bigl[\lra \tl{R} \lra 0\bigr] , \]
 where $F_0 = S \oplus \Sym^{g-2}U \oo S(-1)$, and
 \[F_i = \Bigl(\D^{2i}U \oo \bw^{i+1}\Sym^{g-2}U\Bigr)\oo S(-i-1) \  \mbox{ for }i=1,\ldots,g-2.\]
 The linear part of the differential $\pd^F_i : F_i \lra F_{i-1}$ is defined  via an $\sl_2$-equivariant map
 \[ \D^{2i}U \oo \bw^{i+1}\Sym^{g-2}U \lra \Bigl(\D^{2i-2}U \oo \bw^{i}\Sym^{g-2}U \Bigr) \oo \Sym^g U\]
 induced by
 \[ \Delta^{(2)}:\D^{2i}U \overset{(\ref{eq:comm-diag-co-multiplication})}{\lra} \D^{2i-2}U \oo \Sym^2 U,\quad k_{i+1}^{g-2}:\bw^{i+1}\Sym^{g-2}U \overset{(\ref{eq:dkos-slam})}{\lra} \bw^{i}\Sym^{g-2}U \oo \Sym^{g-2}U,\]
 \[\mbox{ and }\mu:\Sym^2 U \oo \Sym^{g-2}U \overset{(\ref{eq:multiplication})}{\lra} \Sym^g U.\]
 If we ignore the quadratic part of the differential $\pd^F_1 : F_1 \lra F_0$, then we have the formula
 \begin{equation}\label{eq:def-delF}
  \pd^F_i(x^{(t)} \oo s_{\ll}(x)) = \sum_{\substack{u+s = t \\ u\leq 2,\ s\leq 2i-2}} \sum_{j=1}^{i+1} (-1)^{j-1} \cdot x^{(s)} \oo s_{\ll^{\hat{j}}}(x) \oo \left({2\choose u}\cdot x^u \cdot x^{\ll_j + i + 1 - j}\right)
 \end{equation}
 for all $0\leq t\leq 2i$ and $\ll\in\mc{P}_{i+1}(g-2-i)$.
\end{proposition}

\begin{rmk} In  the language of Koszul cohomology, Proposition \ref{prop:resolution-Rbar} yields the $\mathfrak{sl}_2$-identification
$$K_{i,1}\bigl(\mc{T}, \tau_*\OO_{\PP(\mc{J})}, \OO_{\mc{T}}(1)\bigr)\cong D^{2i} U\otimes \bigwedge^{i+1} \mbox{Sym}^{g-2} U, \mbox{ for } i=1, \ldots, g-2.$$
\end{rmk}
\begin{proof}
We let $\xi = \bigl(\bw^2 U\bigr)^{\oo 2} \oo \Sym^{g-2}U \oo \OO_{\PP^1}(-2)$, $V = \Sym^g U$ and $\eta = \mc{J}$, so that (\ref{eq:ses-eta}) is a special instance of (\ref{eq:ses-xi-V-eta}). We let $\mc{V} = \OO_{\PP^1}$ and apply Theorem~\ref{thm:Basic-Theorem} to construct a complex $F_{\bullet}$ with
\begin{equation}\label{eq:res-tildeR}
 F_i = \bigoplus_{j\geq 0} H^j\Bigl(\PP^1,\bw^{i+j}\xi\Bigr) \oo S(-i-j).
\end{equation}
It follows from (\ref{eq:eta-split}) and Section~\ref{subsec:cohom-P1} that $\mc{M}(\mc{V}) = \mc{O}_Z$ has no higher cohomology, so $F_{\bullet}$ is a minimal free resolution of $\tl{R}$. We have that $\bw^{i+j}\xi = \bigl(\bw^2 U\bigr)^{\oo (2i+2j)} \oo \bw^{i+j}\Sym^{g-2}U \oo \OO_{\PP^1}(-2i-2j) $ and therefore
\begin{equation}\label{eq:coh-wedge-xi}
H^j\Bigl(\PP^1,\bw^{i+j}\xi\Bigr) = \begin{cases}
 \kk & \mbox{if }i=j=0; \\
 \bw^2 U \oo \D^{2i}U \oo \bw^{i+1}\Sym^{g-2}U & \mbox{if }1\leq i\leq g-2\mbox{ and }j=1; \\
 0 & \mbox{otherwise}.
\end{cases}
\end{equation}
Indeed, if $i+j=0$ then $\bw^{i+j}\xi=\mc{O}_{\PP^1}$, whose only non-vanishing cohomology is $H^0(\PP^1,\mc{O}_{\PP^1})=\kk$. Since $\rk(\xi)=g-1$, we may assume that $0<i+j\leq g-1$, and since $\OO_{\PP^1}(-2i-2j)$ has no global sections, we may assume that $j=1$. It follows from Section~\ref{subsec:cohom-P1} that
\[H^1\bigl(\PP^1,\bw^{i+1}\xi\bigr) = \bigl(\bw^2 U\bigr)^{\oo (2i+2)} \oo \bw^{i+1}\Sym^{g-2}U \oo \bw^2 U^{\vee} \oo \D^{2i}(U^{\vee}).\]
Using $\bw^2 U \oo \bw^2 U^{\vee} \cong \kk$ and $\bigl(\bw^2 U\bigr)^{\oo 2i} \oo \D^{2i}(U^{\vee}) \cong \D^{2i}U$, we get (\ref{eq:coh-wedge-xi}).

Since we are only interested in the $\sl_2$-structure, we have that $\bw^2 U \cong \kk$, so (\ref{eq:res-tildeR}) and (\ref{eq:coh-wedge-xi}) give the desired formula for the terms in the resolution $F_{\bullet}$. The formula (\ref{eq:def-delF}) follows by combining (\ref{eq:multiplication}), (\ref{eq:co-multiplication}), (\ref{eq:comm-diag-co-multiplication}) and (\ref{eq:dkos-slam}), and recalling that $x^{(u)}\in\Sym^2 U$ should be interpreted as ${2\choose u}\cdot x^u$ for $0\leq u\leq 2$.
\end{proof}

Since $R$ is a quotient of $S$ and a subring of $\tl{R}$, it is isomorphic to $S/I$, where $I$ is the image of the quadratic part of the differential $\pd^F_1:F_1\to F_0$. It is not clear what this is from the construction above, but the construction tells us that $I$ is generated by quadrics. In characteristic $2$, it is easy to identify a subcomplex of $F_{\bullet}$ which gives a resolution of $R$, which we do in Section~\ref{subsec:char2}. In the general case we need one more auxiliary construction. 

\subsection{The minimal resolution of the canonical module $C$ of the affine cone $\widehat{\Gamma}$}
\label{sec:resln-C}

We construct a minimal resolution of $C$ as the minimization of a mapping cone. We show that it admits an explicit map from the resolution of $\tl{R}$, inducing a surjection $\tl{R}\onto  C$. In Section~\ref{subsec:char-not-2} we show that this map gives rise to an isomorphism $C\cong\widetilde{R}/R$ provided that $\mbox{char}(\kk)\neq 2$.

\vskip 4pt

Recall from (\ref{lazb}) and (\ref{eq:Mbun-Oa}) that $M_{\OO_{\PP^1}(g)}= \Sym^{g-1}U \oo \bw^2 U \oo \OO_{\PP^1}(-1)$. We define, in analogy with (\ref{eq:def-mcS-Z-X})
\begin{equation}\label{eq:def-mcS'-Z'-X}
\mc{S}' := \Sym_{\OO_{\PP^1}}\bigl(\OO_{\PP^1}(g)\bigr) \ \mbox{ and } Z' := \relSpec_{\PP^1}(\mc{S}') \lhra \D^g(U^{\vee}) \times \PP^1 = X,
\end{equation}
and we consider the canonical module of $\widehat{\Gamma}$, given as
\begin{equation}\label{eq:defC}
C := H^0(Z',\mc{O}_{Z'}(g-2))=\bigoplus_{n\geq 0} H^0\bigl(\Gamma, \omega_{\Gamma}\oo \OO_{\Gamma}(n)\bigr).
\end{equation}
We note that $C$ is generated in degree one by $\Sym^{g-2}(U)$. We let $C_{\bullet}$ denote the minimal free resolution of $C$, which is dual to that of the homogeneous coordinate ring of $\Gamma$. Using \cite[Corollary~6.2]{E04} it follows that for $i=0,\ldots,g-2$, the free module $C_i$ is generated in degree $i+1$, and $C_{g-1} \cong S(-g-1)$. Our goal is to realize $C_{\bullet}$ as the minimization of a mapping cone. To that end, we start with the following.

\begin{proposition}\label{prop:J-complex}
 There exists an $\mf{sl}_2$-equivariant complex of free $S$-modules
\[J_{\bullet}:\qquad 0\lra J_g \lra \cdots \lra J_1 \lra J_0,\ \mbox{ where}\]
\[J_i = \Bigl(\D^{i}U \oo \bw^{i}\Sym^{g-1}U\Bigr)\oo S(-i), \mbox{ for }i=0,\ldots,g,\]
and whose homology is given by $H_0(J_{\bullet}) = \kk \cong S/\mf{m},\ H_1(J_{\bullet}) = C$ and $H_i(J_{\bullet}) = 0$,  for $i>1$.
Moreover, the differential $\pd^J_i : J_i \lra J_{i-1}$ is defined at the level of generators via an $\sl_2$-equivariant map
\[ \D^{i}U \oo \bw^{i}\Sym^{g-1}U \lra \Bigl(\D^{i-1}U \oo \bw^{i-1}\Sym^{g-1}U \Bigr) \oo \Sym^g U\]
 induced by
 \[ \Delta:\D^{i}U \overset{(\ref{eq:co-multiplication})}{\lra} \D^{i-1}U \oo U,\quad k_i^{g-1}:\bw^{i}\Sym^{g-1}U \overset{(\ref{eq:dkos-slam})}{\lra} \bw^{i-1}\Sym^{g-1}U \oo \Sym^{g-1}U,\]
 \[\mbox{ and }\mu:U \oo \Sym^{g-1}U \overset{(\ref{eq:multiplication})}{\lra} \Sym^g U.\]
 More precisely, for all $0\leq t\leq i$ and $\ll\in\mc{P}_i(g-i)$, we have
 \begin{equation}\label{eq:def-delJ}
  \pd^J_i(x^{(t)} \oo s_{\ll}(x)) = \sum_{\substack{u+s = t \\ u\leq 1,\ s\leq i-1}} \sum_{j=1}^{i} (-1)^{j-1} \cdot x^{(s)} \oo s_{\ll^{\hat{j}}}(x) \oo x^{\ll_j + i - j + u}.
 \end{equation}
\end{proposition}

\begin{proof}
We use (\ref{eq:ses-xi-V-eta}) with $\xi = M_{\OO_{\PP^1}(g)}$, $V = \Sym^g U$, $\eta = \OO_{\PP^1}(g)$, and let $\mc{V}=\OO_{\PP^1}(-2)$. It follows that $\mc{M}(\mc{V}) = \mc{O}_{Z'}(-2)$, whose only non-zero cohomology groups are
\[H^0(X,\mc{O}_{Z'}(-2)) = C \mbox{  and  }H^1(X,\mc{O}_{Z'}(-2)) = H^1(\PP^1,\OO_{\PP^1}(-2)) = \bw^2 U^{\vee} \cong \kk,\]
where $\kk$ is a graded $S$-module concentrated in degree $0$. Theorem~\ref{thm:Basic-Theorem} yields a minimal complex $G_{\bullet}$ with
\[ G_i = \bigoplus_{j\geq 0} H^j\Bigl(\PP^1, \OO_{\PP^1}(-2)\oo\bw^{i+j} M_{\OO_{\PP^1}(g)} \Bigr) \oo S(-i-j),\]
whose non-zero homology is given by $ H_0(G_{\bullet}) = C$ and $H_{-1}(G_{\bullet}) = \kk.$ Note that
\[H^j\Bigl(\PP^1,\OO_{\PP^1}(-2)\oo\bw^{i+j} M_{\OO_{\PP^1}(g)}\Bigr) = \begin{cases}
 \D^{i+1}U \oo \bw^{i+1}\Sym^{g-1}U & \mbox{if }-1\leq i\leq g-1\mbox{ and }j=1; \\
 0 & \mbox{otherwise}.
\end{cases}\]
In the above calculation we have used the description of cohomology from Section~\ref{subsec:cohom-P1} and the isomorphism $\bigl(\bw^2 U\bigr)^{\oo i+1} \oo \D^{i+1}(U^{\vee}) \cong \D^{i+1}(U)$. Letting $J_{\bullet} = G_{\bullet+1}$ we get a complex $J_{\bullet}$ with the desired properties. We get (\ref{eq:def-delJ}) by combining (\ref{eq:multiplication}), (\ref{eq:co-multiplication}) and (\ref{eq:dkos-slam}), and noting that $x^{(u)}= x^u\in U$ for $u=0,1$.
\end{proof}

\noindent To build a mapping cone resolution for $C$, we need a map of complexes between $J_{\bullet}$ and the Koszul complex resolving the residue field $\kk$. We achieve this as follows.

\begin{proposition}\label{prop:map-J-to-K}
 Consider the Koszul complex resolving the residue field $\kk$ as an $S$-module
 \[K_{\bullet}:\qquad 0\lra K_{g+1} \lra \cdots \lra K_1\lra K_0 \left[ \lra \kk \lra 0\right],\]
 where
 \[K_i = \Bigl(\bw^i\Sym^g U\Bigr)\oo S(-i), \mbox{ for }i=0,\ldots,g+1,\]
and the differential given by the maps $k_i^g$ in (\ref{eq:dkos-slam}). Using the notation (\ref{eq:nud-formula}) and letting $\ol{p}_i = \nu_{g-i}^i$, the maps
 \[\ol{p}_i : \D^i U \oo \bw^i\Sym^{g-1}U \lra \bw^i\Sym^g U\]
 induce a map of complexes $p_{\bullet}:J_{\bullet}\lra K_{\bullet}$, with $\ol{p}_i = p_i \oo \op{id}_\kk$.
\end{proposition}

\begin{proof}
 We have to verify that for each $i = 1,\ldots,g+1$ we have a commutative diagram
 \[
 \xymatrix{
 \D^i U \oo \bw^i\Sym^{g-1}U \ar[d]_{\ol{p}_i} \ar[rrr]^{\pd^J_i} & & & \D^{i-1}U \oo \bw^{i-1}\Sym^{g-1}U \oo \Sym^g U \ar[d]^{\ol{p}_{i-1} \oo \id_{\Sym^g U}} \\
 \bw^i\Sym^g U \ar[rrr]_{\pd^K_i} & & & \bw^{i-1}\Sym^g U \oo \Sym^g U
 }
 \]
 The space $\D^i U \oo \bw^i\Sym^{g-1} U$ has a basis consisting of elements of the form $x^{(t)} \oo s_{\ll}(x)$,  where
 $0\leq t\leq i$ and $\ll\in\mc{P}_i(g-i)$.
 It follows from (\ref{eq:def-delJ}) that
 \begin{equation}\label{eq:top-row-dJ}
 \begin{aligned}
 \pd^J_i (x^{(t)} \oo s_{\ll}(x)) & = \Bigl(\sum_{j=1}^i (-1)^{j-1} x^{(t)} \oo s_{\ll^{\hat{j}}}(x) \oo x^{\ll_j+i-j}\Bigr)+ \Bigl(\sum_{j=1}^i (-1)^{j-1} x^{(t-1)} \oo s_{\ll^{\hat{j}}}(x) \oo x^{\ll_j+i-j+1}\Bigr), \\
 \end{aligned}
 \end{equation}
where the second term vanishes identically when $t=0$, while the first one disappears when $t=i$.
Using (\ref{eq:nud-formula}) and the fact that $\pd^K_i = k_i^g$ as defined in (\ref{eq:dkos-slam}) we get
 \begin{equation}\label{eq:dJ-circ-pi}
 (\pd^K_i \circ \ol{p}_i)(x^{(t)} \oo s_{\ll}(x)) = \sum_{j=1}^i (-1)^{j-1}\cdot\Bigl(\sum_{I \in {[i]\choose t}} s_{(\ll+(1^I))^{\hat{j}}}(x) \oo x^{(\ll+(1^I))_j+i-j}\Bigr).
\end{equation}
Comparing (\ref{eq:top-row-dJ}) with (\ref{eq:dJ-circ-pi}), it follows that in order to prove the commutativity of the diagram it suffices to check that for fixed $j\in\{1,\ldots,i\}$ we have
\begin{equation}\label{eq:one-and-two}
\begin{aligned}
\ol{p}_{i-1}\bigl(x^{(t)} \oo s_{\ll^{\hat{j}}}(x)\bigr) &= \sum_{j\notin I \in {[i] \choose t}} s_{(\ll+(1^I))^{\hat{j}}}(x)\mbox{ and } \\
\ol{p}_{i-1}\bigl(x^{(t-1)} \oo s_{\ll^{\hat{j}}}(x)\bigr) &= \sum_{j\in I \in {[i] \choose t}} s_{(\ll+(1^I))^{\hat{j}}}(x).
\end{aligned}
\end{equation}
Note that $\ol{p}_{i-1}\bigl(x^{(t)} \oo s_{\ll^{\hat{j}}}(x)\bigr)  = \sum_{I' \in {[i-1]\choose t}} s_{\ll^{\hat{j}} + (1^{I'})}(x)$
and that the correspondence
\begin{equation}\label{eq:I=I'}
I' \lra I = \{ k : k \in I' \mbox{ and }k < j\} \cup \{ k+1 : k\in I' \mbox{ and }k\geq j\}
\end{equation}
establishes a bijection between the collection of sets $I'\in {[i-1]\choose t}$ and that of sets $I \in {[i] \choose t}$ with $j\notin I$, which has the property that $(\ll+(1^I))^{\hat{j}} = \ll^{\hat{j}} + (1^{I'})$. This proves the first equality in (\ref{eq:one-and-two}). For the second equality we use
\[\ol{p}_{i-1}(x^{(t-1)} \oo s_{\ll^{\hat{j}}}(x))  = \sum_{I' \in {[i-1]\choose t-1}} s_{\ll^{\hat{j}} + (1^{I'})}(x)\]
and note the bijective correspondence between sets $I'\in {[i-1]\choose t}$ and $I \in {[i] \choose t}$ with $j\in I$, given by
\begin{equation}\label{eq:I=I'+j}
 I' \lra I = \{ k : k \in I' \mbox{ and }k < j\} \cup \{j\} \cup \{ k+1 : k\in I' \mbox{ and }k\geq j\}.
\end{equation}
Moreover, this correspondence has the property that $(\ll+(1^I))^{\hat{j}} = \ll^{\hat{j}} + (1^{I'})$, which shows the second equality in (\ref{eq:one-and-two}) and concludes our proof.
\end{proof}

\begin{corollary}
\label{cor:C_i=ker-pi+1}
 Let $p_{\bullet}:J_{\bullet}\lra K_{\bullet}$ be the map of complexes constructed in Proposition~\ref{prop:map-J-to-K}, and let $C_{\bullet}$ denote the minimal resolution of the module $C$ in (\ref{eq:defC}). We have
 \begin{equation}\label{eq:Ci=kerpi+1}
 C_i = \ker(p_{i+1}), \mbox{ for  }i=0,\ldots,g-2.
\end{equation}
\end{corollary}

\begin{proof}
Since $H_0(p_{\bullet}) : H_0(J_{\bullet}) \lra H_0(K_{\bullet})$ is an isomorphism, it follows that by applying the mapping cone construction to the map of complexes $p_{\bullet}:J_{\bullet}\lra K_{\bullet}$ and minimizing the resulting complex as in Lemma~\ref{lem:minimization}, we obtain a minimal complex $M_{\bullet}$ whose terms are
\[M_i = \ker(p_{i-1}) \oplus \coker(p_i),\mbox{ for }i=0,\ldots,g+1,\]
and whose only non-zero homology is $H_2(M_{\bullet}) = H_1(J_{\bullet}) = C$. Since the map $p_0$ is an isomorphism and $p_1$ is surjective, we get that $M_0=M_1=0$ so $M_{\bullet+2}$ is in fact a minimal resolution of $C$, from which we deduce that $M_{i+2} \cong C_{i}$, for  $i=0,\ldots,g-1$.

\vskip 4pt

We have that $C_i$ is a free module with generators of degree $i+1$, so the same must be true about $M_{i+2}$. Since $\ker(p_{i+1})$ is a free module generated in degree $i+1$ and $\coker(p_{i+2})$ is a free module generated in degree $i+2$, it follows that $\coker(p_{i+2})=0$ for all $i=0,\ldots,g-2$ and therefore $p_i$ is surjective for all $i=0,\ldots,g$. It follows that $C_i = M_{i+2} = \ker(p_{i+1})$, for  $i=0,\ldots,g-2$.
\end{proof}

\begin{rmk}
Keeping the notation above, we have an isomorphism $C_{g-1}\cong M_{g+1}$ of free modules of rank one generated in degree $g+1$, and we conclude that $\ker(p_g) = 0$ and $M_{g+1}=\coker(p_{g+1}) = K_{g+1}$.
\end{rmk}

\subsection{The first linear strand of the resolution of $R$}
\label{sec:resln-R}

The goal of this section is to complete the proof of Theorem~\ref{cor=kw}. Our strategy is summarized in the following steps:

\vskip 3pt

\noindent (1) We construct a map of complexes $\tl{q}_{\bullet}:F_{\bullet} \lra C_{\bullet}$ such that the induced map
 $$H_0(\tl{q}_{\bullet}):H_0(F_{\bullet}) \lra H_0(C_{\bullet})$$ is a surjection $\tl{R} \lra C$.

\noindent (2) If we let $R':=\mbox{Ker}(\widetilde{R}\rightarrow C)$, then $R\subseteq R'$ and the mapping cone of $\tl{q}_{\bullet}$ gives a resolution of $R'$. Minimizing it, we get that the terms in the first linear strand of the minimal resolution of $R'$ can be described as kernels of the Koszul maps (\ref{eq:ker-delta-2}).

\noindent (3) The quotient $R'/R$ is a module generated in degree at least $2$, so the first strands of the minimal resolutions of $R$ and $R'$ coincide.


\begin{proposition}\label{prop:map-F-to-J}
 For $i = 0,\ldots,g-2$, we consider the $\sl_2$-equivariant map
\[\ol{q}_i : \D^{2i}U \oo \bw^{i+1} \Sym^{g-2}U \lra \D^{i+1}U \oo \bw^{i+1}\Sym^{g-1}U\]
induced by the map $\Delta_1 : \D^{2i}U \lra \D^{i+1} U \oo \D^{i+1} U$
in Lemma~\ref{lem:mu1-delta1}, followed by
\[\nu^{i+1}_{g-2-i} : \D^{i+1}U \oo \bw^{i+1}\Sym^{g-2}U \overset{(\ref{eq:nud-formula})}{\lra}  \bw^{i+1}\Sym^{g-1}U.\]
The maps $\ol{q}_i$ define a map of complexes $q_{\bullet}:F_{\bullet} \lra J_{\bullet + 1}$ (if we take $q_0$ to send the summand $S\subseteq F_0$ to zero), such that the induced map $H_0(F_{\bullet}) \lra H_1(J_{\bullet})$ is a surjection $\tl{R} \lra C$ containing $R$ in the kernel.
\end{proposition}

\begin{proof}
The map $\Delta_1:\D^0U \lra \D^1 U\oo \D^1 U$ can be identified with the  $\sl_2$-equivariant inclusion of $\kk\cong \bw^2 U \lra U\oo U$, so the map $q_0:F_0\lra J_1$ is given by
 \[S \oplus \Sym^{g-2}U \oo S(-1) \overset{q_0}{\lra} (U \oo \Sym^{g-1}U) \oo S(-1),\]
 where $S$ is sent to $0$, and $\Sym^{g-2}U\cong \bw^2 U \oo \Sym^{g-2}U \lra U \oo \Sym^{g-1}U$ is an inclusion given by the Koszul differential. Since $\pd^J_1:J_1\lra J_0$ is induced by the multiplication map $U \oo \Sym^{g-1}U\lra \Sym^g U$, it follows that $\pd^J_1 \circ q_0 = 0$.
Assuming  $q_{\bullet}:F_{\bullet} \lra J_{\bullet + 1}$ is indeed a map of complexes (which we shall prove shortly), to see that $H_0(q_{\bullet})$ induces a surjection $\tl{R}\onto C$ containing $R$ in its kernel, it suffices to note that $q_0$ sends the generators $\Sym^{g-2}U\subseteq \tl{R}$ isomorphically onto the generators $\Sym^{g-2}U$ of $C$, and that the generator $1\in R$ is sent to $0$ since $q_0(S)=0$.

\vskip 4pt

 To prove that $q_{\bullet}:F_{\bullet} \lra J_{\bullet + 1}$ is a map of complexes we have to verify that for every $i \geq 1$ we get a commutative diagram (note that since $q_0(S)=0$, we can ignore the quadratic part of the differential $\pd^F_1$)
 \[
 \xymatrix@C=1.6em{
  \D^{2i}U \oo \bw^{i+1}\Sym^{g-2}U \ar[rr]^{\ol{q}_i} \ar[d]_{\pd^F_i} & & \D^{i+1}U \oo \bw^{i+1}\Sym^{g-1}U \ar[d]^{\pd^J_{i+1}} \\
  \left(\D^{2i-2}U \oo \bw^i\Sym^{g-2}U\right) \oo \Sym^g U \ar[rr]^{\ol{q}_{i-1} \oo \id} & & \left(\D^i U \oo \bw^i\Sym^{g-1}U\right) \oo \Sym^g U\\
 }
 \]
 The space $\D^{2i}U \oo \bw^{i+1}\Sym^{g-2}U$ has a basis consisting of elements $x^{(t)} \oo s_{\ll}(x)$ where $0\leq t\leq 2i$ and $\ll \in \mc{P}_{i+1}(g-i-2)$. Since
 \[\Delta_1(x^{(t)}) = \sum_{\substack{t'+t'' = t+1 \\ t',t'' \leq i+1}} (t' - t'') \cdot x^{(t')} \oo x^{(t'')}\]
it follows that
\[\ol{q}_i(x^{(t)} \oo s_{\ll}(x)) = \sum_{\substack{t'+t'' = t+1 \\  t',t'' \leq i+1}} (t' - t'') \cdot x^{(t')} \oo \Bigl(\sum_{I \in {[i+1] \choose t''}} s_{\ll + (1^I)}(x) \Bigr) \]
and using (\ref{eq:def-delJ}) we get that $(\pd^J_{i+1}\circ \ol{q}_i)(x^{(t)} \oo s_{\ll}(x))$ is computed by
\[\sum_{\substack{t'+t'' = t+1 \\ t',t'' \leq i+1}} \sum_{I \in {[i+1] \choose t''}}\sum_{\substack{u'+s' = t' \\ u'\leq 1,s'\leq i}} \sum_{j=1}^{i+1} (-1)^{j-1}\cdot (t' - t'') \cdot x^{(s')} \oo s_{(\ll + (1^I))^{\hat{j}}}(x) \oo x^{(\ll + (1^I))_j + i + 1 - j + u'} \]
\begin{equation}\label{eq:1st-comp}
=\sum_{\substack{u'+s'+t'' = t+1 \\ u'\leq 1, s'\leq i,t'' \leq i+1}} \sum_{I \in {[i+1] \choose t''}} \sum_{j=1}^{i+1} (-1)^{j-1}\cdot (u' + s' - t'') \cdot x^{(s')} \oo s_{(\ll + (1^I))^{\hat{j}}}(x) \oo x^{(\ll + (1^I))_j + i + 1 - j + u'}.
\end{equation}

On the other hand we have using (\ref{eq:def-delF}) that
\[ \pd^F_i(x^{(t)} \oo s_{\ll}(x)) = \sum_{\substack{u+s = t \\ u\leq 2,\ s\leq 2i-2}} \sum_{j=1}^{i+1} (-1)^{j-1} \cdot x^{(s)} \oo s_{\ll^{\hat{j}}}(x) \oo (x^{(u)} \cdot x^{\ll_j + i + 1 - j}),\]
where $x^{(u)} = x^u$ when $u=0, 2$, and $x^{(1)} = 2x$. Thus  $\bigl((\ol{q}_{i-1} \oo \id_{\Sym^g U}) \circ \pd^F_i\bigr)(x^{(t)} \oo s_{\ll}(x))$ is computed by
\[\sum_{\substack{u+s = t \\ u\leq 2,\ s\leq 2i-2}} \sum_{j=1}^{i+1} \sum_{\substack{s'+s'' = s+1 \\  s',s'' \leq i}} \sum_{I' \in {[i] \choose s''}} (-1)^{j-1}\cdot (s' - s'') \cdot x^{(s')} \oo s_{\ll^{\hat{j}}+(1^{I'})}(x) \oo (x^{(u)} \cdot x^{\ll_j + i + 1 - j})\]
\begin{equation}\label{eq:2nd-comp}
=\sum_{\substack{u+s'+s'' = t + 1 \\ u\leq 2,\ s',s''\leq i}} \sum_{j=1}^{i+1} \sum_{I' \in {[i] \choose s''}} (-1)^{j-1}\cdot (s' - s'') \cdot x^{(s')} \oo s_{\ll^{\hat{j}}+(1^{I'})}(x) \oo (x^{(u)} \cdot x^{\ll_j + i + 1 - j})
\end{equation}

\vskip 4pt

To conclude the proof we show that we can identify (\ref{eq:1st-comp}) with (\ref{eq:2nd-comp}). We fix $j$ and $s'$ with $1\leq j\leq i+1$ and $0\leq s' \leq i$. It is then enough to verify that
\begin{equation}\label{eq:LHS}
\sum_{\substack{u'+t'' = t+1-s' \\ u'\leq 1,t'' \leq i+1}} \sum_{I \in {[i+1] \choose t''}}  (u' + s' - t'') \cdot s_{(\ll + (1^I))^{\hat{j}}}(x) \oo x^{(\ll + (1^I))_j + i + 1 - j + u'}
\end{equation}
\[ = \sum_{\substack{u+s'' = t + 1 - s' \\ u\leq 2,s''\leq i}} \sum_{I' \in {[i] \choose s''}} (s' - s'') \cdot s_{\ll^{\hat{j}}+(1^{I'})}(x) \oo (x^{(u)} \cdot x^{\ll_j + i + 1 - j}).\]
We fix $u,s''$ with $u+s'' = t + 1 - s'$, $u\leq 2$, $s''\leq i$, and fix $I' \in {[i] \choose s''}$. We consider three separate cases:

\vskip 4pt

\ul{$u=0$}: We let $u'=0$, $t''=s''$, and define $I$ via (\ref{eq:I=I'}). The term in (\ref{eq:LHS}) corresponding to this choice of parameters is equal to $(s' - s'') \cdot s_{\ll^{\hat{j}}+(1^{I'})}(x) \oo (x^{(u)} \cdot x^{\ll_j + i + 1 - j})$ since $u'-t'' = -s''$, $(\ll + (1^I))^{\hat{j}} = \ll^{\hat{j}}+(1^{I'})$, $x^{(u)} = 1$ and $(\ll + (1^I))_j + i + 1 - j + u' = \ll_j + i + 1 - j$, where the last equality follows from the fact that $j\not\in I$ and $u'=0$.

\vskip 4pt

\ul{$u=2$}: We let $u'=1$, $t'' = s'' + 1$, and define $I$ via (\ref{eq:I=I'+j}). The term in (\ref{eq:LHS}) corresponding to this choice of parameters is equal to $(s' - s'') \cdot s_{\ll^{\hat{j}}+(1^{I'})}(x) \oo (x^{(u)} \cdot x^{\ll_j + i + 1 - j})$ since $u'-t'' = -s''$, $(\ll + (1^I))^{\hat{j}} = \ll^{\hat{j}}+(1^{I'})$, $x^{(u)} = x^2$ and $(\ll + (1^I))_j + i + 1 - j + u' = 2 + \ll_j + i + 1 - j$, where the last equality follows from the fact that $j\in I$ and $u'=1$.

\vskip 4pt

\ul{$u=1$}: In this case there are two terms in (\ref{eq:LHS}) whose sum is
\[(s' - s'') \cdot s_{\ll^{\hat{j}}+(1^{I'})}(x) \oo (x^{(u)} \cdot x^{\ll_j + i + 1 - j}) = 2\cdot(s' - s'') \cdot s_{\ll^{\hat{j}}+(1^{I'})}(x) \oo x^{\ll_j + i + 2 - j}\]
and they correspond to the following choice of parameters. If we let $u'=0$, $t'' = s''+1$, and define $I$ via (\ref{eq:I=I'+j}) then we get
\[(u' + s' - t'') \cdot s_{(\ll + (1^I))^{\hat{j}}}(x) \oo x^{(\ll + (1^I))_j + i + 1 - j + u'} = (s'-s''-1) \cdot s_{\ll^{\hat{j}}+(1^{I'})}(x) \oo x^{\ll_j + i + 2 - j},\]
while if we take $u'=1$, $t''=s''$, and define $I$ via (\ref{eq:I=I'}) then we get
\[(u' + s' - t'') \cdot s_{(\ll + (1^I))^{\hat{j}}}(x) \oo x^{(\ll + (1^I))_j + i + 1 - j + u'} = (1+s'-s'') \cdot s_{\ll^{\hat{j}}+(1^{I'})}(x) \oo x^{\ll_j + i + 2 - j}.\]

To finish the proof, we make sure that all the terms in (\ref{eq:LHS}) are accounted for in the above case analysis:
\begin{itemize}
 \item If $u'=0$ and $j\in I$ then take $u=1$, $s''=t''-1$ and define $I'$ by reversing (\ref{eq:I=I'+j}).
 \item If $u'=0$ and $j\not\in I$ then take $u=0$, $s''=t''$ and define $I'$ by reversing (\ref{eq:I=I'}).
 \item If $u'=1$ and $j\in I$ then take $u=2$, $s''=t''-1$ and define $I'$ by reversing (\ref{eq:I=I'+j}).
 \item If $u'=1$ and $j\not\in I$ then take $u=1$, $s''=t''$ and define $I'$ by reversing (\ref{eq:I=I'}).\qedhere
\end{itemize}
\end{proof}

The map $q_{\bullet}$ constructed in Proposition~\ref{prop:map-F-to-J} induces a map of complexes $F_{\bullet}\lra C_{\bullet}$ as follows

\begin{corollary}\label{cor:F-to-C}
 For $i=0,\ldots,g-2$ we have $p_{i+1} \circ q_i = 0$, therefore $q_i(F_i) \subseteq C_i = \ker(p_{i+1})$, by applying Corollary~\ref{cor:C_i=ker-pi+1}. The induced map of complexes $\tl{q}_{\bullet}:F_{\bullet}\lra C_{\bullet}$ has the property that
 \[H_0(\tl{q}_{\bullet}) : H_0(F_{\bullet}) =\tl{R} \lra H_0(C_{\bullet}) = C\mbox{ is a surjective map.}\]
\end{corollary}

\begin{proof}
 If we replace $i$ by $i+1$, let $d=g-2-i$, and tensor the commutative diagram~(\ref{eq:sym-to-wedge}) with $\D^{i+1}U$ then we obtain a commutative diagram
\[
 \xymatrix{
  (\D^{i+1}U)^{\oo 2} \oo \Sym^{g-2-i}(\D^{i+1} U) \ar[rr]_(.52){\id^{\oo 2}\oo \psi_{g-2-i}} \ar[d] & & (\D^{i+1}U)^{\oo 2}  \oo \bw^{i+1}(\Sym^{g-2} U) \ar[d]^{\id\oo \nu^{i+1}_{g-2-i}} \\
  \D^{i+1}U \oo \Sym^{g-1-i}(\D^{i+1} U) \ar[rr]_{\id\oo \psi_{g-1-i}} & & \D^{i+1}U \oo \bw^{i+1}(\Sym^{g-1} U) \\
}
 \]
Restricting the top row of the diagram above along the chain of maps
 \[ \D^{2i}U \overset{\Delta_1}{\lra} \bw^2 \D^{i+1}U \subseteq \bigl(\D^{i+1}U\bigr)^{\oo 2} \]
 and recalling that $q_i$ was defined by restricting along the same maps, we obtain the commutativity of the top square in the diagram
\begin{equation}\label{eq:big-sym-to-wedge}
\begin{aligned}
 \xymatrix{
  \D^{2i}U \oo \Sym^{g-2-i}(\D^{i+1} U) \ar[rr]_{\id\oo \psi_{g-2-i}} \ar[d]_{\delta_2} & & \D^{2i}U \oo \bw^{i+1}(\Sym^{g-2} U) \ar[d]^{\ol{q}_i} \\
  \D^{i+1}U \oo \Sym^{g-1-i}(\D^{i+1} U) \ar[d]_{\delta_1} \ar[rr]_{\id\oo \psi_{g-1-i}} & & \D^{i+1}U \oo \bw^{i+1}(\Sym^{g-1} U) \ar[d]^(.6){\ol{p}_{i+1}} \\
  \Sym^{g-i}(\D^{i+1} U) \ar[rr]_{\psi_{g-i}} & & \bw^{i+1}(\Sym^{g} U) \\
}
 \end{aligned}
 \end{equation}
Note that the bottom square of this diagram also commutes, being another instance of~(\ref{eq:sym-to-wedge}), namely that where we replace $i$ by $i+1$, let $d=g-1-i$, and recall that $\ol{p}_{i+1}=\nu_{g-i-1}^{i+1}$.

In the diagram \eqref{eq:big-sym-to-wedge} the maps in the left column are induced by Koszul differentials, so their composition must be zero. Since the horizontal maps are isomorphisms the maps in the right column must compose to zero as well, hence $\ol{p}_{i+1} \circ \ol{q}_i=0$ and therefore $p_{i+1}\circ q_i=0$. Using the fact that $C_i = \ker(p_{i+1})$ (see \eqref{eq:Ci=kerpi+1}), it follows that $q_{\bullet}$ induces (by restricting the range) a map of complexes $\tl{q}_{\bullet}:F_{\bullet}\lra C_{\bullet}$. Just as in the proof of Proposition~\ref{prop:map-F-to-J} it follows by inspecting the generators that $H_0(\tl{q}_{\bullet})$ is surjective.
\end{proof}

\begin{proof}[Proof of Theorem~\ref{cor=kw}]
We  consider the minimization $N_{\bullet}$ of the mapping cone of $\tl{q}_{\bullet}$, whose terms are computed by Lemma~\ref{lem:minimization} via
\begin{equation}\label{eq:Ni=ker-coker}
N_i = \ker(\tl{q}_{i-1}) \oplus \coker(\tl{q}_i).
\end{equation}
We have that $N_0=0 $  and  $N_1 = S \oplus \coker(\tl{q}_1)$,
where $\coker(\tl{q}_1)$ is a free module generated in degree $2$ (if non-zero). From the long exact sequence for the homology of the mapping cone, we obtain that $N_{\bullet}$ has only one non-zero homology group, namely $H_1(N_{\bullet}) = \ker\bigl\{H_0(\tl{q}_{\bullet}) : H_0(F_{\bullet}) \lra H_0(C_{\bullet})\bigr\}=R'\supseteq  R$.
It follows that $N_{\bullet}$ is quasi-isomorphic to the minimal free resolution of $R'$, shifted by~$1$. Since the summand $S\subseteq N_1$ maps onto $R$, the quotient $R'/R$ is generated in degree $2$ (if non-zero), in particular
\begin{equation}\label{eq:TorR=R'}
\Tor_i^S(R,\kk)_{i+1} = \Tor_i^S(R',\kk)_{i+1} \mbox{ for all }i\geq 0.
\end{equation}
Since $\ker(\tl{q}_i)$ is a free module generated in degree $i+1$ for $i\geq 1$, and $\coker(\tl{q}_i)$ is a free module generated in degree $i+1$ for $1\leq i\leq g-2$, and in degree $g+1$ when $i=g-1$, we conclude that
\begin{equation}\label{eq:TorR'=ker-tlq}
\Tor_i^S(R',\kk)_{i+1} = \Tor_0^S(N_{i+1},\kk)_{i+1} = \ker(\tl{q}_i)_{i+1}, \mbox{ for all }i\geq 0.
\end{equation}
 Recall from Proposition~\ref{prop:map-J-to-K} and (\ref{eq:Ci=kerpi+1}) that $C_i$ is a direct summand of $J_{i+1}$. From Proposition~\ref{prop:map-F-to-J} and Corollary~\ref{cor:F-to-C} we recall that the map of complexes $\tl{q}_{\bullet}:F_{\bullet}\lra C_{\bullet}$ is induced by a map of complexes $q_{\bullet}:F_{\bullet}\lra J_{\bullet+1}$. It follows that $\ker(\tl{q}_i)=\ker(q_i)$ and in particular
\begin{equation}\label{eq:ker-tlq=ker-q}
\ker(\tl{q}_i)_{i+1} = \ker\Bigr\{\D^{2i}U \oo \bw^{i+1}(\Sym^{g-2} U) \overset{\ol{q}_i}{\lra} \D^{i+1}U \oo \bw^{i+1}(\Sym^{g-1} U)\Bigr\},
\end{equation}
where the map $\ol{q}_i$ was defined in Proposition~\ref{prop:map-F-to-J}. Using the commutative diagram
\[
 \xymatrix@C=3em{
  \D^{2i}U \oo \Sym^{g-2-i}(\D^{i+1} U) \ar[r]_(0.52){\id\oo \psi_{g-2-i}} \ar[d]_{\delta_2} & \D^{2i}U \oo \bw^{i+1}(\Sym^{g-2} U)  \ar[d]^{\ol{q}_i} \\
  \D^{i+1}U \oo \Sym^{g-1-i}(\D^{i+1} U)  \ar[r]_(0.52){\id\oo \psi_{g-1-i}} & \D^{i+1}U \oo \bw^{i+1}(\Sym^{g-1} U)  \\
}
 \]
from the proof of Corollary~\ref{cor:F-to-C}, together with (\ref{eq:TorR=R'}), (\ref{eq:TorR'=ker-tlq}) and (\ref{eq:ker-tlq=ker-q}) we obtain
\[
 K_{i,1}\bigl(\mc{T}, \OO_{\mc{T}}(1)\bigr) = \Tor_i^S(R,\kk)_{i+1}= \Tor_i^S(R',\kk)_{i+1}= \ker(\ol{q}_i) = \ker(\delta_2),
\]
concluding the proof of the theorem.
\end{proof}

\subsection{The minimal resolution of $R$}
\label{subsec:char-not-2}

Throughout this section we assume that $\chr(\kk)\neq 2$, and we analyze more carefully the constructions from the previous sections in order to prove Theorems~\ref{thm:shape-res-R} and~\ref{thm:Ki2=Wi+2}. The main observation is the following.

\begin{lemma}\label{lem:coker-tl-q}
 For $i=1,\ldots,g-2$, we have that $\coker(\tl{q}_i)$ is isomorphic to $W^{(i+1)}_{g-2-i} \oo_\kk S(-i-1)$.
\end{lemma}

\begin{proof}
 Since $F_i$ and $C_i$ are free $S$-modules generated in degree $i+1$ for $1\leq i\leq g-2$, and since $\tl{q}_i$ is a homogeneous map, it follows that $\coker(\tl{q}_i)$ is a free module generated in degree $i+1$. To identify the minimal generators of $\coker(\tl{q}_i)$, we recall that $\tl{q}_i$ is induced by the map $q_i$, whose image lands inside $C_i=\ker(p_{i+1})$ by Corollary~\ref{cor:F-to-C}. It follows that the minimal generators of $\coker(\tl{q}_i)$ are described as
 \[\frac{\ker(\ol{p}_{i+1})}{\Im(\ol{q}_i)} \overset{(\ref{eq:big-sym-to-wedge})}{\cong} \frac{\ker(\delta_1)}{\Im(\delta_2)} = W^{(i+1)}_{g-2-i}.\qedhere\]
\end{proof}

Using this observation, we can now deduce Theorems~\ref{thm:shape-res-R} and~\ref{thm:Ki2=Wi+2}. We start with:

\begin{corollary}\label{cor:R=R'} We have $R'=R$ and Theorem~\ref{thm:Ki2=Wi+2} holds.
\end{corollary}

\begin{proof}
The complex $N_{\bullet+1}$ gives the minimal free resolution of~$R'$. Since $R\subseteq R'$ is the image of the summand $S$ of $N_1$ in (\ref{eq:Ni=ker-coker}), in order to prove that $R'=R$, it suffices to check that $\coker(\tl{q}_1)=0$, which by Lemma~\ref{lem:coker-tl-q} is equivalent to the vanishing $W^{(2)}_{g-3}=0$. It is then enough to prove that $W^{(2)}$ is identically zero, which is equivalent to the fact that the inclusion $\Delta_1:\D^2 U \hookrightarrow \bw^2(\D^2 U)$ is an equality. Since $\dim(\D^2 U) = 3$, the source and target of $\Delta_1$ have both dimension three, hence $\Delta_1$ is indeed an equality.

Since the complex $N_{\bullet+1}$ gives the minimal free resolution of $R'=R$, we obtain for $i=1,\cdots,g-3$
\[ K_{i,2}\bigl(\mc{T},\OO_{\mc{T}}(1)\bigr) = \Tor_i^S(R,\kk)_{i+2} = (N_{i+1} \oo_S \kk)_{i+2} \overset{(\ref{eq:Ni=ker-coker})}{=} \coker(\tl{q}_{i+1}) \oo_S \kk = W^{(i+2)}_{g-3-i}.\qedhere\]
\end{proof}

We end this section by verifying Theorem~\ref{thm:shape-res-R}.

\begin{proof}[Proof of Theorem~\ref{thm:shape-res-R}]
 We first prove that $R$ is Cohen--Macaulay (see also \cite[Proposition~6.1]{Sch0} for a different proof in characteristic zero). Since $\dim(R)=\dim(\mc{T})+1=3$, and $\dim(S)=g+1$, it suffices to check that $R$ has projective dimension (at most) $g-2$ as an $S$-module. Since $N_{\bullet+1}$ is the minimal resolution of $R$, this is equivalent to the fact that $N_i = 0$ for $i\geq g$. We obtain this conclusion using (\ref{eq:Ni=ker-coker}), since the source of $\tl{q}_{i-1}$ is $F_{i-1} = 0$ for $i\geq g$, and the target of $\tl{q}_i$ is $C_i=0$ for $i\geq g$.

To prove that $R$ is Gorenstein of regularity $3$, it is then enough to check that
\[ \Tor_{g-2}^S(R,\kk) = N_{g-1} \oo_S \kk\]
is one-dimensional, concentrated in degree $g+1$. Since $\coker(\tl{q}_{g-1}) = C_{g-1} \cong S(-g-1)$, it follows from (\ref{eq:Ni=ker-coker}) that it is enough to check that $\tl{q}_{g-2}$ is injective, or equivalently, that $\ol{q}_{g-2}$ is injective. Using the commutativity of the top square in (\ref{eq:big-sym-to-wedge}) with $i=g-2$, this is further equivalent to the injectivity of
\[ \D^{2(g-2)}U \overset{\delta_2}{\lra} \D^{g-1}U \oo \D^{g-1}U.\]
Since $\delta_2$ is given in this case by the map $\Delta_1$ in Lemma~\ref{lem:mu1-delta1} with $a=g-1$, the conclusion follows.
\end{proof}

We record one more consequence of the previous results:

\begin{corollary}\label{cor:omega-T}
 If $\chr(\kk)\neq 2$, the dualizing sheaf $\omega_{\mc{T}}$ is trivial, and $H^1(\mc{T},\OO_{\mc{T}}(d)) = 0$ for all $d$.
\end{corollary}

\begin{proof}
 It follows from Theorem~\ref{thm:shape-res-R} that $\op{Ext}^{g-2}_S(R,S) = R(g+1)$. By sheafification this implies that
$ \ShExt^{\ g-2}_{\PP^g}(\OO_{\mc{T}},\OO_{\PP^g}) = \OO_{\mc{T}}(g+1)$.
Using  $\omega_{\PP^g} \cong \OO_{\PP^g}(-g-1)$ and \cite[Proposition~III.7.5]{hartshorne}, it follows that
\[ \omega_{\mc{T}} = \ShExt^{\ g-2}_{\PP^g}(\OO_{\mc{T}},\omega_{\PP^g}) = \ShExt^{\ g-2}_{\PP^g}(\OO_{\mc{T}},\OO_{\PP^g})(-g-1) = \OO_{\mc{T}},\]
as desired. The vanishing $H^1(\mc{T},\OO_{\mc{T}}(d)) = 0$ follows from the fact that $\mc{T}$ is aritmetically Cohen--Macaulay (combine for instance \cite[Propositions~A1.11 and~A1.16]{E04}).
\end{proof}

One consequence of the Cohen--Macaulay property of $R$ is the isomorphism $R\cong \Gamma_{\mc{T}}(\OO_{\mc{T}}(1))$, which amounts to $\Gamma_{\mc{T}}(\OO_{\mc{T}}(1))$ being generated in degree one.

\subsection{The resolution of $R$ in characteristic $2$}
\label{subsec:char2}

The goal of this section is to prove that (unlike in the general case) the resolution of $R$ consists of a single linear strand when $\op{char}(\kk)=2$, which is obtained as a subcomplex of $F_{\bullet}$ as explained below. We identify $S=\kk[z_0,\ldots,z_g]$, where $z_i$ is dual to $x^{(i)}$, and claim that the $2\times 2$ minors of the matrix
\[
M=\begin{bmatrix}
z_0 & z_1 & \cdots & z_{g-2} \\
z_2 & z_3 & \cdots & z_g
\end{bmatrix}
\]
vanish on $\mc{T}$. Since $\mc{T}$ is irreducible, it is enough to check the vanishing on some affine chart. Choosing the chart as in (\ref{eq:chart-T}) and evaluating $M$ on an element $f\in\mc{T}$ we obtain
\[
M(f) = u\cdot \begin{bmatrix}
1 & t & t^2 & t^3 & \cdots & t^{g-2} \\
t^2 & t^3 & t^4 & t^5 & \cdots & t^g
\end{bmatrix}
+ v\cdot \begin{bmatrix}
0 & t & 0 & t^3 & \cdots & (g-2)\cdot t^{g-2} \\
0 & t^3 & 0 & t^5 & \cdots & g\cdot t^g
\end{bmatrix}
\]
for some scalars $u,v,t\in\kk$. Since the second row of $M(f)$ is obtained from the first by multiplying by $t^2$,  the $2\times 2$ minors of $M(f)$ vanish on $\mc{T}$, as desired.  The $2\times 2$ minors of $M$ define a rational normal scroll of dimension~$2$ and degree $g-1$, containing $\mc{T}$, see \cite[Section~A2H]{E04}. Carrying out in characteristic $2$ the local analysis in Section~\ref{subsec:pparts}, we obtain that $\mbox{deg}(\mc{T})=g-1$, hence $\mc{T}$ equals the scroll. We conclude that the minimal resolution of $R$ is given by an Eagon--Northcott complex.


To identify this as a subcomplex of $F_{\bullet}$, we let for $i\geq 1$
\begin{equation}\label{eq:D2i-odd}
\D^{2i}_{\op{odd}}U = \bigoplus_{1\leq t\leq i} \kk\cdot x^{(2t-1)} = \{f\in\D^{2i}: \Lop\cdot f = \Rop\cdot f = 0 \}\subseteq \D^{2i}U.
\end{equation}
We define $F_0^{\op{odd}}=S$ and
\begin{equation}\label{eq:Fi-odd}
F_i^{\op{odd}} = \Bigl(\D^{2i}_{\op{odd}}U \oo \bw^{i+1}\Sym^{g-2}U\Bigr)\oo S(-i-1), \mbox{ for }i=1,\ldots,g-2.
\end{equation}
Since ${2\choose 1}=0$ in $\kk$, it follows from (\ref{eq:def-delF}) that $F_{\bullet}^{\op{odd}}$ determines a subcomplex of $F_{\bullet}$. Taking into account that $\rk(F_1^{\op{odd}})={g-1\choose 2} = \dim(I_2)$,
we get that that $F_1^{\op{odd}}$ surjects onto the quadratic part of the ideal $I$. Since $I$ is generated in degree $2$, we get that $H_0(F_{\bullet}^{\op{odd}})=R$. We can now define $\D^{2i}_{\op{even}}$ and $F_{\bullet}^{\op{even}}$ in analogy with (\ref{eq:D2i-odd}) and (\ref{eq:Fi-odd}), and observe that $F_{\bullet}^{\op{even}}$ is also a subcomplex of $F_{\bullet}$: it follows from (\ref{eq:def-delF}) that the linear part of the differentials $\pd_{\bullet}^F$ send $F_{\bullet}^{\op{even}}$ to itself, so we only need to worry about the quadratic part of $\pd_1^F$; since $F_1^{\op{odd}}$ already gives the minimal generators of $I$, the quadratic part of $\pd_1^F$ restricts to $0$ on $F_1^{\op{even}}$, showing that $F_{\bullet}^{\op{even}}$ is indeed a subcomplex. It follows that $F_{\bullet}=F_{\bullet}^{\op{odd}}\oplus F_{\bullet}^{\op{even}}$, and in particular $H_i(F_{\bullet}^{\op{odd}})=0$ for $i>0$, showing that $F_{\bullet}^{\op{odd}}$ is the minimal resolution of $R$.

\begin{remark}\label{rem:T-on-scroll}
 A similar local calculation  shows that if $3\leq \chr(\kk)=p\leq g$, then $\mc{T}$ is contained in a rational normal scroll of codimension $g-p$, defined by the vanishing of the $2\times 2$ minors of the matrix
\[
\begin{bmatrix}
z_0 & z_1 & \cdots & z_{g-p} \\
z_p & z_{p+1} & \cdots & z_g
\end{bmatrix}.
\]
The (linear) syzygies of the scroll are then going to be contained inside the linear syzygies of $\mc{T}$. Since the Eagon--Northcott resolution of the scroll has length $g-p$, we conclude that $K_{i,1}(\mathcal{T}, \OO_{\mathcal{T}}(1))\neq 0$ for $1\leq i\leq g-p$. In the range $3\leq p\leq \frac{g+1}{2}$ this non-vanishing is sharp: indeed, we have  $b_{i,1} = b_{g-2-i,2}$, so the conclusion $b_{i,1}\neq 0$ if and only if $1\leq i\leq g-p$ follows from (\ref{eq:ki2-not-0-psmall}).
\end{remark}


\section{Moduli of pseudo-stable curves and Green's Conjecture}
\label{sec:moduli-Green}

In this section we explain how Theorem~\ref{thm:vanishing-tangential} implies Green's Conjecture for general curves of genus~$g$, that is, Theorem~\ref{thm:genericGreen}. The key idea is to consider a moduli space of curves of genus $g$ which contains among its points the cuspidal curves that are hyperplane sections of the tangential variety $\mathcal{T}\subseteq \mathbb P^g$. We fix throughout an algebraically closed field $\kk$ of characteristic $p\neq 2,3$. This assumption excludes the cases $p=3$ and $g=3,4$ from Theorem~\ref{thm:genericGreen}, but Green's Conjecture can be verified directly in these cases (see also~\cite{Sch1} for a proof of Green's Conjecture in small genera).

\begin{defn}\label{defps}
A \emph{pseudo-stable} curve $C$ of genus $g$ is a  connected curve over $\kk$ of arithmetic genus $g$ with only nodes and cusps as singularities such that (i) its dualizing sheaf $\omega_C$  is ample and (ii) each irreducible component of arithmetic genus $1$ of $C$ intersects the rest of the curve in at least two points.
\end{defn}

Schubert \cite{Schu} constructs the coarse moduli space of pseudo-stable curves of genus $g$ as the Chow quotient of $3$-canonically embedded curves. The  construction and geometry of $\mmp$ is described in \cite{Schu} and \cite{HH}. In particular, one has a divisorial contraction
$$\pi:\mm_g\rightarrow  \mmp,$$
which at the level of geometric points replaces each elliptic tail of a stable curve with a cusp. More precisely, if $[C]\in \mm_g$ is a stable curve having elliptic tails $E_1, \ldots, E_{\ell}$ with $E_i\cap \overline{(C\setminus E_i)}=\{p_i\}$, for $i=1, \ldots, \ell$, then $[C'] := \pi([C])$ is described as follows: if we denote by $D$ the union of the components of $C$ different from the elliptic tails then the pseudo-stable curve $C'$ is birational to $D$, and there exists a morphism $\nu:C\rightarrow C'$ contracting each tail $E_i$ to a cusp $q_i\in C'$ (see \cite[Proposition~3.1]{HH} for details).

\vskip 3pt

The following result is hinted at in \cite{Sm}.

\begin{prop}
Pseudo-stable curves of genus $g\geq 3$ form a Delige-Mumford stack $\mmp$ over $\mathrm{Spec } \ \mathbb Z\bigl[\frac{1}{6}\bigr]$.
\end{prop}
\begin{proof}
To conclude that the algebraic stack of pseudo-stable curves of genus $g$ is of Deligne-Mumford type, it suffices to show that  a pseudo-stable curve $C$ of genus
$g\geq 3$ admits no infinitesimal deformations, that is, $H^0(C,\Omega_C^{\vee})=0$. We denote by $p_1, \ldots, p_{\ell}\in C$ the cusps
of $C$, by $\nu:D\rightarrow C$ the normalization map at the cusps and set $q_i:=\nu^{-1}(p_i)$, for $i=1, \ldots, \ell$. Using \cite[Proposition~2.3]{Sm}, we have $\Omega_C^{\vee}=\nu_*\bigl(\Omega_D^{\vee}(-q_1-\cdots-q_{\ell})\bigr)$, in particular $H^0(C,\Omega_C^{\vee})\cong H^0\bigl(D,\omega_D^{\vee}(-q_1-\cdots-q_{\ell})\bigr)$. Condition (ii) in Definition \ref{defps} implies that the $\ell$-pointed curve $[D,q_1, \ldots, q_{\ell}]$ is stable. Since infinitesimal deformations of $[D, p_1, \ldots, p_{\ell}]$ are parametrized by $H^0\bigl(D,\Omega_D^{\vee}(-q_1-\cdots-q_{\ell})\bigr)$ this space vanishes and the conclusion follows.
\end{proof}

\begin{rmk} In characteristic $2$ and $3$ the stack $\mmp$ is not Deligne-Mumford, for in these cases cuspidal curves may have infinitesimal automorphisms. Precisely, if $C$ is a pseudo-stable with a cusp at the point $q$ and $\nu:D\rightarrow C$ the normalization at $q$, then it is no longer the case that the inclusion $\Omega_{C,q}^{\vee}\subseteq \bigl(\nu_*\Omega_D\bigr)_q$ holds, that is, there exist regular vector fields on $C$ that do not extend to a regular vector field on $D$, as explained in \cite[Example~1]{Sm}.
\end{rmk}
\vskip 4pt

Assume $g\geq 3$ and let $\cM_g^{\sharp}$ be the open substack of $\mmp$ classifying irreducible non-hyperelliptic pseudo-stable curves of genus $g$. If $C$ is  a non-hyperelliptic  irreducible pseudo-stable curve, it follows from \cite[Theorem~1.6]{H} that~$\omega_C$ is very ample. We  show that the condition $K_{\lfloor \frac{g}{2} \rfloor,1}(C,\omega_C)\neq 0$ can be described as the degeneracy locus of a morphism of vector bundles on $\cM_g^{\sharp}$, and in particular that it is a closed condition. We will then use Theorem~\ref{thm:vanishing-tangential} to produce examples of curves in $\cM_g^{\sharp}$  for which $K_{\lfloor \frac{g}{2} \rfloor,1}(C,\omega_C)=0$, and conclude that this vanishing holds generically.

\vskip 3pt

Let $M_{\omega_C}$ be the kernel bundle associated with $(C,\omega_C)$, as defined in (\ref{lazb}) and  $\bb{P}_C := \bb{P}(H^0(C,\omega_C)^{\vee})\cong \bb{P}^{g-1}$ and denote by  $M_C$ the kernel bundle associated with $(\bb{P}_C,\mc{O}_{\bb{P}_C}(1))$. We recall from (\ref{koszint}) that
\begin{equation}\label{eq:ker-alpha-C}
K_{i+1,1}(C,\omega_C) = \ker\Bigl\{ H^0\bigl(\bb{P}_C, \bigwedge^i M_C (2)\bigr) \overset{\a_C}{\lra} H^0\bigl(C, \bigwedge^i M_{\omega_C}\otimes \omega_C^2\bigr) \Bigr\}
\end{equation}


\vskip 3pt

We construct vector bundles $\mathcal{A}$ and $\mathcal{B}$ on $\cM_g^{\sharp}$, whose fibres at $[C]\in\cM_g^{\sharp}$ are
\begin{equation}\label{eq:fiber-A-B}
\mathcal{A}([C])=H^0\Bigl(\bb{P}_C, \bigwedge^i M_C (2)\Bigr)\  \mbox{ and } \ \mathcal{B}([C])=H^0\Bigl(C, \bigwedge^i M_{\omega_C}\otimes \omega_C^2\Bigr),
\end{equation}
and a morphism $\a:\mc{A}\lra\mc{B}$ whose restriction to the fiber over $[C]$ is the map $\a_C$ in (\ref{eq:ker-alpha-C}). To construct~$\mc{B}$, we start with the following.

\begin{prop}\label{vanishing1}
 If $C$ is an irreducible pseudo-stable curve of genus $g$,  then $$H^1\Bigl(C, \bigwedge^i M_{\omega_C}\otimes \omega_C^{2}\Bigr)=0\ \mbox{ for } i<g-1.$$
\end{prop}

\proof
We have that  $\mbox{det}(M_{\omega_C})=\omega_C^{\vee}$,  which implies that
$\bw^i M_{\omega_C} \otimes \omega_C^{2} \cong \bw^{g-1-i} M_{\omega_C}^{\vee} \otimes \omega_C$.
Combining this with Serre duality, it follows that
\[H^1\Bigl(C, \bigwedge^i M_{\omega_C}\otimes \omega_C^{2}\Bigr) \cong H^0\Bigl(C, \bigwedge^{g-1-i} M_{\omega_C}\Bigr)^{\vee}.\]
Thinking of (\ref{koszint}) as a two-step resolution of the kernel bundle, it follows that $\bigwedge^{g-1-i} M_{\omega_C}$ is resolved by the $(g-1-i)$-th exterior power of this resolution, which is the Koszul-type complex
\[\bigwedge^{g-1-i} H^0(C,\omega_C) \oo \mc{O}_C \lra \bigwedge^{g-2-i} H^0(C,\omega_C) \oo \omega_C \lra \cdots\]
Using the left exactness of the global sections functor, we get
\[H^0\Bigl(C, \bigwedge^{g-1-i} M_{\omega_C}\Bigr) = \ker\Bigl\{\bigwedge^{g-1-i} H^0(C,\omega_C) \lra \bigwedge^{g-2-i} H^0(C,\omega_C) \oo H^0(C,\omega_C) \Bigr\},\]
which vanishes since $g-1-i>0$ and the Koszul differential is an injection.
\endproof

We now let $f:\mathcal{C}\rightarrow \cM_g^{\sharp}$ denote the universal pseudo-stable curve, we let $\omega_f$ denote the relative dualizing sheaf, and define the vector bundle $\mc{M}$ on $\cM_g^{\sharp}$ by
\begin{equation}\label{eq:def-cM}
\cM:=\ker\bigl(f^*f_*(\omega_f)\twoheadrightarrow \omega_f),
\end{equation}
where surjectivity follows because $\omega_C$ is globally generated for $[C]\in\cM_g^{\sharp}$. Let $\cB:=f_*\Bigl(\bigwedge^i\cM\otimes \omega_f^{2}\Bigr).$ Using Grauert's theorem \cite[Corollary~III.12.9]{hartshorne} and Proposition \ref{vanishing1}, we conclude that $\cB$ is locally free, and by construction, the fiber $\mc{B}([C])$ is given as in (\ref{eq:fiber-A-B}).

\vskip 4pt

To construct $\cA$, we proceed in a way similar to \cite{F}. We begin by observing as in the proof of Proposition~\ref{vanishing1} that $\bw^i M_C(2)$ is resolved by a Koszul type complex
\[ \bw^i H^0(\bb{P}_C,\mc{O}_{\bb{P}_C}(1)) \oo \mc{O}_{\bb{P}_C}(2) \lra \bw^{i-1} H^0(\bb{P}_C,\mc{O}_{\bb{P}_C}(1)) \oo \mc{O}_{\bb{P}_C}(3) \lra \cdots\]
By taking global sections and using the fact that $H^0(\bb{P}_C,\mc{O}_{\bb{P}_C}(1)) = H^0(C,\omega_C)$, it follows that
$$H^0\Bigl(\bb{P}_C, \bigwedge^i M_C (2)\Bigr) = \ker\Bigl\{\bw^i H^0(C,\omega_C) \oo \Sym^2 H^0(C,\omega_C) \lra \bw^{i-1} H^0(C,\omega_C) \oo \Sym^3 H^0(C,\omega_C) \Bigr\}$$
We consider the locally free sheaf $\H = f_*(\omega_f)$, whose fiber over $[C]$ is $H^0(C,\omega_C)$, and define $\mc{A}$ to be the sheaf resolved by the Koszul-type complex
\[ \bw^i\H \oo \Sym^2\H \lra \bw^{i-1}\H \oo \Sym^3 \H \lra \cdots \lra \Sym^{i+2}\H \lra 0.\]
Since this is a finite locally free complex which is exact except at the beginning, it follows that $\mc{A}$ is locally free. Moreover, restricting to the fiber over $[C]$ we get that $\mc{A}([C])$ is as in (\ref{eq:fiber-A-B}).

To construct the map $\a:\mc{A} \rightarrow \mc{B}$, it is enough by adjunction to define a map $f^*\mc{A} \rightarrow \bigwedge^i\cM\otimes \omega_f^{2}$. Since $f^*$ commutes with tensor constructions, it follows that $f^*\mc{A}$ is resolved by the Koszul-type complex
\begin{equation}\label{eq:kosz-1}
\bw^i (f^*\H) \oo \Sym^2(f^*\H) \lra \bw^{i-1}(f^*\H) \oo \Sym^3 (f^*\H) \lra \cdots \lra \Sym^{i+2}(f^*\H) \lra 0.
\end{equation}
Moreover, it follows from (\ref{eq:def-cM}) that $\bigwedge^i\cM\otimes \omega_f^{2}$ is resolved by the Koszul-type complex
\begin{equation}\label{eq:kosz-2}
\bw^i (f^*\H) \oo \omega_f^2 \lra \bw^{i-1}(f^*\H) \oo \omega_f^3 \lra \cdots \lra \omega_f^{i+2} \lra 0.
\end{equation}
The  morphism $f^*\H \rightarrow \omega_f$ induces a map between the complexes (\ref{eq:kosz-1}) and (\ref{eq:kosz-2}), which in turn induces a map between the $0$-th homology groups, which is the desired map $f^*\mc{A} \rightarrow \bigwedge^i\cM\otimes \omega_f^{2}$. Defining $\a$ by adjunction and restricting to the fiber over $[C]$ we obtain the natural map $\a_C$ in (\ref{eq:ker-alpha-C}), as desired.

\vskip 3pt

We are now in a position to conclude that Theorem \ref{thm:vanishing-tangential} implies Green's Conjecture for generic curves of genus $g$ in characteristic $0$, or characteristic $p\geq \frac{g+2}{2}$.


\vskip 5pt


\begin{proof}[Proof of Theorem \ref{thm:genericGreen}]
Since $\mbox{char}(\kk)=0$ or $\mbox{char}(\kk)\geq \frac{g+2}{2}$, it follows from Theorem~\ref{thm:vanishing-tangential} that if $C$ is a generic linear section of the tangential variety $\mc{T}$ then $K_{\lfloor \frac{g}{2} \rfloor,1}(C,\OO_{C}(1))=0$. Such a curve $C$ is an irreducible, $g$-cuspidal rational curve, where the cusps correspond to the intersection of the generic hyperplane with the rational normal curve. Necessarily, $C$ has maximal gonality $\lfloor \frac{g+3}{2}\rfloor $, therefore  $C$ defines a point $[C] \in \cM_g^{\sharp}$. By adjunction, $\omega_C\cong \omega_{\mc{T}}(1)_{|_{C}}\cong \OO_C(1)$, by Corollary~\ref{cor:omega-T}. We conclude that $C$ is a canonically embedded pseudo-stable curve, and that
$K_{\lfloor \frac{g}{2} \rfloor,1}(C,\omega_{C}) = 0$.
If we take $i=\lfloor \frac{g}{2}\rfloor - 1$ and construct the vector bundles $\mc{A},\mc{B}$ and the map $\a:\mc{A}\rightarrow \mc{B}$ as above, it follows that the map $\a_{C}$ is injective, so the same is true for the map $\a_{C'}$ for a general point $[C'] \in \cM_g^{\sharp}$. We conclude that the vanishing $K_{\lfloor \frac{g}{2}\rfloor,1}(C',\omega_{C'})=0$ holds for a general curve $[C']\in \cM_g$, as desired.
\end{proof}

\begin{rmk}
The proof above shows that every $g$-cuspidal rational curve satisfies Green's conjecture in suitable characteristics, regardless of the position of the cusps. Indeed, taking any $g$ distinct points on the rational normal curve $\Gamma\subseteq \PP^g$ defines a rational $g$-cuspidal hyperplane section of $\mc{T}$ with the same syzygies as $\mc{T}$, since $H^1(\mc{T},\OO_{\mc{T}}(d))=0$ for all $d\geq 0$, see \cite[Theorem~3.b.7]{G84}.
\end{rmk}

\begin{rmk}
Using Corollary \ref{cor:omega-T} , we observe  that the tangent developable $\mathcal{T}\subseteq \mathbb P^g$ can be viewed as a degenerate $K3$ surface in the following sense. We denote by $\F_g$ the moduli
space of quasi-polarized $K3$ surfaces $[X,H]$, where $H\in \mbox{Pic}(X)$ is the polarization class with $H^2=2g-2$. Inside $\F_g$, we consider the \emph{unigonal} divisor
$D_g^{\mathrm{ug}}$ consisting of $K3$ surfaces $[X,H]$, such that $H=gE+\Gamma$, where $\Gamma^2=-2, E^2=0$ and $\Gamma\cdot E=0$.  The linear system $|H|$ has the rational curve $\Gamma$ as fixed component and each element of $|H|$ is a flag curve of the form $\Gamma+E_1+\cdots+E_g$, where $E_i\in |E|$ are elliptic curves. Let $\mathcal{H}_g$ be the Hilbert scheme of quasi-polarized $K3$ surfaces $X\subseteq \mathbb P^g$ with $\mbox{deg}(X)=2g-2$. Then $[\mathcal{T}]\in \mathcal{H}_g$. We let $\mathfrak{F}_g:=\mathfrak{H}_g\dblq SL(g+1)$ be the the corresponding GIT quotient.
There is a birational map $\phi:\F_g\dashrightarrow \mathfrak{F}_g$, which blows down the unigonal divisor $D_g^{\mathrm{un}}$ to a \emph{single point} corresponding to the tangential variety $\mathcal{T}\subseteq \mathbb P^g$.
\end{rmk}

\end{document}